\documentclass[onefignum,onetabnum]{siamart190516}

\usepackage{amsmath,amssymb,amsfonts,latexsym,stmaryrd}
\usepackage{graphicx,float,epsfig} 
\usepackage{mathrsfs} 
\usepackage{color}
\usepackage{setspace} 
\usepackage{lipsum}
\usepackage{graphicx,float}
\usepackage{color}
\usepackage{epstopdf}
\usepackage{multirow}
\usepackage{svg}
\usepackage{url}
\usepackage{diagbox}
\ifpdf
  \DeclareGraphicsExtensions{.eps,.pdf,.png,.jpg}
\else
  \DeclareGraphicsExtensions{.eps}
\fi
\newcommand\err{\texttt{err}}
\newcommand\eff{\texttt{eff}}

\def\O{\Omega}

\renewcommand\sp{\mathop{\mathrm{Sp}}\nolimits}

\newcommand{\jump}[1]{\llbracket #1 \rrbracket}

\newcommand{\n}{\boldsymbol{n}}
\usepackage{booktabs}
\usepackage{float}

\newcommand\bu{\boldsymbol{u}}
\newcommand\bv{\boldsymbol{v}}
\newcommand\bw{\boldsymbol{w}}

\newcommand\curl{\mathop{\mathbf{curl}}\nolimits}

\def\hdel{\widehat{\delta}}
\def\CM{\mathcal{X}}
\def\CN{\mathcal{Y}}

\newcommand\bF{\boldsymbol{f}}
\newcommand\bI{\boldsymbol{I}}

\newcommand\bT{\boldsymbol{T}}

\newcommand\bQ{\boldsymbol{Q}}

\newcommand\0{\mathbf{0}}

\def\CT{{\mathcal T}}

\renewcommand\H{\mathrm{H}}

\renewcommand\L{\mathrm{L}}

\renewcommand\O{\Omega}

\newcommand\bdiv{\mathop{\mathbf{div}}\nolimits}

\renewcommand\div{\mathop{\mathrm{div}}\nolimits}

\renewcommand\sp{\mathop{\mathrm{sp}}\nolimits}
\newcommand\dist{\mathop{\mathrm{dist}}\nolimits}

\newcommand{\vertiii}[1]{{\left\vert\kern-0.25ex\left\vert\kern-0.25ex\left\vert #1 
    \right\vert\kern-0.25ex\right\vert\kern-0.25ex\right\vert}}

\newsiamremark{remark}{Remark}
\newsiamremark{hypothesis}{Hypothesis}
\crefname{hypothesis}{Hypothesis}{Hypotheses}
\newsiamthm{problem}{Problem}
\newsiamthm{claim}{Claim}

\headers{The Navier-Lam\'e eigenvalue problem}{Felipe Lepe,  Gonzalo Rivera and Jesus Vellojin}

\title{Finite element analysis for the Navier-Lam\'e eigenvalue problem\thanks{Submitted to the editors DATE.
\funding{The first author was partially supported by
	DIUBB through project 2120173 GI/C Universidad del B\'io-B\'io and ANID-Chile through FONDECYT project 11200529 (Chile).
	The second author was supported by Universidad de Los Lagos Regular R02/21.  The third author was partially supported by ANID-Chile through project Anillo of Computational Mathematics for Desalination Processes ACT210087.
}}}

\author{Felipe Lepe\thanks{GIMNAP-Departamento de Matem\'atica, Universidad del B\'io - B\'io, Casilla 5-C, Concepci\'on, Chile. \email{flepe@ubiobio.cl}.}
\and Gonzalo Rivera\thanks{Departamento de Ciencias Exactas,
	Universidad de Los Lagos, Casilla 933, Osorno, Chile. \email{gonzalo.rivera@ulagos.cl}.}
\and Jesus Vellojin\thanks{GIMNAP-Departamento de Matem\'atica, Universidad del B\'io - B\'io, Casilla 5-C, Concepci\'on, Chile. \email{jvellojin@ubiobio.cl}.}}

\usepackage{amsopn}

\ifpdf
\hypersetup{
  pdftitle={The Navier-Lam\'e eigenvalue problem},
  pdfauthor={Felipe Lepe, Gonzalo Rivera and Jesus Vellojin}
}
\fi

\begin{document}

\maketitle

\begin{abstract}
The present paper introduces the analysis of the eigenvalue problem for the elasticity equations when the so called Navier-Lam\'e
system is considered. Such a system introduces the displacement, rotation and pressure of some linear and elastic structure. The analysis of the spectral problem is based in the compact operators theory. A finite element method based in polynomials
of degree $k\geq 1$ are considered  in order to approximate the eigenfrequencies and eigenfunctions of the system. Convergence and error estimate are presented. An a posteriori error analysis is performed, where the reliability and efficiency of the proposed estimator is proved. We end this contribution reporting a series of numerical tests in order to assess the performance of the proposed numerical method, for the a priori and a posteriori estimates.

\end{abstract}



\section{Introduction}\label{sec:intro}

Different approaches to analyze the elasticity equations have been analyzed in the past years, since the stability of different
structures   plays an important role for the design and construction of more complex structures. The displacement of elastic structures, when external forces are applied, reveals the importance of how structures manifest their behaviors depending on their physical
features, which are given by the Lam\'e constants. 

To handle the linear elasticity equations, several formulations and numerical methods have emerged in the past years.  In this sense, mixed formulations are a path to avoid the locking phenomenon in the computational implementation that arises when the Lam\'e constant $\lambda\rightarrow\infty$, leading to the so called nearly incompressible elasticity equations.

In the literature, an important number of contributions related to mixed formulations for elasticity are available such as \cite{MR3872756, MR4461583, MR2293249, MR3452773, MR3453481, MR3361707}, just to mention some of them. These references have the particularity that the results are focused in the load problems. These contributions are  relevant  for our
work in spectral problems, since in this context, the analysis of  load problems provide a source of techniques that can be extended for the eigenvalue approach.

The research related to mixed formulations for the elasticity spectral problem is an ongoing subject and different formulations and methods have emerged, as \cite{MR4279087, inzunza2021displacementpseudostress, MR3962898,  LRV_JSCV, MR3036997}, where these formulations and numerical methods  are not only  interested in the approximation of the displacement of the structure, but also on other quantities as the pseudostress, the Cauchy stress tensor, the rotations, etc. Moreover, since these formulations are mixed, the locking phenomenon is avoided and the spectrum of the solution operators is approximated safely with any numerical method.

In the present work, we continue with the analysis of mixed formulations for the elasticity equations and the associated eigenvalue problem, where the formulation analyzed in  \cite{MR3872756, MR4461583} is considered. These works concern the Navier-Lam\'e formulation of the elasticity equations, where not only the displacement of the structure is the unknown, but also the vorticity and pressure of the structure. The proposed mixed  formulation, compared with the one analyzed in \cite{inzunza2021displacementpseudostress}, avoids the $\H(\div)$ spaces and allows to handle more simple Hilbert spaces as $\L^2$ and $\H^1$. The advantage behind this approach is the relaxation of the finite element spaces where the solutions are approximated by means of  continuous polynomials for the displacement and discontinuous elements for the pressure and vorticity. The above establishes a difference with respect to the use of Raviart-Thomas or Brezzi-Douglas-Marini elements, classic in the $\H(\div)$ approach.

With this Navier-Lam\'e eigensystem, our goal is to analyze a mixed finite element method for which we obtain a priori and a posteriori error estimates. These two analysis are important since the a priori analysis not only gives us the basic convergence results and error estimates, but also the spurious free feature of the method, whereas the a posteriori analysis takes relevance when the eigenfunctions of the spectral problem are not smooth enough when the geometries of certain domains or physical parameters are not sufficiently suitable to recover the optimal order of convergence given by the a priori analysis. The a priori analysis is based on the classic theory of compact operators given by \cite{MR1115235}, where the convergence of eigenvalues and eigenfunctions is obtained with standard arguments. Also, we derive an error estimate for the displacement of the structure in the  $\L^2$ norm, with a duality argument. Controlling this error is important since it arises naturally when performing the analysis of the reliability and efficiency of the a posteriori estimator.

The paper is organized in the following way: In Section \ref{sec:model_problem} we present the model problem, where we introduce the classic elasticity eigenvalue problem and, with some algebraic manipulations, we write the Navier-Lam\'e system of our interest. We define the bilinear forms that we need for the analysis and recall a  well posedness result, fundamental to introduce the solution operator that will be involved along the manuscript. A spectral characterization of this operator is presented. In Section \ref{sec:fem} we introduce the tools for the numerical analysis, namely, definitions and properties associated to  the meshes, the finite element spaces, the discrete eigenvalue problem, and the discrete solution operator. This allows us to conclude a  convergence result. This section contains the core of our paper, since the a priori and a posteriori analysis are reported, where we prove basic error estimates for the eigenvalues and eigenfunctions, together with the analysis of an a posteriori error estimator, whose reliability and efficiency is proven. Finally, in Section \ref{sec:numerics}, we present a series of numerical tests in order to confirm the theoretical results that we present. These results, that we perform in two and three dimensions, show that the a priori error estimates are attained at computational level, and that the a posteriori error estimator that we propose is capable to perform an adaptive refinement in order to recover the optimal order of convergence, for nonregular eigenfunctions.

\section{The model problem}
\label{sec:model_problem}

Let  $\O\subset\mathbb{R}^n$,  with $n\in\{2,3\}$,  be an open bounded domain with Lipschitz boundary $\partial\O$. 
To derive the problem of our interest, we need to recall the elasticity eigenvalue problem given as follows
\begin{equation}\label{def:navier_lame}
\left\{
\begin{array}{rcll}
\bdiv(2\mu_s\boldsymbol{\varepsilon}(\bu)+\lambda_s\div\bu\mathbf{I})&=&-\kappa\bu&\text{in}\,\O,\\
\bu&=&\boldsymbol{0}&\text{on}\,\partial\O,
\end{array}
\right.
\end{equation}
where $\bu$ is the displacement, $\mathbf{I}\in\mathbb{R}^{n\times n}$ is the identity tensor, $\lambda_s$ and $\mu_s$ are de Lam\'e constants and $\boldsymbol{\varepsilon}(\cdot)$ is the strain tensor given by $\boldsymbol{\varepsilon}(\bu):=\frac{1}{2}(\nabla\bu+(\nabla\bu)^{\texttt{t}})$. For the derivation of the Navier-Lam\'e system, as a first ingredient, we need the following identity
$$
\bdiv(\boldsymbol{\varepsilon}(\bu))=\frac{1}{2}(\Delta\bu+\nabla(\div\bu)).
$$
With this identity at hand, and replacing it in \eqref{def:navier_lame}, we arrive to the following system
$$
\left\{
\begin{array}{rcll}
\mu_s\Delta\bu+(\mu_s+\lambda_s)\nabla(\div\bu)&=&-\kappa\bu&\text{in}\,\O,\\
\bu&=&\boldsymbol{0}&\text{on}\,\partial\O.
\end{array}
\right.
$$

Now, defining the pressure $p$ by $p:=-(2\mu_s+\lambda_s)\div\bu$, we obtain the following displacement-pressure system
$$
\left\{
\begin{array}{rcll}
\mu_s\Delta\bu-\nabla p&=&-\kappa\bu&\text{in}\,\O,\\
\div\bu+(2\mu_s+\lambda_s)^{-1}p&=&0&\text{in}\,\O,\\
\bu&=&\boldsymbol{0}&\text{on}\,\partial\O.
\end{array}
\right.
$$
Introducing the field of scaled rotations by $\boldsymbol{\omega}:=\sqrt{\mu_s}\curl\bu$, we obtain the following rotation-pressure-displacement system
\begin{equation}\label{def:lame_system}
	\left\{
	\begin{array}{rcll}
\sqrt{\mu_s}\curl\boldsymbol{\omega}+\nabla p&=&\kappa\bu&\quad\text{in}\,\O,\\
\boldsymbol{\omega}-\sqrt{\mu_s}\curl\bu&=&0&\quad\text{in}\,\O,\\
\div\bu+(2\mu_s+\lambda_s)^{-1}p&=&0&\quad\text{in}\,\O,\\
\bu&=&\boldsymbol{0}&\quad\text{on}\,\partial\O.
	\end{array}
	\right.
\end{equation}

{Now we introduce a variational formulation for \eqref{def:lame_system}. To do this task, we multiply such a system with suitable 
tests functions, integrate by parts, and use the boundary conditions in order to obtain the following bilinear forms $a: \L^2(\O)^{n(n-1)/2}\times \L^2(\O)\rightarrow\mathbb{R}$
and $b:(\L^2(\O)^{n(n-1)/2}\times \L^2(\O))\times\H^1(\O)^n\rightarrow\mathbb{R}$ defined by 
$$
a((\boldsymbol{\omega},p),(\boldsymbol{\theta},q)):=\int_{\Omega}\boldsymbol{\omega}\cdot\boldsymbol{\theta}+(2\mu_s+\lambda_s)^{-1}\int_{\O}pq,
$$
and
$$
b((\boldsymbol{\theta},q),\bu):=\int_{\O}\div\bu q-\sqrt{\mu_s}\int_{\O}\boldsymbol{\theta}\cdot\curl\bu.
$$

With these bilinear forms at hand, we write the following weak formulation for problem \eqref{def:lame_system}: Find $\lambda\in\mathbb{R}$ and $((\boldsymbol{\omega},p),\bu)\in \L^2(\O)^{n(n-1)/2}\times \L^2(\O)\times \H^1(\O)^n$ such that 
\begin{equation}\label{def:lame_system_weak}
	\left\{
	\begin{array}{rcll}
a((\boldsymbol{\omega},p),(\boldsymbol{\theta},q))+b((\boldsymbol{\theta},q),\bu)&=&0&\forall (\boldsymbol{\theta},q)\in \L^2(\O)^{n(n-1)/2}\times \L^2(\O),\\
b((\boldsymbol{\omega},p),\bv)&=&\kappa (\bu,\bv)&\forall\bv\in\H^1(\O)^n.
\end{array}
	\right.
\end{equation}

Now we define the following bilinear form
$$
\begin{aligned}
	\mathcal{A}((\bu,\boldsymbol{\omega},p),(\bv,\boldsymbol{\theta},q))&:=a((\boldsymbol{\omega},p),(\boldsymbol{\theta},q))+b((\boldsymbol{\theta},q),\bu)+b((\boldsymbol{\omega},p),\bv).
	 \end{aligned}
$$
With this definition at hand, the weak eigenvalue problem associated to \eqref{def:lame_system_weak} is stated below.
\begin{problem}
\label{problem1}
Find $\lambda\in\mathbb{R}$ and $\boldsymbol{0}\neq (\bu,\boldsymbol{\omega},p)\in \H_0^1(\Omega)^n\times \L^2(\Omega)^{n(n-1)/2}\times \L^2(\Omega)$, such that 
\begin{equation} 
	\label{eq:formA}
	\mathcal{A}((\bu,\boldsymbol{\omega},p),(\bv,\boldsymbol{\theta},q))=-\kappa(\bu,\bv)_{0,\Omega},
\end{equation}
for all $(\bv,\boldsymbol{\theta},q)\in \H_0^1(\Omega)^n\times \L^2(\Omega)^{n(n-1)/2}\times \L^2(\Omega)$. 
\end{problem}

Since we have a prescribed Dirichlet condition on $\partial\Omega$, then the pressure needs to satisfy the zero-mean condition in order to control the norm in $\L^2(\O)$ as we approach to the incompressible limit $\nu=0.5$. Following \cite{MR4461583}, the pressure is decomposed by  $p=p_0+P_mp$, where $P_mp$ and $p_0$ are the mean value and zero-mean value parts of the pressure, and obtain the following result \cite[Theorem 2.1]{MR4461583}.

\begin{lemma}
\label{lmm:elliptic}
For all $\boldsymbol{0}\neq (\bu,\boldsymbol{\omega},p)\in \H_0^1(\Omega)^n\times \L^2(\Omega)^{n(n-1)/2}\times \L^2(\Omega)$, there exists $(\bv,\boldsymbol{\theta},q)\in \H_0^1(\Omega)^n\times \L^2(\Omega)^{n(n-1)/2}\times \L^2(\Omega)$ with $\vertiii{(\bv,\boldsymbol{\theta},q)}\leq C\vertiii{(\bu,\boldsymbol{\omega},p)}$ such that
$$
\mathcal{A}((\bu,\boldsymbol{\omega},p),(\bv,\boldsymbol{\theta},q))\geq C \vertiii{(\bu,\boldsymbol{\omega},p)}^2,
$$
where the triple norm is defined by
\begin{multline*}
\vertiii{(\bv,\boldsymbol{\theta},q)}^2:=\mu_s\Vert\curl\bv\Vert_{0,\Omega}^2+\mu_s\Vert\div\bv\Vert_{0,\Omega}^2+\Vert\boldsymbol{\theta}\Vert_{0,\Omega}^2\\
+(2\mu_s+\lambda_s)^{-1}\Vert q\Vert_{0,\Omega}^2 + \mu_s^{-1}\Vert q_0\Vert_{0,\Omega}^2.
\end{multline*}
\end{lemma}

Let us define the solution operator associated to the eigenvalue problem. Let $\bT$ be such an operator, which is linear and defined by 
$$
 \bT:\H_0^1(\Omega)^n\rightarrow\H_0^1(\Omega)^n,\,\,\,\bF\mapsto\bT\bF:=\widetilde{\bu},
$$
where the triplet $(\widetilde{\bu},\widetilde{\boldsymbol{\omega}},\widetilde{p})$ is the solution of the following well posed source problem
\begin{equation} 
\label{eq:source}
	\mathcal{A}((\widetilde{\bu},\widetilde{\boldsymbol{\omega}},\widetilde{p}),(\bv,\boldsymbol{\theta},q))=-(\bF,\bv)_{0,\O},\\
\end{equation}
for all $(\bv,\boldsymbol{\theta},q)\in \H_0^1(\Omega)^n\times \L^2(\Omega)^{n(n-1)/2}\times \L^2(\Omega)$. 
As a consequence of Lemma \ref{lmm:elliptic}, an application of the Bab\v uska-Brezzi theory reveals that the  operator $\bT$ is well defined and bounded.
Observe that $\bT$   is   selfadjoint with respect to the $\L^2(\O)$ inner product. Indeed,  let 
$\bF$ and $\widetilde{\bF}$ be such that $\bu=\bT\bF$ and $\widetilde{\bu}=\bT\widetilde{\bF}$. Then, from the symmetry of $a(\cdot,\cdot)$  we have 
\begin{multline*}
(\bF,\bT\widetilde{\bF})_{0,\O}=(\bF,\widetilde{\bu})_{0,\O}= b((\boldsymbol{\omega},p),\widetilde{\bu})=-a((\widetilde{\boldsymbol{\omega}},\widetilde{p}),(\boldsymbol{\omega},p))\\
=-a((\boldsymbol{\omega},p),(\widetilde{\boldsymbol{\omega}},\widetilde{p}))=b((\widetilde{\boldsymbol{\omega}},\widetilde{p}),\bu)=(\widetilde{\bF},\bu)_{0,\O}=(\widetilde{\bF},\bT\bF)_{0,\O}=(\bT\bF,\widetilde{\bF})_{0,\O}.
\end{multline*}

Let $\eta$ be a real number such that  $\eta\neq 0$. Notice that $(\eta,\boldsymbol{u})\in \mathbb{R}\times\H_0^1(\Omega)^n$ is an eigenpair of $\bT$ if and only if  there exists $(\boldsymbol{\omega},p)\in \L^2(\O)^{n(n-1)/2}\times\L^2(\O)$ such that,  $(\kappa,\boldsymbol{u},\boldsymbol{\omega},p)$ solves Problem \ref{problem1}  with $\eta:=1/\kappa$.

A key ingredient for the numerical analysis of Problem \ref{problem1} is  the following regularity result  for the source problem, which is a direct consequence of Lemma \ref{lmm:elliptic} (see \cite[Theorem 2.1]{MR3872756}) and the well-known regularity for the elasticity equations (see, for instance, \cite{grisvard1986problemes} and \cite[Theorem 5.2]{rossle2000corner}).
\begin{theorem}[Regularity]
\label{th:regularity}
System \eqref{eq:source} has a unique solution and there exists a positive constant $C$ such that
$$
\vertiii{(\widetilde{\bu},\widetilde{\boldsymbol{\omega}},\widetilde{p})}\leq C \|\bF\|_{0,\O}.
$$
Morever, there exists $\widehat{s}\in(0,1]$ such that for all $s\in(0,\widehat{s})$, we have $\widetilde{\bu}\in \H^{1+s}(\O)^n$ and 
$$
\|\widetilde{\bu}\|_{1+s,\Omega}+\|\widetilde{\boldsymbol{\omega}}\|_{s,\O}+\|\widetilde{p}\|_{s,\O}\leq \widehat{C} \|\bF\|_{0,\O},
$$
where $\widehat{C}>0$ which in principle depends on $\lambda_s$.
On the other hand, the regularity for the eigenfunctions is as follows
$$
\|\bu\|_{1+s,\Omega}+\|\boldsymbol{\omega}\|_{s,\O}+\|p\|_{s,\O}\leq \widehat{C} \|\bu\|_{0,\O}.
$$
\end{theorem}
Hence, $\bT$ is a compact operator thanks to the compact inclusion of $\H^{1+s}(\O)^n$ onto $\H^1(\O)^n$.
Now we are in position to present the following spectral characterization for $\bT$.
\begin{theorem}[Spectral characterization]
The spectrum of the operator $\bT$ decomposes by  $\sp(\bT):=\{\mu_k\}_{k\in\mathbb{N}}\cup\{0\}$, where $\{\mu_k\}_{k\in\mathbb{N}}$
is a sequence of positive eigenvalues converging to zero. These eigenvalues are repeated according to their respective multiplicities.
\end{theorem}


\section{Finite element discretization}
\label{sec:fem}
The aim of this section is to develop a numerical method to approximate the eigenfrequencies of Problem \ref{problem1}. We state several results in order to provide estimates that allow to control the corresponding eigenfunctions.

\subsection{Mesh properties}
Let $\{\mathcal{T}_h\}_{h>0}$ be a shape-regular family of partitions of $\O$ that, depending on the domain, will consist in triangles (for two dimensions) or tetrahedrons (in three dimensions).  Let $h_T$ be the diameter of a triangle $T\in\CT_h$ and let us define $h:=\max\{h_T\,:\, T\in \CT_h\}$.

For $T\in\mathcal{T}_h$, let $\mathcal{E}(T)$ be the set of its edges/faces, and let $\mathcal{E}_h$ be the set of all
the faces/edges of the triangulation $\mathcal{T}_h$. For this set, we denote by $h_e$ the corresponding diameter of a face/edge $e$. 
With these definitions at hand, we write $\mathcal{E}_h=\mathcal{E}_h(\O)\cup\mathcal{E}_h(\partial\O)$, where
$$
	\mathcal{E}_h(\O):=\{e\in\mathcal{E}_h\,:\,e\subseteq\O\}\quad\text{and}\quad\mathcal{E}_h(\partial\O):=\{e\in\mathcal{E}_h\,:\, e\subseteq\partial\O\}.
$$
Additionally, we will denote by $\jump{\cdot}$ the corresponding edge/face jump for scalars and vector functions.

Let us introduce the finite element spaces which we will operate. Given $k\geq 1$, the finite element spaces to develop the numerical scheme are the following 
\begin{align*}
	\mathbf{H}_h&:=\{\bv_h\in\H_0^1(\Omega)\,:\, \bv_h|_T\in\mathbb{P}_k(T)^n\quad\forall T\in\mathcal{T}_h\},\\
	\mathbf{Z}_h&:=\{\boldsymbol{\theta}_h\in 	\L^2(\Omega)^{n(n-1)/2}\,:\, \boldsymbol{\theta}_h|_T\in\mathbb{P}_{k-1}(T)^{n(n-1)/2}\quad\forall T\in\mathcal{T}_h\},\\
	Q_h&:=\{q_h\in \L^2(\Omega)\,:\, q_h|_T\in\mathbb{P}_{k-1}(T)\quad\forall T\in\mathcal{T}_h\},
\end{align*} 
in which the rotation, displacement, and pressure are approximated, respectively.
\subsection{Discrete problem}

We start this section by stating the corresponding finite element discretization of Problem \ref{problem1}, which reads as follows.
\begin{problem}
\label{problem1_discrete}
Find $\kappa_h\in\mathbb{R}$ and $\boldsymbol{0}\neq (\bu_h,\boldsymbol{\omega}_h,p_h)\in \mathbf{H}_h\times\mathbf{Z}_h\times Q_h$, such that  
$$
\mathcal{A}((\bu_h,\boldsymbol{\omega}_h,p_h),(\bv_h,\boldsymbol{\theta}_h,q_h))=-\kappa_h(\bu_h,\bv_h)_{0,\Omega},
$$
for all $(\bv_h,\boldsymbol{\theta}_h,q_h)\in \mathbf{H}_h\times\mathbf{Z}_h\times Q_h$.
\end{problem}

At this point we must emphasize what was commented in \cite[Section 2.2]{MR4461583}, where we have that the fact that the non-uniqueness of the pressure in the incompressible limit requires that Lagrange multipliers be introduced to fix the zero-mean condition, or, as was done in \cite{hughes1987new}, to add a stabilization term in the case that the pressure is approximated using discontinuous elements (see Section \ref{sec:numerics} below). 

By means of the arguments used in Theorem, we conclude that Problem \ref{problem1_discrete} is well-posed, and therefore we have the respective discrete estimates for the discrete source problem.

Let $\bT_h$ be the discrete  linear operator defined by 
$$
 \bT_h:\H^{1+s}(\Omega)^n\rightarrow\mathbf{H}_h,\,\,\,\bF\mapsto\bT_h\bF:=\widetilde{\bu}_h,
$$
where the triplet $(\widetilde{\bu}_h,\widetilde{\boldsymbol{\omega}}_h,\widetilde{p}_h)$  is the solution of the following discrete source problem
\begin{equation}
\label{eq:source_h} 
	\mathcal{A}((\widetilde{\bu}_h,\widetilde{\boldsymbol{\omega}}_h,\widetilde{p}_h),(\bv_h,\boldsymbol{\theta}_h,q_h))=-(\bF,\bv_h)_{0,\Omega},
\end{equation}
for all $(\bv_h,\boldsymbol{\theta}_h,q_h)\in \mathbf{H}_h\times\mathbf{Z}_h\times Q_h$. Observe that $\bT_h$ is selfadjoint and  as in the continuous case,  $(\eta_h,\bu_h)\in \mathbb{R}\times\mathbf{H}_h$, with $\eta_h\neq 0$, is an eigenpair of $\bT_h$ if only if  there exists $(\boldsymbol{\omega}_h,p_h)\in\mathbf{Z}_h\times Q_h$ such that,  $(\kappa_h,\boldsymbol{u}_h,\boldsymbol{\omega}_h,p_h)$ solves Problem \ref{problem1_discrete} for $\eta_h:=1/\kappa_h$. 

The following result is a direct consequence of the well-posedness of Problems \ref{problem1}--\ref{problem1_discrete} and their corresponding source problems.
\begin{corollary}[Approximation between $\bT$ and $\bT_h$]
	\label{cor:app_TT1}
Let $\boldsymbol{f}\in \L^2(\O)^n$. Under the assumptions of Theorem \ref{th:regularity}, there exists $C_{\mu_s}>0$, independent of $h$, such that
$$
		\|(\bT-\bT_h)\boldsymbol{f}\|_{1,\O}\leq C_{\mu_s}h^{s}\|\boldsymbol{f}\|_{0,\O}\leq C_{\mu_s}h^{s}\|\boldsymbol{f}\|_{1,\O}.
$$
\end{corollary}
\begin{proof}
Let $\widetilde{\bu}:=\bT\bF$ and $\widetilde{\bu}_h:=\bT_h\bF$. Hence, we have directly
$$
\|(\bT-\bT_h)\bF\|_{1,\O}=\|\widetilde{\bu}-\widetilde{\bu}_h\|_{1,\O}\leq \dfrac{1}{\sqrt{\mu_s}}\vertiii{(\widetilde{\bu}-\widetilde{\bu}_h,\widetilde{\boldsymbol{\omega}}-\widetilde{\boldsymbol{\omega}}_h,\widetilde{p}-\widetilde{p}_h)}.
$$
From \cite[Theorem 3.1]{MR3872756}, the following best approximation property holds
$$
\vertiii{(\widetilde{\boldsymbol{\omega}}-\widetilde{\boldsymbol{\omega}}_h,\widetilde{p}-\widetilde{p}_h,\widetilde{\bu}-\widetilde{\bu}_h)}\leq C\inf_{(\bv_h,\boldsymbol{\theta}_h,q_h)\in \mathbf{H}_h\times\mathbf{Z}_h\times Q_h}\vertiii{(\widetilde{\boldsymbol{\omega}}-\boldsymbol{\theta}_h,\widetilde{p}-q_h,\widetilde{\bu}-\bv_h)},
$$
where the constant $C$ is positive and independent of $h$. The result is concluded from  \cite[Theorem 3.2]{MR3872756} and the additional regularity provided by Theorem \ref{th:regularity}.
%
%
%
\end{proof}

\begin{remark}
We note from the above corollary that the following result is also true
\begin{equation}\label{eq:_aux_estimate}
\vertiii{(\widetilde{\bu}-\widetilde{\bu}_h,\widetilde{\boldsymbol{\omega}}-\widetilde{\boldsymbol{\omega}}_h,\widetilde{p}-\widetilde{p}_h)}\leq C_{\mu_s}h^{s}\|\bF\|_{1,\O}.
\end{equation}
\end{remark}

As a consequence of the previous results, it is immediate that the proposed  numerical method is spurious free, as is stated in the following result (see \cite{MR0203473} for instance).
\begin{theorem}
	\label{thm:spurious_free}
	Let $V\subset\mathbb{C}$ be an open set containing $\sp(\bT)$. Then, there exists $h_0>0$ such that $\sp(\bT_h)\subset V$ for all $h<h_0$.
\end{theorem}
Let us  recall the definition of the resolvent operator of $\bT$ and $\bT_h$ respectively:
\begin{gather*}
	R_z(\bT):=(z\bI-\bT)^{-1}\,:\, \H_0^1(\Omega)^n \to \H_0^1(\Omega)^n\,, \quad z\in\mathbb{C}\setminus \sp(\bT), \\
	R_z(\bT_{h}):=(z\bI-\bT_h)^{-1}\,:\, \mathbf{H}_h \to \mathbf{H}_h\,, \quad z\in\mathbb{C}\setminus\sp(\bT_h) .
\end{gather*}
We also  invoke the following result for the resolvent of $\bT$.
\begin{proposition}\label{prop:bounded_resolv}
If $z\notin\sp(\bT)$, then there exists a positive constant $C$, independent of $\lambda$ and $z$ such that
$$
\|(z\bI-\bT)\bu\|_{1,\O}\geq C\dist(z,\sp(\bT))\|\bu \|_{1,\O},
$$
where $\dist(z,\sp(\bT ))$ represents the distance between $z$ and the spectrum of $\bT$ in the complex plane, which in principle depends on $\lambda_s$. 
\end{proposition}
\begin{proof}
See \cite[Proposition 2.4]{MR3376135}.
\end{proof}

Now we prove the analogous result presented above, but for the resolvent of the discrete solution operator:
  \begin{lemma}
 \label{thm:bounded_resolvent}
 Let $F\subset\rho(\bT)$ be closed. Then, there exist
 positive constants $C$ and $h_0$, independent of $h$, such that for $h<h_0$
$$
 \displaystyle\|(z\bI-\bT_h)^{-1}\boldsymbol{f}\|_{1,\O}\leq C_{\mu_s}\|\boldsymbol{f}\|_{1,\O}\qquad\forall z\in F.
$$
 \end{lemma}

Our next task is to derive error estimates for the eigenvalues and eigenfunctions. 
Let $\boldsymbol{E}:\mathbf{Q}\rightarrow\mathbf{Q}$ be the spectral projector of $\bT$ corresponding to the isolated 
eigenvalue $\xi$, namely
$$
\displaystyle \boldsymbol{E}:=\frac{1}{2\pi i}\int_{\gamma} R_{z}(\bT)dz.
$$
On the other, we define $\boldsymbol{E}_h:\mathbf{Q}\rightarrow\mathbf{Q}$ as the spectral projector of $\bT_h$ corresponding to the isolated 
eigenvalue $\xi_h$, namely
$$
\displaystyle \boldsymbol{E}_h:=\frac{1}{2\pi i}\int_{\gamma} R_{z}(\boldsymbol{T}_h)dz.
$$
Let $\kappa$ be an isolated eigenvalue of $\bT$. We define the following distance
$$
\texttt{d}_\kappa:=\frac{1}{2}\dist\left(\kappa,\sp(\bT)\setminus\{\kappa\}\right).
$$
With this distance at hand, we define the disk centered in $\kappa$ and boundary $\gamma$ as follows
$
D_\kappa:=\{z\in\mathbb{C}:\,\,|z-\kappa|\leq \texttt{d}_\kappa\}.
$
We observe that the disk defined above satisfies $D_\kappa\cap\sp(\bT)=\{\kappa\}$.

We owe the following result to \cite[Lemma 5.3]{MR3962898}.

\begin{lemma}
\label{lemma:spectral_projectors}
Let $\boldsymbol{f}\in \bQ$. There exist constants $C_{\mu_s}>0$ and $h_{0}>0$ such that, for all $h<h_{0}$,
$$
	\|  (\boldsymbol{E}-\boldsymbol{E}_{h})\boldsymbol{f} \|_{1,\O}\leq \dfrac{C_{\mu_s}}{\texttt{d}_{\kappa}}\|(\bT-\bT_h)\boldsymbol{f}\|_{1,\Omega}\leq\dfrac{C_{\mu_s}}{\texttt{d}_{\kappa}}\,h^{ \min\{ s,k \} } \| \boldsymbol{f} \|_{1,\O}.
$$
\end{lemma}

\subsection{A priori error estimates}
\label{sec:conv}
Now our aim is to obtain error estimates for the eigenfunctions and eigenvalues. with our method.
We begin by noticing that, according to Corollary \ref{cor:app_TT1}, if $\eta \in (0,1)$ is an isolated eigenvalue of $\bT$ with multiplicity $m$, and $\mathfrak{E}$ its associated eigenspace, then, there exist $m$
eigenvalues $\eta_{h}^{(1)},...,\eta_{h}^{(m)}$ of $\bT_{h}$, repeated according to their respective multiplicities, which converge to $\eta$.
Let $\mathfrak{E}_{h}$ be the direct sum of their corresponding associated eigenspaces (see \cite{MR0203473}) and let us define the  \textit{gap} $\hdel$ between two closed
subspaces $\CM$ and $\CN$ of $\H^1(\O)^n$ by
$$
\hdel(\CM,\CN)
:=\max\big\{\delta(\CM,\CN),\delta(\CN,\CM)\big\}, \text{ where } \delta(\CM,\CN)
:=\sup_{\underset{\left\|x\right\|_{1,\O}=1}{x\in\CM}}
\left(\inf_{y\in\CN}\left\|x-y\right\|_{1,\O}\right).
$$
With these definitions and hand, we derive the following error estimates for eigenfunctions and eigenvalues. Since the proof is direct from 
applying the  results of  \cite{MR1115235,MR1655512,MR1642801}, we do not incorporate further details.
\begin{theorem}
	\label{millar2015}
	For $k\geq1$, the following error estimates for the eigenfunctions and eigenvalues hold
$$
		\widehat{\delta} ( \mathfrak{E}, \mathfrak{E}_{h} )\leq \frac{C_{\mu_s}}{\texttt{d}_\kappa}\,h^{ \min\{ s,k \}} \quad \mbox{and} \quad | \eta-\eta_{h}(i) | \leq \frac{C_{\mu_s}}{\texttt{d}_\kappa}\,h^{ \min\{s,k \}},
$$
	where the hidden constants are independent of $h$.
\end{theorem}

The following result states a preliminary error estimate for the eigenvalues.
\begin{lemma}\label{lmmtriple}
Let $(\kappa,\bu,\boldsymbol{\omega},p)\in\mathbb{R}\times\H_0^1(\O)^n\times\L^2(\O)^{n(n-1)/2}\times\L^2(\O)$ be the solution of Problem \ref{problem1} with $\|\bu\|_{1,\O}=1$
and let $(\kappa_h,\bu_h,\boldsymbol{\omega}_h,p_h)\in\mathbb{R}\times\mathbf{H}_h\times\mathbf{Z}_h\times Q_h$ be its finite element approximation given as  the solution of Problem \ref{problem1_discrete} with $\|\bu_h\|_{1,\O}=1$. Then, there exists $C_{\mu_s}>0$ such that 
\begin{equation*}
\vertiii{(\bu-\bu_h,\boldsymbol{\omega}-\boldsymbol{\omega}_h,p-p_h)}\leq \frac{C_{\mu_s}}{\texttt{d}_\kappa}h^{\min\{s,k\}}.
\end{equation*}
\end{lemma}
\begin{proof}
The result follows by adapting the proof of \cite[Lemma 12]{MR3451491} to our spectral problem.
\end{proof}

We are now in position to provide an improvement in the order of convergence presented in Theorem \ref{millar2015}. } 
\begin{theorem}
	For $k\geq 1$, there exists a strictly positive constant such that, for $h<h_0$ there holds
$$
		|\kappa-\kappa_h|\leq \frac{C_{\mu_s}}{\texttt{d}_\kappa^2} h^{2\min\{s,k\}},
$$
	where $C_{\mu_s}>0$,  is independent of $h$.
\end{theorem}
\begin{proof}
Let us define $\boldsymbol{U}:=(\bu,\boldsymbol{\omega}, p)$ and $\boldsymbol{U}_h:=(\bu_h,\boldsymbol{\omega}_h, p_h)$ with $\|\bu\|_{1,\O}=\|\bu_h\|_{1,\O}=1$. Recalling the definition of $\mathcal{A}(\cdot,\cdot)$ given in \eqref{eq:formA} and the previous definitions,  we consider the following  eigenvalue problems
$$
	\mathcal{A}(\boldsymbol{U},\boldsymbol{U})=-\kappa(\bu,\bu),\quad
	\mathcal{A}(\boldsymbol{U}_h,\boldsymbol{U}_h)=-\kappa_h(\bu_h,\bu_h),
$$
and the well known identity
$$
	(\kappa-\kappa_h)(\bu_h,\bu_h)=\mathcal{A}(\boldsymbol{U}-\boldsymbol{U}_h,\boldsymbol{U}-\boldsymbol{U}_h)+\kappa(\bu-\bu_h,\bu-\bu_h)_{0,\O}.
$$
Taking absolute value to this identity, we obtain that
$$
\begin{aligned}
	\left|(\kappa-\kappa_h)(\bu_h,\bu_h)\right|\leq\underbrace{\left|\mathcal{A}(\boldsymbol{U}-\boldsymbol{U}_h,\boldsymbol{U}-\boldsymbol{U}_h)\right|}_{\mathbf{I}}+\underbrace{\left|\kappa(\bu-\bu_h,\bu-\bu_h)_{0,\O}\right|}_{\mathbf{II}}.
\end{aligned}
$$
For the term $\mathbf{I}$, let us consider the decompositions $p=p_0+P_m{p}$ and $p_h=p_{0,h}+P_mp_h$. Then, using that $\bu=\bu_h=\boldsymbol{0}$ on $\partial\Omega$ and integration by parts we have 
$$
\begin{aligned}
&\left|\mathcal{A}(\boldsymbol{U}-\boldsymbol{U}_h,\boldsymbol{U}-\boldsymbol{U}_h)\right|\\
&\leq|a((\boldsymbol{\omega}-\boldsymbol{\omega}_h,p-p_h),(\boldsymbol{\omega}-\boldsymbol{\omega}_h,p-p_h))|+2|b((\boldsymbol{\omega}-\boldsymbol{\omega}_h,p-p_h),\bu-\bu_h)|\\
&\leq\Vert \boldsymbol{\omega}-\boldsymbol{\omega}_h\Vert_{0,\Omega}^2+(2\mu_s+\lambda_s)^{-1}\Vert p - p_h\Vert_{0,\Omega}^2\\
&\hspace{2cm}+2\int_{\O}\left|\div(\bu-\bu_h)(p-p_h)\right|+2\sqrt{\mu_s}\int_{\O}\left|(\boldsymbol{\omega}-\boldsymbol{\omega}_h)\cdot\curl(\bu-\bu_h)\right|\\
&\leq \mu_s\Vert \curl(\bu-\bu_h)\Vert_{0,\Omega}^2+ 2\Vert \boldsymbol{\omega}-\boldsymbol{\omega}_h\Vert_{0,\Omega}^2 + \mu_s^{-1}\Vert p_0-p_{0,h}\Vert_{0,\Omega}^2\\
&\hspace{4cm}+ \mu_s\Vert\div(\bu-\bu_h)\Vert_{0,\Omega}^2+(2\mu_s+\lambda_s)^{-1}\Vert p - p_h\Vert_{0,\Omega}^2
\\
&\leq C\vertiii{(\bu-\bu_h,\boldsymbol{\omega}-\boldsymbol{\omega}_h,p-p_h)}^2,
\end{aligned}
$$
with $C=2$. on the other hand, for the  term $\mathbf{II}$ it follows that
\begin{multline*}
\left|\kappa(\bu-\bu_h,\bu-\bu_h)_{0,\O}\right|\leq |\kappa|\Vert \bu - \bu_h\Vert_{0,\Omega}^2\leq \frac{|\kappa|}{\mu_s}\Vert\bu-\bu_h\Vert_{1,\Omega}^2\\
\leq \frac{|\kappa|}{\mu_s}\vertiii{(\bu-\bu_h,\boldsymbol{\omega}-\boldsymbol{\omega}_h,p-p_h)}^2.
\end{multline*}
Observe that from Lemma \ref{lmmtriple},  we have
$$
\vertiii{({\bu}-{\bu}_h,{\boldsymbol{\omega}}-{\boldsymbol{\omega}}_h,{p}-{p}_h)}^2\leq \frac{C_{\mu_s}}{\texttt{d}_\kappa^2} h^{2\min\{s,k\}}.
$$
On the other hand, by Lemma \ref{lmm:elliptic} and  Poincar\'e inequality,  together with the fact that $\kappa_h^{(i)} \to \kappa$ as $h$
goes to 0, we have
$$(\bu_h,\bu_h)=\dfrac{\mathcal{A}(\boldsymbol{U}_h,\boldsymbol{U}_h)}{\kappa_h}\geq C_{\mu_{s}}\dfrac{\|\bu_h\|_{1,\O}^2}{\kappa_h}>0.$$
This concludes the proof.
	\end{proof}
	\begin{remark}
	\label{rmrk:ordoble}
	From the above proof it is easy to note that 
$$
	|\kappa-\kappa_h|\leq C_{\mu_s}\vertiii{({\bu}-{\bu}_h,{\boldsymbol{\omega}}-{\boldsymbol{\omega}}_h,{p}-{p}_h)}^2.
$$
	This result will be needed later in the paper, when the a posteriori estimator is derived.
	\end{remark}
	
	The following technical result available in \cite{MR3647956} states that,  since the numerical method is spurious free, for $h$ small enough, except for $\kappa_h$, the rest of the eigenvalues of Problem \ref{problem1_discrete} are well separated from $\kappa$.
\begin{proposition}\label{separa_eig}
Let us enumerate the eigenvalues of Problem \ref{problem1}  and Problem \ref{problem1_discrete}  in increasing order as follows: $0<\kappa_1\leq\cdots\kappa_i\leq\cdots$ and 
$0<\kappa_{h,1}\leq\cdots\kappa_{h,i}\leq\cdots$. Let us assume  that $\kappa_J$ is a simple eigenvalue of Problem \ref{problem1_discrete} . Then, there exists $h_0>0$ such that
$$
|\kappa_J-\kappa_{h,i}|\geq\frac{1}{2}\min_{j\neq J}|\kappa_j-\kappa_J|\quad\forall i\leq \dim\mathbf{H}_h,\,\,i\neq J,\quad \forall h<h_0.
$$
\end{proposition}

On the other hand, a correct control of the error $\|\bu-\bu_h\|_{0,\O}$ will be needed for the a posteriori error analysis. In order to obtain such a bound for the error, we proceed as is customary, by means of a duality argument. To perform this analysis, we adapt the arguments presented in \cite{MR3197278} for the Stokes problem.
	\begin{lemma}
	\label{lmm:norma_L2}
	Given $\bF\in \L^2(\O)^n$, let $(\widetilde{\bu},\widetilde{\boldsymbol{\omega}}, \widetilde{p})\in\H_0^1(\O)^n\times\L^2(\O)^{n(n-1)/2}\times \L^2(\O)$ be the solution of \eqref{eq:source} and 
	$(\widetilde{\bu}_h,\widetilde{\boldsymbol{\omega}}_h, \widetilde{p}_h)\in\mathbf{H}_h\times\mathbf{Z}_h\times Q_h$ be its finite element approximation, given as the solution of \eqref{eq:source_h}. Then, there exists a positive constant $C_{\mu_s}$ such that 
$$
	\|\widetilde{\bu}-\widetilde{\bu}_h\|_{0,\O}\leq C_{\mu_s}h^s\vertiii{(\widetilde{\bu}-\widetilde{\bu}_h,\widetilde{\boldsymbol{\omega}}-\widetilde{\boldsymbol{\omega}}_h,\widetilde{p}-\widetilde{p}_h)}.
$$
	\end{lemma}
\begin{proof}
Let us consider the following well posed problem, with source term $\widetilde{\bu}-\widetilde{\bu}_h$: find $(\boldsymbol{z},\boldsymbol{\xi},\phi)\in \H_0^1(\O)^n\times\L^2(\O)^{n(n-1)/2}\times\L^2(\O)$ such that 
$$	\left\{
	\begin{array}{rcll}
a((\boldsymbol{\theta},q),(\boldsymbol{\xi},\phi))+b((\boldsymbol{\theta},q),\boldsymbol{z})&=&0&\forall (\boldsymbol{\theta},q)\in \L^2(\O)^{n(n-1)/2}\times \L^2(\O),\\
b((\boldsymbol{\xi},\phi),\bv)&=& (\widetilde{\bu}-\widetilde{\bu}_h,\bv)&\forall\bv\in\H^1(\O)^n,
\end{array}
	\right.
$$
where the solution satisfies 
$$
\|\boldsymbol{z}\|_{1+s,\O}+\|\boldsymbol{\xi}\|_{s,\O}+\|\phi\|_{s,\O}\leq C\|\widetilde{\bu}-\widetilde{\bu}_h\|_{0,\O},
$$
with $C$ being  a generic positive constant. On the other hand, we have
\begin{multline}
\label{eq:cuenta}
\|\widetilde{\bu}-\widetilde{\bu}_h\|_{0,\O}^2=b((\boldsymbol{\xi},\phi),\widetilde{\bu}-\widetilde{\bu}_h)=b((\boldsymbol{\xi}-\boldsymbol{\xi}_h,\phi-\phi_h),\widetilde{\bu}-\widetilde{\bu}_h)+b((\boldsymbol{\xi}_h,\phi_h),\widetilde{\bu}-\widetilde{\bu}_h)\\
=b((\boldsymbol{\xi}-\boldsymbol{\xi}_h,\phi-\phi_h),\widetilde{\bu}-\widetilde{\bu}_h)+a((\boldsymbol{\widetilde{\omega}}-\widetilde{\boldsymbol{\omega}}_h,\widetilde{p}-\widetilde{p}_h),(\boldsymbol{\xi}-\boldsymbol{\xi}_h,\phi-\phi_h))\\
-b((\widetilde{\boldsymbol{\omega}}-\widetilde{\boldsymbol{\omega}}_h,\widetilde{p}-\widetilde{p}_h),\boldsymbol{z}-\boldsymbol{z}_h)\\
=\int_{\Omega}\div(\widetilde{\bu}-\widetilde{\bu}_h)(\phi-\phi_h)-\sqrt{\mu_s}\int_{\O}(\boldsymbol{\xi}-\boldsymbol{\xi}_h)\cdot\curl(\widetilde{\bu}-\widetilde{\bu}_h)-\int_{\O}\div(\boldsymbol{z}-\boldsymbol{z}_h)(\widetilde{p}-\widetilde{p}_h)\\
+\sqrt{\mu_s}\int_{\O}(\widetilde{\boldsymbol{\omega}}-\widetilde{\boldsymbol{\omega}}_h)\cdot\curl(\boldsymbol{z}-\boldsymbol{z}_h)+\int_{\O}(\widetilde{\boldsymbol{\omega}}-\widetilde{\boldsymbol{\omega}}_h)\cdot(\boldsymbol{\xi}-\boldsymbol{\xi}_h)+(2\mu_s+\lambda_s)^{-1}\int_{\O}(\widetilde{p}-\widetilde{p}_h)(\phi-\phi_h).
\end{multline}
At this point, we recall the decompositions $\widetilde{p}=\widetilde{p}_0+P_m\widetilde{p}$ and $\widetilde{p}_h=\widetilde{p}_{0,h}+P_m\widetilde{p}_{0,h}$ which, together with an integration by parts,  reveal  that
\begin{equation}
\label{eq:int_zeros}
-\int_{\Omega}\div(\boldsymbol{z}-\boldsymbol{z}_h)(\widetilde{p}-\widetilde{p}_h)=-\int_{\Omega}\div(\boldsymbol{z}-\boldsymbol{z}_h)(\widetilde{p}_0-\widetilde{p}_{0,h}).
\end{equation} 
Hence, replacing \eqref{eq:int_zeros} in \eqref{eq:cuenta}, applying Cauchy-Schwarz inequality, approximation errors, multipliying and dividing by $\sqrt{\mu_s}$,  and using the definition of $\vertiii{\cdot}$,  we obtain
$$
\|\widetilde{\bu}-\widetilde{\bu}_h\|_{0,\O}^2\leq C_{\mu_s}h^s\vertiii{(\widetilde{\boldsymbol{\omega}}-\widetilde{\boldsymbol{\omega}}_h,\widetilde{p}-\widetilde{p}_h,\widetilde{\bu}-\widetilde{\bu}_h)}\|\widetilde{\bu}-\widetilde{\bu}_h\|_{0,\O},
$$
which concludes the proof.
\end{proof}

With the above result at hand, we are allowed  to  define the following operators:
for any $\bF\in \L^2(\O)^n$ we introduce the   linear and compact operator $\widehat{\bT}$ defined by 
$$
 \widehat{\bT}:\L^2(\Omega)^n\rightarrow\L^2(\O)^n,\,\,\,\bF\mapsto\widehat{\bT}\bF:=\widetilde{\bu}_h,
$$
where the triplet $(\widetilde{\bu},\widetilde{\boldsymbol{\omega}},\widetilde{p})$  is the solution of \eqref{eq:source}.
Also, we introduce  $\widehat{\bT}_h$  as the discrete linear counterpart of $\widehat{\bT}$, defined by
$$
 \widehat{\bT}_h:\L^2(\Omega)^n\rightarrow\mathbf{H}_h,\,\,\,\bF\mapsto\widehat{\bT}_h\bF:=\widetilde{\bu}_h,
$$
where the triplet $(\widetilde{\bu}_h,\widetilde{\boldsymbol{\omega}}_h,\widetilde{p}_h)$  is the solution of \eqref{eq:source_h}.
It is easy to check that the operators $\widehat{\bT}$ and $\widehat{\bT}_h$ are  self-adjoint with respect to the $\L^2(\O)$ inner product.

Therefore, thanks to Corollary \ref{cor:app_TT1} and  Lemma \ref{lmm:norma_L2}, we guarantee the convergence in norm of $\widehat{\bT}_h$ to $\widehat{\bT}$ as $h\rightarrow 0$, i.e.
\begin{equation}
\label{eq:_norml2surce}
\|\widetilde{\bu}-\widetilde{\bu}_h\|_{0,\O}\leq C_{\mu_s}h^s\vertiii{(\widetilde{\bu}-\widetilde{\bu}_h,\widetilde{\boldsymbol{\omega}}-\widetilde{\boldsymbol{\omega}}_h,\widetilde{p}-\widetilde{p}_h)}
\leq 	C_{\mu_s}h^{2s}.
\end{equation}
Since we have the convergence of $\widehat{\bT}_h$ to $\widehat{\bT}$ as $h\rightarrow 0$, the compactness of $\widehat{\bT}$ and using the fact that $\widehat{\bT}$ is selfadjoint with respect to the $\L^2$ product, we conclude, as we already proved in Corollary \ref{cor:app_TT1} for the $\H^1$ norm, that the eigenfunctions converge also in $\L^2$ norm. Then we have proved the  following result.
\begin{lemma}\label{lm:aprioriloworder}
Let $(\kappa_h,\bu_h,\boldsymbol{\omega}_h,p_h)$ be the solution of Problem \ref{problem1_discrete} with $\|\bu_h\|_{0,\O}=1$. There exists a solution $(\kappa,\bu,\boldsymbol{\omega},p)$ of Problem \ref{problem1} with $\|\bu\|_{0,\O}=1$, such that $\kappa_h\to \kappa$. Moreover, there exists $C_{\mu_s}>0$ such that 
\begin{align*}\label{eq:error1}
\|\bu-\bu_h\|_{0,\O}&\leq C_{\mu_s}h^{2s},
\end{align*} 
where $s\in (0,\widehat{s})$  as in Lemma \ref{th:regularity}.  
\end{lemma}
Now, with all the ingredients at hand, we are in a position to establish a reliable and efficient  a posteriori estimator of  low order.
\subsection{A posteriori error analysis}
\label{sec:apost}
The following section is dedicated to the design and analysis of  an a posteriori error estimator for our 
mixed eigenvalue problem.  Let us remark that for simplicity, we consider eigenvalues with simple multiplicity. Also, the numerical analysis of the a posteriori error estimator will be for the lowest order of approximation.

Let  $\boldsymbol{I}_{h}:\H^{1}(\O)^n\rightarrow \mathbf{H}_h$, be the Scott-Zhang  interpolant of $\bu$. The following lemma establishes the local approximation properties of $\boldsymbol{I}_{h}$ (see \cite{MR1011446}).
\begin{lemma}
\label{I:Scott-Zhang}
There exist constants $c_{1}$, $c_{2}>0$, independent of $h$, such that for all $\bv\in \H^{1}(\O)^n$ there holds
$$
\|\bv-\boldsymbol{I}_{h}\bv\|_{0,T}\leq c_{1} h_{T}\|\bv\|_{1,\triangle_{T}}\quad \forall T\in\CT_{h},
$$
and
$$
\|\bv-\boldsymbol{I}_{h}\bv\|_{0,e}\leq c_{2}h_{e}^{1/2}\|\bv\|_{1,\triangle_{e}}\quad \forall e\in \mathcal{E}_h,
$$
where $\triangle_{T}:=\{T'\in\CT_{h}: T' \text { and } T \text{ share an edge}\}$ and $\triangle_{e}:=\cup\{T'\in\CT_{h}: e\in \mathcal{E}(T')\}$.
\end{lemma}

\subsection{The local and global error indicators} 
Now we present the local indicator associated to our spectral problem. 
The proposed local indicator  is defined as follows
\begin{multline*}
	\label{eq:est_local}
	\zeta_T^2:=\frac{h_T^2}{\mu_s}\|\kappa_h\bu_h - \nabla p_h- \sqrt{\mu_s}\curl\boldsymbol{\omega}_h\|_{0,T}^2
	+\|\sqrt{\mu_s}\curl\bu_h-\boldsymbol{\omega}_h\|_{0,T}^2\\
	+\frac{\mu_s(2\mu_s+\lambda_s)}{3\mu_s+\lambda_s}\|(2\mu_s+\lambda_s)^{-1}p_h+\div\bu_h\|_{0,T}^2+\sum_{e\in\partial T}\frac{h_e}{\mu_s}\|\boldsymbol{J}_e\|_{0,\partial T}^2,
\end{multline*}
where 
$$
\boldsymbol{J}_e:=\left\{
\begin{array}{rcll}
\dfrac{1}{2}\jump{p_h\n+\sqrt{\mu_s}\boldsymbol{\omega}_h\times\boldsymbol{n}},&e\in\mathcal{E}_h(\O),\\
\boldsymbol{0},&e\in\mathcal{E}_h(\partial\O).
\end{array}
\right.
$$
and hence, the a posteriori error estimator is 
\begin{equation}
\label{eq_est_global}
\zeta:=\left(\sum_{T\in\mathcal{T}_h}\zeta_T^2\right)^{1/2}.
\end{equation} 

Now the study focus in the analysis of $\zeta$, where the aim is to prove that the proposed estimator is equivalent with the approximation error of the eigenfunctions. More precisely, as is customary in a posteriori analysis,  the task is to prove that the estimator $\zeta$ is reliable and efficient. We begin with the reliability bound.
\subsection{Reliability}
The goal of this section is to derive a global reliability bound for our proposed estimator defined in \eqref{eq_est_global}. This is contained in the following result.
\begin{theorem}
\label{thrm:gonzalo1}
There exists a constant $C > 0$ independent of $h$ such that
$$
\vertiii{({\bu}-{\bu}_h,{\boldsymbol{\omega}}-{\boldsymbol{\omega}}_h,{p}-{p}_h)}\leq C \left( \zeta+ \dfrac{1}{\sqrt{\mu_s}}|\kappa-\kappa_h|+\dfrac{|\kappa|}{\sqrt{\mu_s}}\|\bu-\bu_h\|_{0,\O}\right).
$$
\end{theorem}
\begin{proof}
Since $(\kappa,\bu,\boldsymbol{\omega},p)\in \H_0^1(\Omega)^n\times \L^2(\Omega)^{n(n-1)/2}\times \L^2(\Omega)$ is the  solution of Problem \ref{problem1} and $(\kappa_h,\bu_h,\boldsymbol{\omega}_h,p_h)\in \mathbf{H}_h\times\mathbf{Z}_h\times Q_h$ is the solution of Problem \ref{problem1_discrete}, we have:
$$\mathcal{A}(({\bu}-{\bu}_h,{\boldsymbol{\omega}}-{\boldsymbol{\omega}}_h,{p}-{p}_h),(\bv_h,\0,0))=(\kappa_h\bu_h-\kappa\bu,\bv_h),$$
where $\bv_h$ is the Scott-Zhang  interpolant of $\bv$.
Let us define the error by $\boldsymbol{e}:=({\bu}-{\bu}_h,{\boldsymbol{\omega}}-{\boldsymbol{\omega}}_h,{p}-{p}_h)\in \H_0^1(\Omega)^n\times \L^2(\Omega)^{n(n-1)/2}\times \L^2(\Omega)$ and let $\vertiii{(\bv-\bv_h,\boldsymbol{\theta},q)}\leq C\vertiii{\boldsymbol{e}}$. Then, by Lemma \ref{lmm:elliptic} and the previous estimate, it follows that  
 \begin{multline*}
C\vertiii{\boldsymbol{e}}^2\leq  \mathcal{A}(\boldsymbol{e},(\bv,\boldsymbol{\theta},q))
=\mathcal{A}(\boldsymbol{e},(\bv-\bv_h,\boldsymbol{\theta},q))+\mathcal{A}(\boldsymbol{e},(\bv_h,\0,0))\\
=\kappa_h(\bu_h,\bv_h)-\kappa(\bu,\bv)-\mathcal{A}((\bu_h,\boldsymbol{\omega}_h,p_h,(\bv-\bv_h,\boldsymbol{\theta},q))\\
=\kappa_h(\bu_h,\bv_h-\bv)+(\kappa_h-\kappa)(\bu_h,\bv)+\kappa(\bu_h-\bu,\bv_h)-\mathcal{A}((\bu_h,\boldsymbol{\omega}_h,p_h,(\bv-\bv_h,\boldsymbol{\theta},q))\\
= (\kappa_h-\kappa)(\bu_h,\bv)+\kappa(\bu_h-\bu,\bv_h)+\int_{\O}\left(\sqrt{\mu_s}\curl \bu_h-\boldsymbol{\omega}_h\right)\cdot\boldsymbol{\theta}\\
-\int_{\O}\left((2\mu_s+\lambda_s)^{-1}p_h+\div\bu_h\right)q +\sum_{T\in\mathcal{T}_h}\left(\kappa_h\int_T\bu_h\cdot(\bv_h-\bv)-\int_T\div(\bv-\bv_h)p_h\right.\\
\left.\int_T\sqrt{\mu_s}\boldsymbol{\omega}_h\cdot\curl(\bv-\bv_h) \right)\\
\leq C\left(\dfrac{1}{\sqrt{\mu_s}}(|\kappa_h-\kappa|+\frac{|\kappa|}{\sqrt{\mu_s}}\|\bu_h-\bu\|_{0,\O})+\zeta\right)\vertiii{(\bv,\boldsymbol{\theta},q)}.
\end{multline*}
The proof is concluded  using the fact that  $\vertiii{(\bv-\bv_h,\boldsymbol{\theta},q)}\leq C\vertiii{\boldsymbol{e}}$.
\end{proof}

Now, from Remark \ref{rmrk:ordoble}, Theorem \ref{thrm:gonzalo1},  and in view of Lemma \ref{lm:aprioriloworder}, the terms $|\kappa-\kappa_h|$ and $\|\bu-\bu_h\|_{0,\O}$ are  highe order terms and so, 
we deduce the following reliability result for $\zeta$.
\begin{corollary}
There exists a constant $C_{1,\mu_s} > 0$ independent of $h$ such that, for all $h < h_0$, there holds
$$
\vertiii{({\bu}-{\bu}_h,{\boldsymbol{\omega}}-{\boldsymbol{\omega}}_h,{p}-{p}_h)}\leq C_{1,\mu_s}  (\zeta+h^{2s}) .
$$
Moreover, there exists a constant $C_{2,\mu_s}>$ such that
$$
|\kappa-\kappa_h|\leq C_{2,\mu_s} (\zeta^2+h^{4s}).
$$

\end{corollary}
\subsection{Efficiency}
In this section we analyze the efficiency for the proposed estimator $\zeta$. To do this task, we introduce some technical preliminaries.
We begin by introducing the bubble functions for two dimensional elements. Given $T\in\mathcal{T}_h$ and $e\in\mathcal{E}(T)$, we let $\psi_T$ and $\psi_e$ be the usual triangle-bubble and edge-bubble functions, respectively (see \cite{MR3059294} for further details about these functions), which satisfy the following properties

\begin{enumerate}
\item $\psi_T\in\mathrm{P}_{\ell}(T)$, with $\ell=3$ for 2D or $\ell=4$ for 3D, $\text{supp}(\psi_T)\subset T$, $\psi_T=0$ on $\partial T$ and $0\leq\psi_T\leq 1$ in $T$;
\item $\psi_e|_T\in\mathrm{P}_{\ell}(T)$,  with $\ell=2$ for 2D or $\ell=3$ for 3D, $\text{supp}(\psi_e)\subset \omega_e:=\cup\{T'\in\mathcal{T}_h\,:\, e\in\mathcal{E}(T')\}$, $\psi_e=0$ on $\partial T\setminus e$ and $0\leq\psi_e\leq 1$ in $\omega_e$.
\end{enumerate}

The following properties, proved in \cite[Lemma 1.3]{MR1284252} for an arbitrary polynomial order of approximation, hold.
\begin{lemma}[Bubble function properties]
\label{lmm:bubble_estimates}
Given $\ell\in\mathbb{N}\cup\{0\}$, and for each $T\in\mathcal{T}_h$ and $e\in\mathcal{E}(T)$, there following estimates  hold
$$
\|\psi_T q\|_{0,T}^2\leq \|q\|_{0,T}^2\leq C \|\psi_T^{1/2} q\|_{0,T}^2\quad\forall q\in\mathrm{P}_{\ell}(T),
$$
$$
\|\psi_e L(p)\|_{0,e}^2\leq \| p\|_{0,e}^2\leq C \|\psi_e^{1/2} p\|_{0,e}^2\quad\forall p\in\mathrm{P}_{\ell}(e),
$$
and 
$$
h_e\|p \|_{0,e}^2\leq C \|\psi_e^{1/2} L(p)\|_{0,T}^2\leq C
 h_e\|p\|_{0,e}^2\quad\forall p\in\mathrm{P}_{\ell}(e),
$$
 where $L$ is the extension operator defined by  $L: \mathcal{C}(e)\rightarrow \mathcal{C}(T)$ with $\mathcal{C}(e)$ and $\mathcal{C}(T)$ being the spaces of continuous functions defined on $e$ and $T$, respectively,  and satisfying $L(p)\in\mathrm{P}_k(T)$ and $L(p)|_e=p$ for all $p\in\mathrm{P}_k(e)$, where the hidden constants depend on $k$ and the shape regularity of the triangulation.
 \end{lemma}
Also, we requiere the following technical result (see \cite[Theorem 3.2.6]{MR0520174}).

\begin{lemma}[Inverse inequality]\label{inversein}
Let $l,m\in\mathbb{N}\cup\{0\}$ such that $l\leq m$. Then, for each $T\in\mathcal{T}_h$ there holds
$$
|q|_{m,T}\leq C h_T^{l-m}|q|_{l,T}\quad\forall q\in\mathrm{P}_k(T),
$$
where the hidden constant depends on $k,l,m$ and the shape regularity of the triangulations.
\end{lemma}

Now our task is to prove the efficiency. In order to simplify the presentation of the material, we define $\mathbf{R}_1:=(\kappa_h\bu_h - \nabla p_h- \sqrt{\mu_s}\curl\boldsymbol{\omega}_h)|_T$ and $\boldsymbol{\chi}|_T:=(\mu_s)^{-1} h_T^2\mathbf{R}_1\psi_T$, where $\psi_T$ is the bubble function defined in Lemma \ref{lmm:bubble_estimates}. Then, we have
\begin{equation}
\label{eq:R1}
(\mu_s)^{-1}h_T^2\|\mathbf{R}_1\|_{0,T}^2\leq C\int_T\mathbf{R}_1\cdot\boldsymbol{\chi}.
\end{equation}
 From the first equation of \eqref{def:lame_system} we have that $\kappa\bu-\sqrt{\mu_s}\curl\boldsymbol{\omega}-\nabla p=0$ in $\O$. Then, using this in \eqref{eq:R1} we have 
\begin{multline}
\label{eq:bound_R1}
C^{-1}(\mu_s)^{-1}h_T^2\|\mathbf{R}_1\|_{0,T}^2\leq \int_T\mathbf{R}_1\cdot\boldsymbol{\chi}=\int_T(\kappa_h\bu_h-\kappa\bu)\cdot\boldsymbol{\chi}\\
+\sqrt{\mu_s}\int_T\curl(\boldsymbol{\omega}-\boldsymbol{\omega}_h)\cdot\boldsymbol{\chi}
+\int_T\nabla(p-p_h)\cdot\boldsymbol{\chi}\\
=C\int_T(\kappa_h\bu_h-\kappa\bu)\cdot\boldsymbol{\chi}
+\sqrt{\mu_s}\int_T(\boldsymbol{\omega}-\boldsymbol{\omega}_h)\cdot\curl\boldsymbol{\chi}
+\int_T(p-p_h)\div\boldsymbol{\chi}\\
\leq ((\mu_s)^{-1/2} h_T\|\kappa_h\bu_h-\kappa\bu\|_{0,T}+\mu_s^{-1/2}\|p-p_h\|_{0,T}\\
+\|\boldsymbol{\omega}-\boldsymbol{\omega}_h\|_{0,T})(\mu_s^{1/2}\|\nabla\boldsymbol{\chi}\|_{0,T}+\mu_s^{1/2} h_T^{-1}\|\boldsymbol{\chi}\|_{0,T}).
\end{multline}
Note that the following estimate holds 
$$
\mu_s^{1/2}(\|\nabla\boldsymbol{\chi}\|_{0,T}+h_T^{-1}\|\boldsymbol{\chi}\|_{0,T})\leq \mu^{-1/2} h_T\|\mathbf{R}_1\|_{0,T},
$$
which we replace in \eqref{eq:bound_R1} in order to obtain
\begin{multline}\label{eq:eficiencia1}
h_T(\mu_s)^{-1/2}\|\mathbf{R}_1\|_{0,T}\leq C( h_T(\mu_s)^{-1/2}\|\kappa_h\bu_h-\kappa\bu\|_{0,T}\\
+(\mu_s)^{-1/2}\|p-p_h\|_{0,T}+\|\boldsymbol{\omega}-\boldsymbol{\omega}_h\|_{0,T}). 
\end{multline}

Now we define the quantities $\mathbf{R}_2:=(\sqrt{\mu_s}\curl\bu_h-\boldsymbol{\omega}_h)|_T$ and $\mathrm{R}_3:=((2\mu_s+\lambda_s)^{-1}p_h+\div\bu_h)|_T$. Then, from  \cite[Lemma 2.2]{MR4461583}  and \cite[Lemma 2.3]{MR4461583}, respectively,  we have
\begin{equation}
\label{eq:eficiencia2}
\|\mathbf{R}_2\|_{0,T}\lesssim \|\boldsymbol{\omega}-\boldsymbol{\omega}_h\|_{0,T}+\sqrt{\mu_s}\|\curl(\bu-\bu_h)\|_{0,T}.
\end{equation}
and 
\begin{equation}
\label{eq:eficiencia3}
(\mu_s^{-1}+(2\mu_s+\lambda_s)^{-1})^{-1/2}\|\mathrm{R}_3\|_{0,T}\lesssim \sqrt{\mu}\|\div(\bu-\bu_h)\|_{0,T}+(2\mu_s+\lambda_s)^{-1/2}\|p-p_h\|_{0,T}.
\end{equation}
Similarly, proceeding as in \cite[lemma 2.4]{MR4461583} we can prove the following property
\begin{align}\nonumber
h_e^{1/2}(\mu_s)^{-1/2}\|\boldsymbol{J}_e\|_{0,e}&\leq C\bigg(\sum_{T\in\omega_{e}}h_T(\mu_s)^{-1/2}\|\mathbf{R}_1\|_{0,T}+h_T(\mu_s)^{-1/2}\|\kappa_h\bu_h-\kappa\bu\|_{0,T}\\\label{eq:eficiencia4}
&+(\mu_s)^{-1/2}\|p-p_h\|_{0,T}+\|\boldsymbol{\omega}-\boldsymbol{\omega}_h\|_{0,T}\bigg).
\end{align}
Finally, gathering \eqref{eq:eficiencia1}, \eqref{eq:eficiencia2}, \eqref{eq:eficiencia3}, and \eqref{eq:eficiencia4}, we have the  following result, which allows us to deduce the efficiency of the error indicator up to higher order terms.
\begin{theorem}[Efficiency]
Let $(\bu,\boldsymbol{\omega},p)$ be the solution to Problem \ref{problem1} and let $(\bu_h,\boldsymbol{\omega}_h,p_h)$ be its finite element approximation, given as the solution to Problem \ref{problem1_discrete}. Then we have
$$
\zeta\leq C\left(\vertiii{(\bu-\bu_h,\boldsymbol{\omega}-\boldsymbol{\omega}_h,p-p_h)}+h^{2s+1}\right),
$$
where $C>0$ and is independent of $h$, $\lambda_s$ and the discrete solution.
\end{theorem}
\begin{proof}
From  \eqref{eq:eficiencia1}--\eqref{eq:eficiencia4} we have 
$$
\zeta\leq C(\vertiii{(\bu-\bu_h,\boldsymbol{\omega}-\boldsymbol{\omega}_h,p-p_h)}+h(\mu_s)^{-1/2}\|\kappa_h\bu_h-\kappa\bu\|_{0,\O}).
$$
Observe that, using the fact $\|\bu\|_{0,\O}=1$, we have 
$
\|\kappa\bu-\kappa_h\bu_h\|_{0,\O}^2\leq 2(\kappa^2\|\bu-\bu_h\|_{0,\O}^2+|\kappa-\kappa_h|^2).
$
Hence
$$
\zeta\leq C(\vertiii{(\bu-\bu_h,\boldsymbol{\omega}-\boldsymbol{\omega}_h,p-p_h)}+h^{2s+1}).
$$
\end{proof}
\section{Numerical experiments}
\label{sec:numerics}
In this section we report some numerical tests in order to assess the performance of the proposed mixed finite  element method. Some of the meshes used in this work were constructed using the meshing software Gmsh \cite{geuzaine2009gmsh}. The numerical scheme have been  implemented in a FEniCS script \cite{AlnaesBlechta2015a}, where the results of the  convergence rates associated to the a priori analysis and  the a  posteriori error estimator were obtained. More precisely, the convergence rates of the eigenvalues have been  obtained with a standard least square fitting and highly refined meshes. 

We denote by $\kappa_{h_i}$ the $i$-th lowest computed eigenvalue, whereas $\kappa_{extr}$ or $\kappa_i$ denotes the corresponding extrapolated eigenvalue. Note that the square root of this values gives the corresponding eigenfrequency. We also denote by $N$ the mesh refinement level, whereas $\texttt{dof}$ denotes the number of degrees of freedom.

Hence, we denote the error on the $i$-th eigenfrequency by $\err(\kappa_i)$ with 
$$
\err(\kappa_i):=\vert \sqrt{\kappa_{h_i}}-\sqrt{\kappa_{i}}\vert.
$$
For the computation of order of convergence, in  two-dimensional geometries, we consider polynomial degrees $k=1,2,3$, whereas for three-dimensional cases we consider only the lowest order $k=1$ due to machine memory limitations.

Let us remark that, when the Poisson  ratio is too close to $1/2$, the numerical method loses stability, which is reflected in the convergence order of approximation of the eigenvalues. To remedy this problem, an option is to stabilize the method, adding a term that compensates the presence of $p_{0,h}$ for   the nearly incompresible case (see \cite[Section 3]{hughes1987new}). This stabilization, which we add to Problem \ref{problem1_discrete} is as follows
	\begin{equation}
		\label{eq:problem-modified-bilinear-discrete}
		\mathcal{A}((\bu_h,\boldsymbol{\omega}_h,p_h),(\bv_h,\boldsymbol{\theta}_h,q_h)) + \alpha^{-1}\sum_{e\in\mathcal{E}_h(\O)}h_e\int_e \jump{p}\,\jump{q}=-\kappa_h(\bu_h,\bv_h)_{0,\Omega}.
	\end{equation}
	The value of $\alpha$ is conditioned by the value of $\nu$, so that the stabilization gets relevant when $\nu\approx 0.5$.
	As in Section \ref{sec:conv}, the arguments of the continuous problem allow us to conclude that \eqref{eq:problem-modified-bilinear-discrete} is well-posed, and therefore we have the respective discrete estimates for the source problem.
The stabilization parameter for the limiting case can take different values. In our work, we will consider those values such that the convergence of the experiment can be captured in low and high order discontinuous pressure elements (see, for example, \cite[Remark 3]{hughes1987new}).

\subsection{A priori numerical experiments}
\label{subsec:uniformes}
In this section we present experiments using uniform meshes in order to study the convergence of the method in the compressible and quasi-incompressible cases. The experiments that we report are performed in two and three dimensions, using different polynomial degrees.

\subsubsection{Test 1: Square domain} 
\label{subsec:test-cuadrado}
In this experiment we consider the unit square $\O:=(0,1)^2$ with boundary conditions $\bu=0$ over the entire domain. An example of the used for this experiment is depicted in Figure \ref{fig:malla-cuadrado}. The mesh refinement level on this mesh is such that the number of elements is asymptotically $2N^2$. The Poisson's coefficient is taken as $\nu\in\{0.35,0.49, 0.499,0.4999\}$.

Table \ref{tabla:square_nu035} show the convergence behavior when we consider different values of $\nu$ for the lowest order polynomial approximation. It notes that the convergence order remains optimal when $\nu$ approaches to the incompressibilty limit. Values of $\nu$ closer to $0.5$ than the ones showed here gave similar result. We also considered higher order polynomials for the near incompresible case, whose behavior is depicted in Figure \ref{fig:error-cuadrado} for $\nu=0.4999$. It is clear that $\vert \sqrt{\kappa_h}-\sqrt{\kappa_{extr}}\vert\approx \mathcal{O}(\texttt{dof}^{-k})\approx \mathcal{O}(h^{2k})$. The noise in the curves for $k=3$ are produced by the closeness between the computed eigenvalues and the extrapolated values, together with round-off error.
\begin{figure}[!h]
	\centering
	\includegraphics[scale=0.16,trim= 0 5.5cm 0 5.5cm, clip]{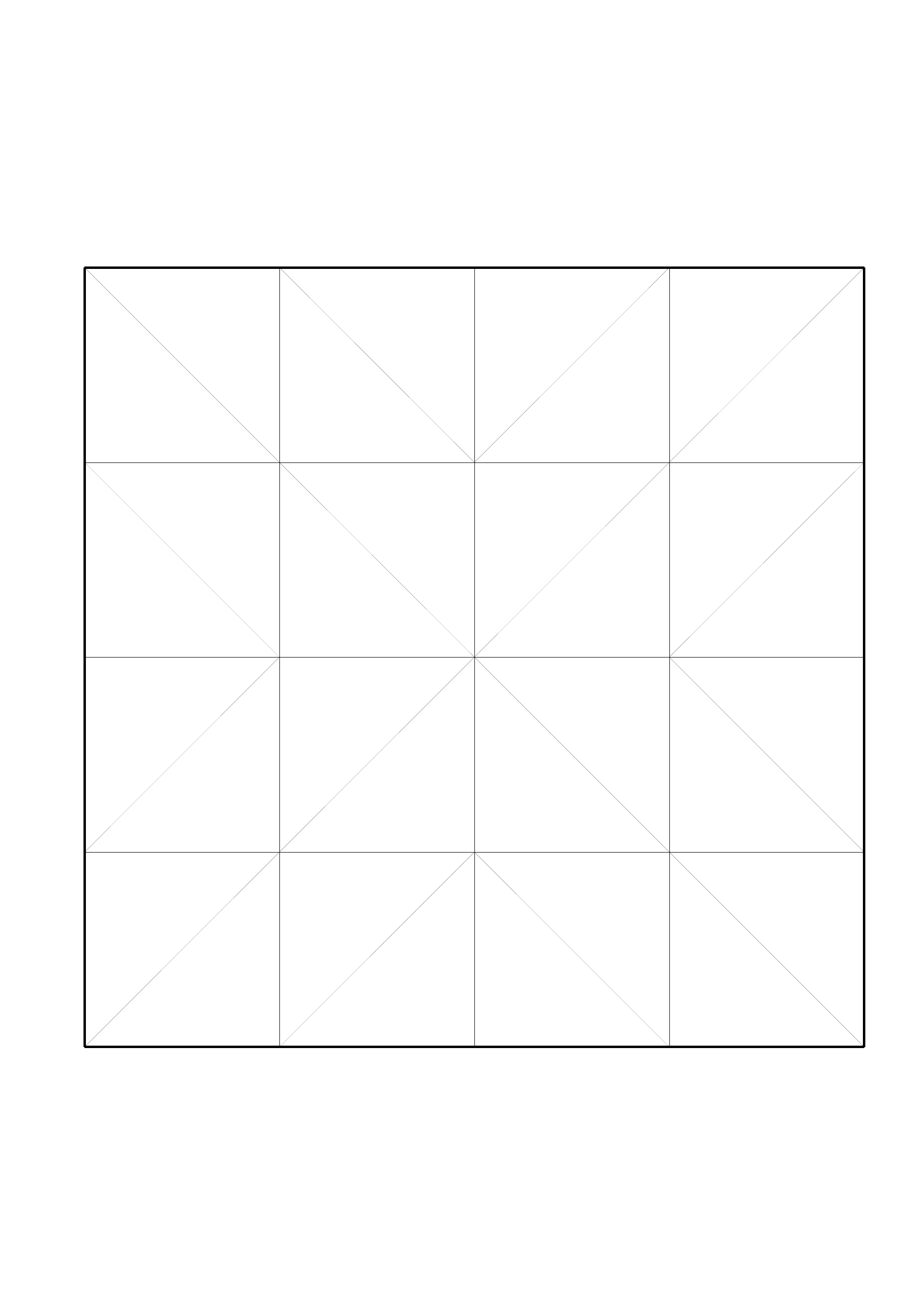}
	\caption{Test 1. Sample mesh of the unit square $\Omega=(0,1)^2$, with $N=4$.}
	\label{fig:malla-cuadrado}
\end{figure}
\begin{table}[h!]
	{\footnotesize
		\begin{center}
			\caption{Test 1. Lowest computed eigenvalues for polynomial degree $k=1$ and different values of $\nu$. }
			\begin{tabular}{c |c c c c |c| c}
				\toprule
				$\nu $        & $N=20$             &  $N=30$         &   $N=40$         & $N=50$ & Order & $\sqrt{\kappa_{extr}}$ \\ 
				\midrule
				& 4.21034 &  4.20058   & 4.19725 &4.19573 & 2.10  & 4.19317   \\
				& 4.21034 &  4.20058   & 4.19725 &4.19573 & 2.11  & 4.19317    \\
				\multirow{2}{0.85cm}{0.35}
				& 4.41054  &  4.38897  & 4.38155 &4.37814 &  2.08 & 4.37235   \\
				& 5.98625  &  5.95624  & 5.94595 &5.94125 &  2.09 & 5.93336   \\
				
				\midrule

				&4.24033&4.21090&4.20083 &4.19627&  2.08&4.18836     \\
				&5.62215&5.56252&5.54232 &5.53317&  2.07&5.51698    \\
				\multirow{2}{0.85cm}{0.49} 
				&5.62215&5.56252&5.54232 &5.53317&  2.07&5.51698   \\
				&6.76422&6.63819&6.59545 &6.57609&  2.08&6.54236    \\

				\midrule
				&4.23323&4.20247&4.19176 &4.18682&  2.03&4.17803     \\
				&5.65631&5.59002&5.56741 &5.55713&  2.07&5.53911   \\
				\multirow{2}{0.85cm}{0.499} 
				&5.65631&5.59002&5.56741 &5.55713&  2.07&5.53911   \\
				&6.77552&6.64170&6.59562 &6.57447&  2.04&6.53699    \\

				\midrule
				
				&4.23261&4.20174&4.19098 &4.18601&  2.01&4.17701     \\
				&5.66107&5.59383&5.57062 &5.55994&  2.03&5.54082    \\
				\multirow{2}{0.85cm}{0.4999} 
				&5.66107&5.59383&5.57062 &5.55994&  2.03&5.54082   \\
				&6.77726&6.64276&6.59637 &6.57502&  2.02&6.53643   \\	
				\bottomrule             
			\end{tabular}
	\end{center}}
	
	\label{tabla:square_nu035}
\end{table}
\begin{figure}[!h]
	\centering
	\begin{minipage}{0.49\linewidth}
		\includegraphics[scale=0.17]{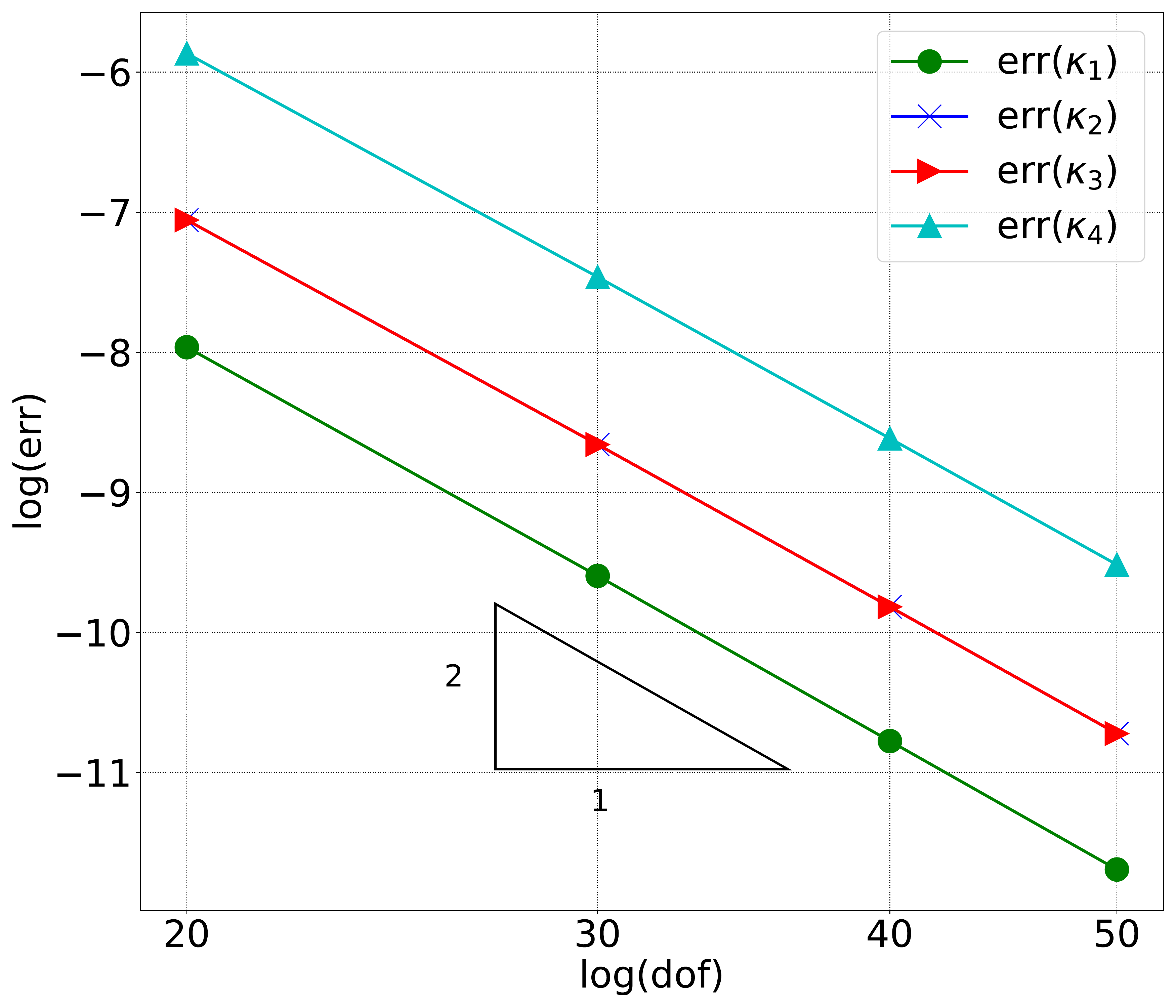}
	\end{minipage}
	\begin{minipage}{0.49\linewidth}
		\includegraphics[scale=0.17]{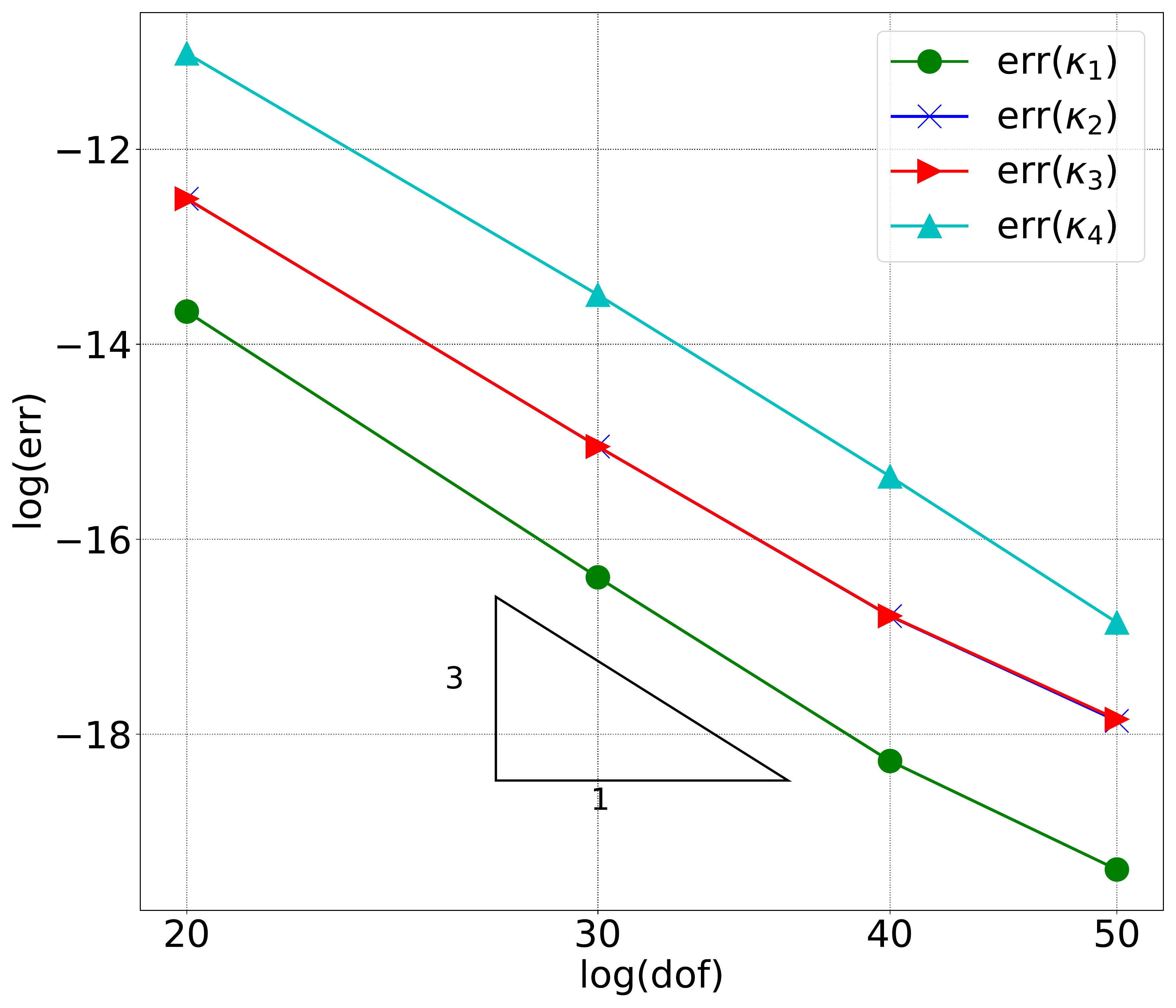}
	\end{minipage}
	\caption{Test 1. Error curves on the incompresibility limit $nu=0.4999$ for $k=2$ (left) and $k=3$ (right).}
	\label{fig:error-cuadrado}
\end{figure}
\subsubsection{Test 2: Torus domain} This test focus on a three dimensional non-polygonal domain. We consider a torus azimuthally symmetric about the $z-$axis, whose domain is characterized by $$\Omega:=\left\{(x,y,z)\in\mathbb{R}^3\;:\;\left(\sqrt{x^2+y^2}-R\right)^2+z^2=r^2\right\},$$ where $R=1/2$ and $r=1/4$. We test the deformation, pressure and rotation with several values of $\nu$, for which the stabilization parameter in the incompresibility case is chosen as $\alpha^{-1}=1/8$. The refinement levels for the mesh are such that the mesh parameter is asymptotically $h=1/N$, where $N=30,40,50,60$.

Table \ref{tabla:torus_tabla} we observe the convergence behavior for the first four lowest eigenfrequencies. Since each slice of the torus is a circle, we are commiting a variational crime when we discretise this domain using polygons. Hence, the convergence rate $\mathcal{O}(h^2)$ will remain for $k>1$. In the case provided in this experiment, the optimal rate is recovered even near the incompressible limit. In Figures \ref{fig:torus-nu035} -- \ref{fig:torus-nu04999} we depict some of the eigenmodes computed in Table \ref{tabla:torus_tabla} with $N=60$.
\begin{table}[h!]
	{\footnotesize
		\begin{center}
			\caption{Test 2. Lowest computed eigenvalues for the lowest order scheme $k=1$ and different values of $\nu$ in the torus domain. }
			\begin{tabular}{c |c c c c |c| c}
				\toprule
				$\nu $        & $N=30$             &  $N=40$         &   $N=50$         & $N=60$ & Order & $\sqrt{\kappa_{extr}}$ \\ 
				\midrule
				& 6.04924 &  5.99800   & 5.97977 &5.97149 & 2.13  & 5.95707   \\
				& 6.43508 &  6.37877   & 6.35886 &6.34987 & 2.16  & 6.33465    \\
				\multirow{2}{0.85cm}{0.35}
				& 6.43601  &  6.37934  & 6.35917 &6.34995 &  2.20 & 6.33538   \\
				& 7.39830  &  7.32671  & 7.30098 &7.28926 &  2.19 & 7.27061   \\
				
				\midrule

				&5.73450&5.69843&5.68588 &5.68013&  2.17&5.67049     \\
				&8.68185&8.60391&8.57378 &8.55986&  1.92&8.53153    \\
				\multirow{2}{0.85cm}{0.49} 
				&8.68229&8.60483&8.57428 &8.56007&  1.87&8.53011   \\
				&8.91109&8.73963&8.67989 &8.65273&  2.17&8.60686    \\

				\midrule				
				&5.71598&5.67978&5.66718 &5.66142&  2.22&5.65241     \\
				&8.91911&8.74350&8.68319 &8.65521&  2.26&8.61376    \\
				\multirow{2}{0.85cm}{0.4999} 
				&8.92069&8.74412&8.68322 &8.65542&  2.27&8.61404   \\
				&8.92674&8.74496&8.68354 &8.65639&  2.28&8.61386   \\	
				\bottomrule             
			\end{tabular}
	\end{center}}	
	\label{tabla:torus_tabla}
\end{table}
\begin{figure}[!h]
	\centering
	\begin{minipage}{0.32\linewidth}
		\centering
		\includegraphics[scale=0.05, trim=40cm 3cm 38cm 10cm, clip]{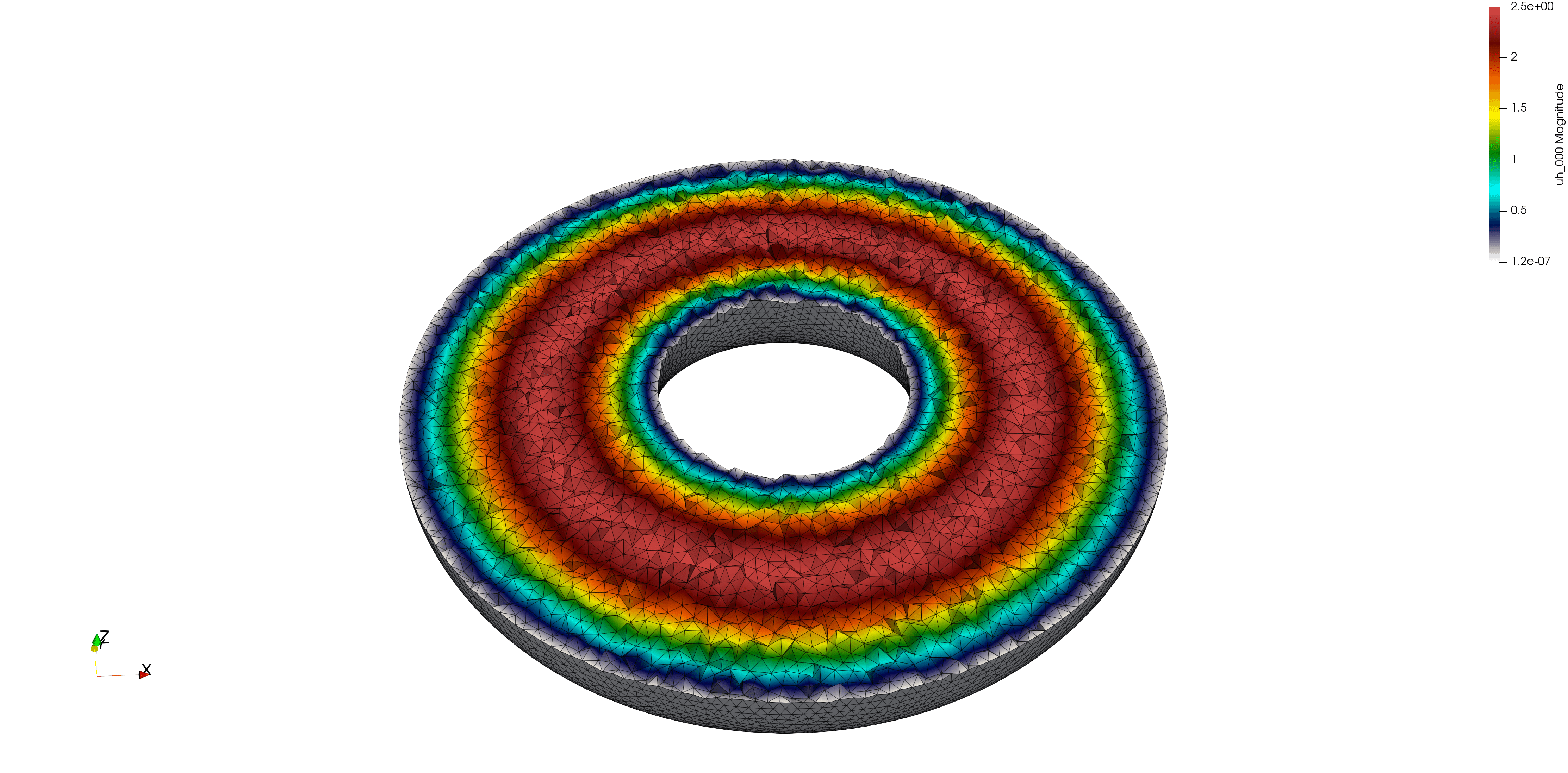}
		{\footnotesize{$|\bu_h|$}}
	\end{minipage}
	\begin{minipage}{0.32\linewidth}
		\centering
		\includegraphics[scale=0.05, trim=40cm 3cm 38cm 10cm, clip]{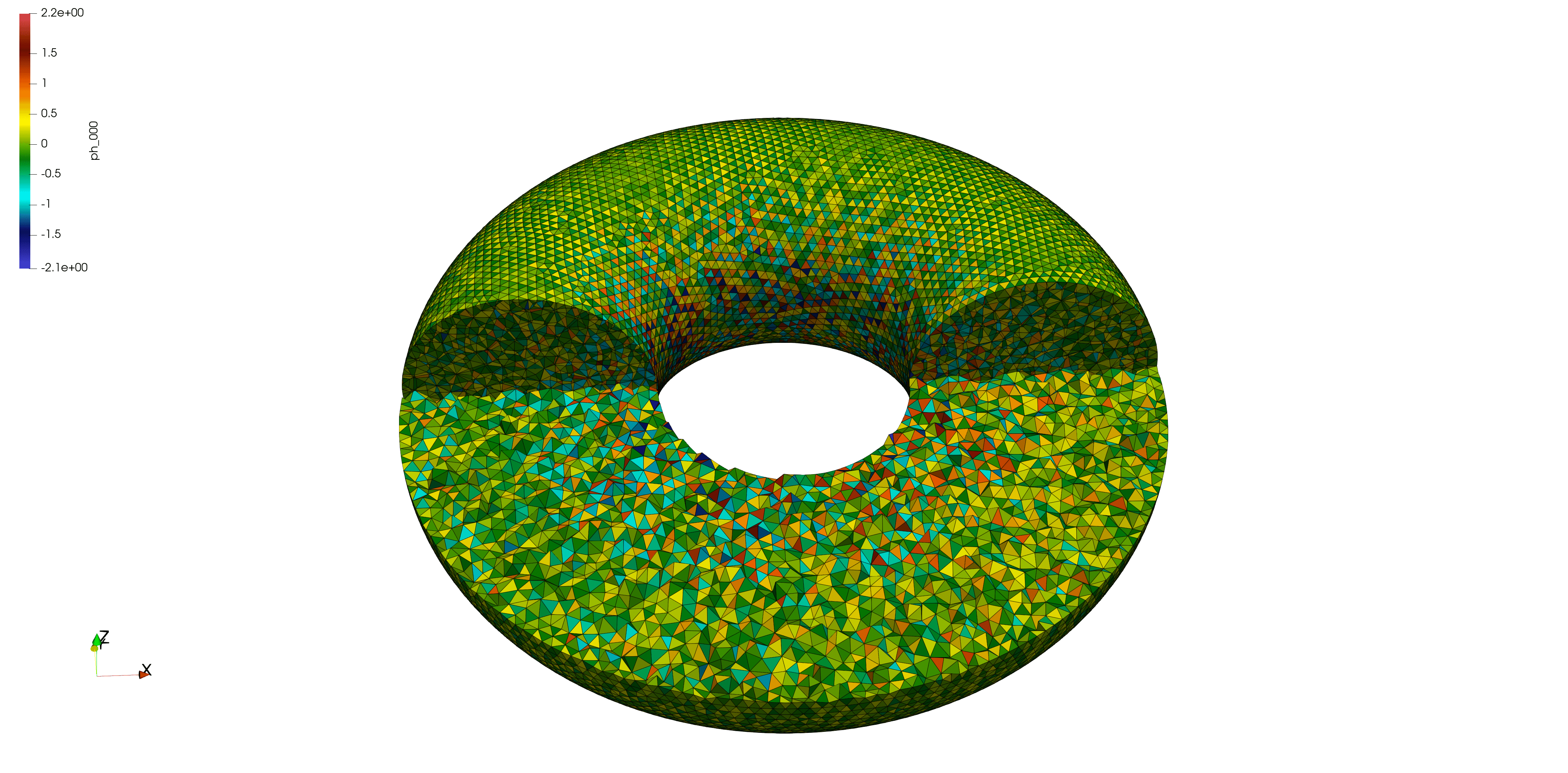}
		{\footnotesize{$p_h$}}
	\end{minipage}
	\begin{minipage}{0.32\linewidth}
		\centering
		\includegraphics[scale=0.05, trim=40cm 3cm 38cm 10cm, clip]{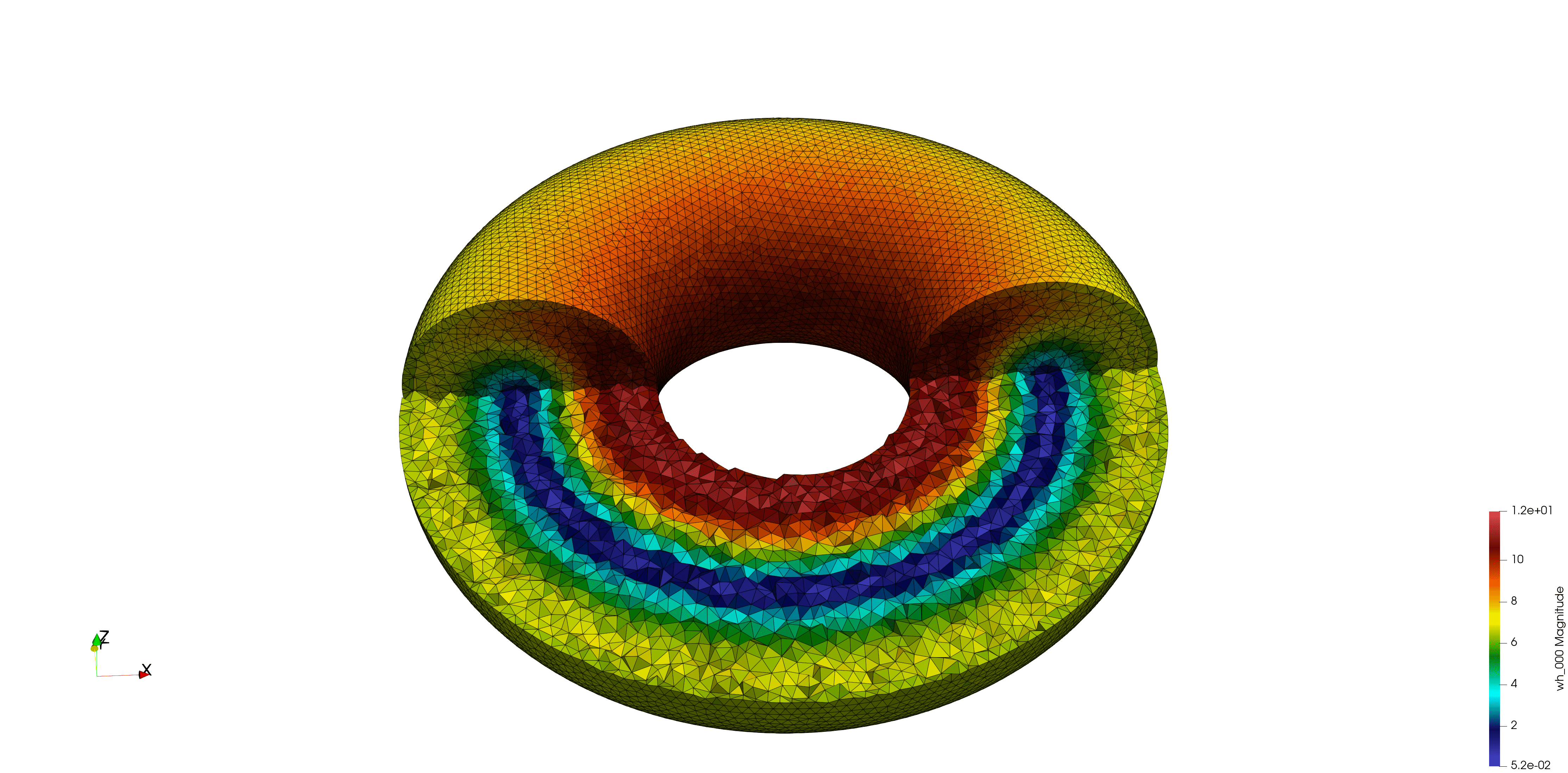}
		{\footnotesize{$|\boldsymbol{\omega}_h|$}}
	\end{minipage}
	\caption{Test 2. Slices of the torus domain showing the eigenmodes for the first lowest frequency, with $\nu=0.35$.}
	\label{fig:torus-nu035}
\end{figure}
\begin{figure}[!h]
	\centering
	\begin{minipage}{0.32\linewidth}
		\centering
		\includegraphics[scale=0.05, trim=40cm 3cm 38cm 10cm, clip]{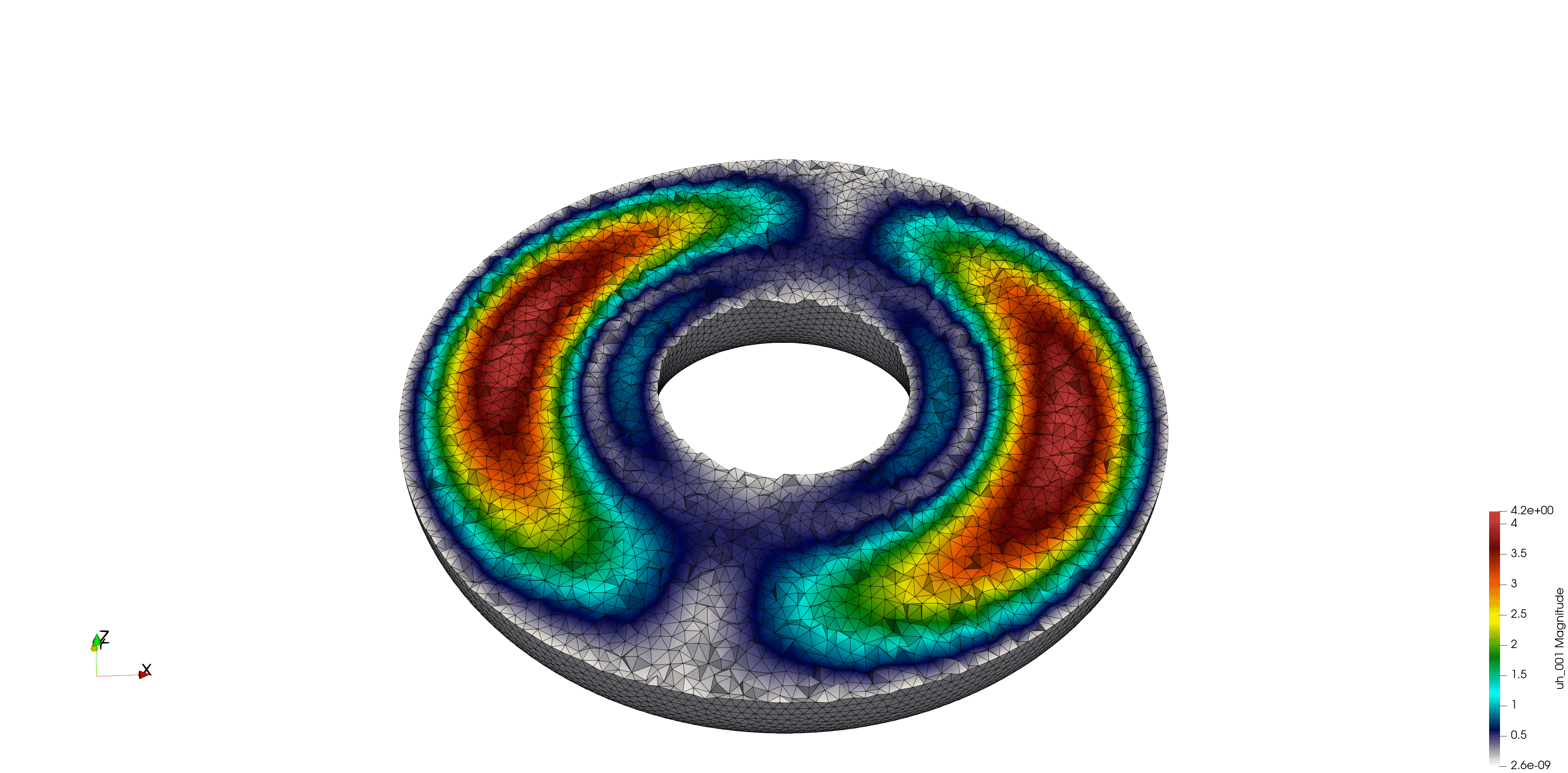}
		{\footnotesize{$|\bu_h|$}}
	\end{minipage}
	\begin{minipage}{0.32\linewidth}
		\centering
		\includegraphics[scale=0.05, trim=40cm 3cm 38cm 10cm, clip]{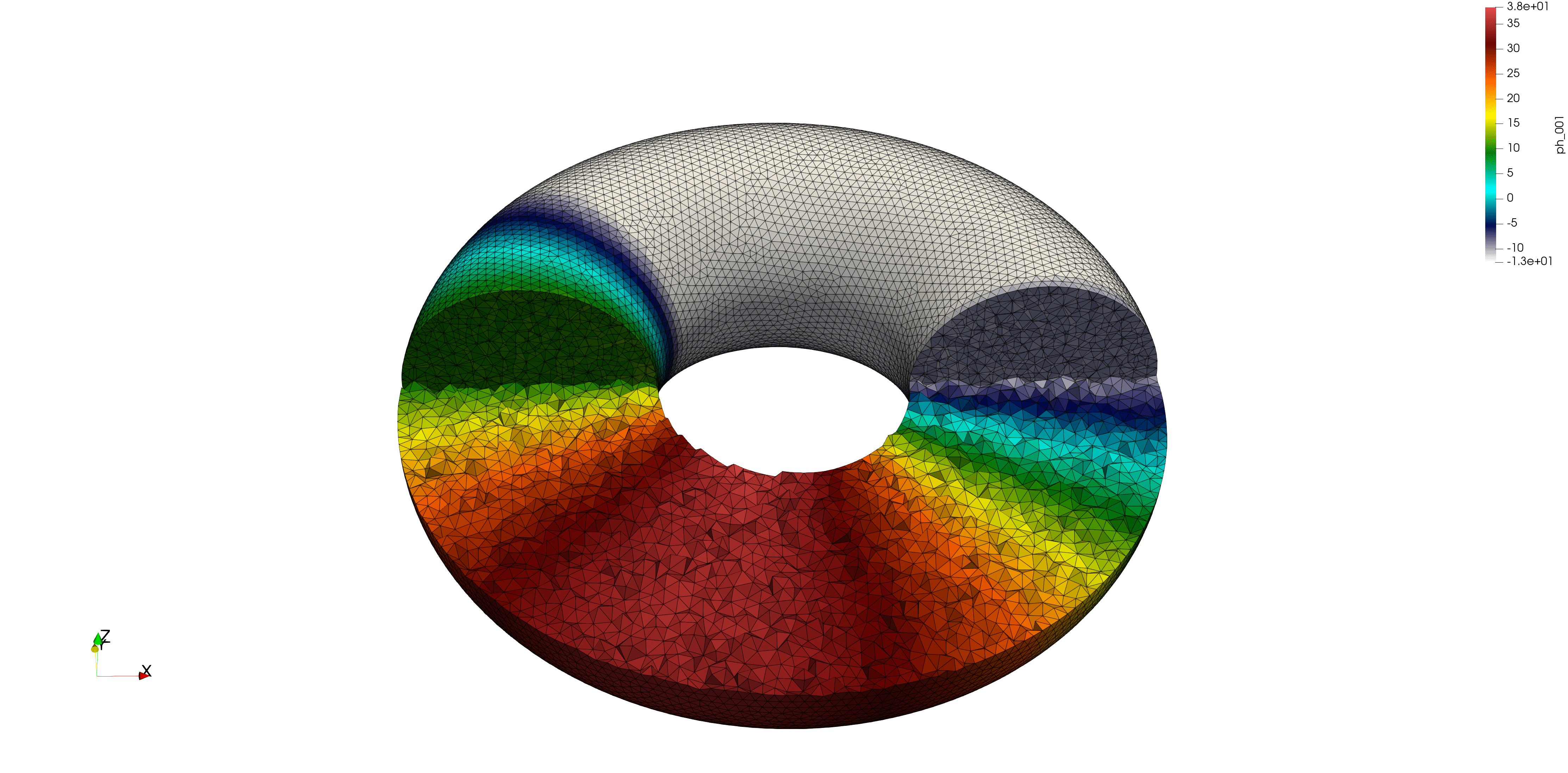}
		{\footnotesize{$p_h$}}
	\end{minipage}
	\begin{minipage}{0.32\linewidth}
		\centering
		\includegraphics[scale=0.05, trim=40cm 3cm 38cm 10cm, clip]{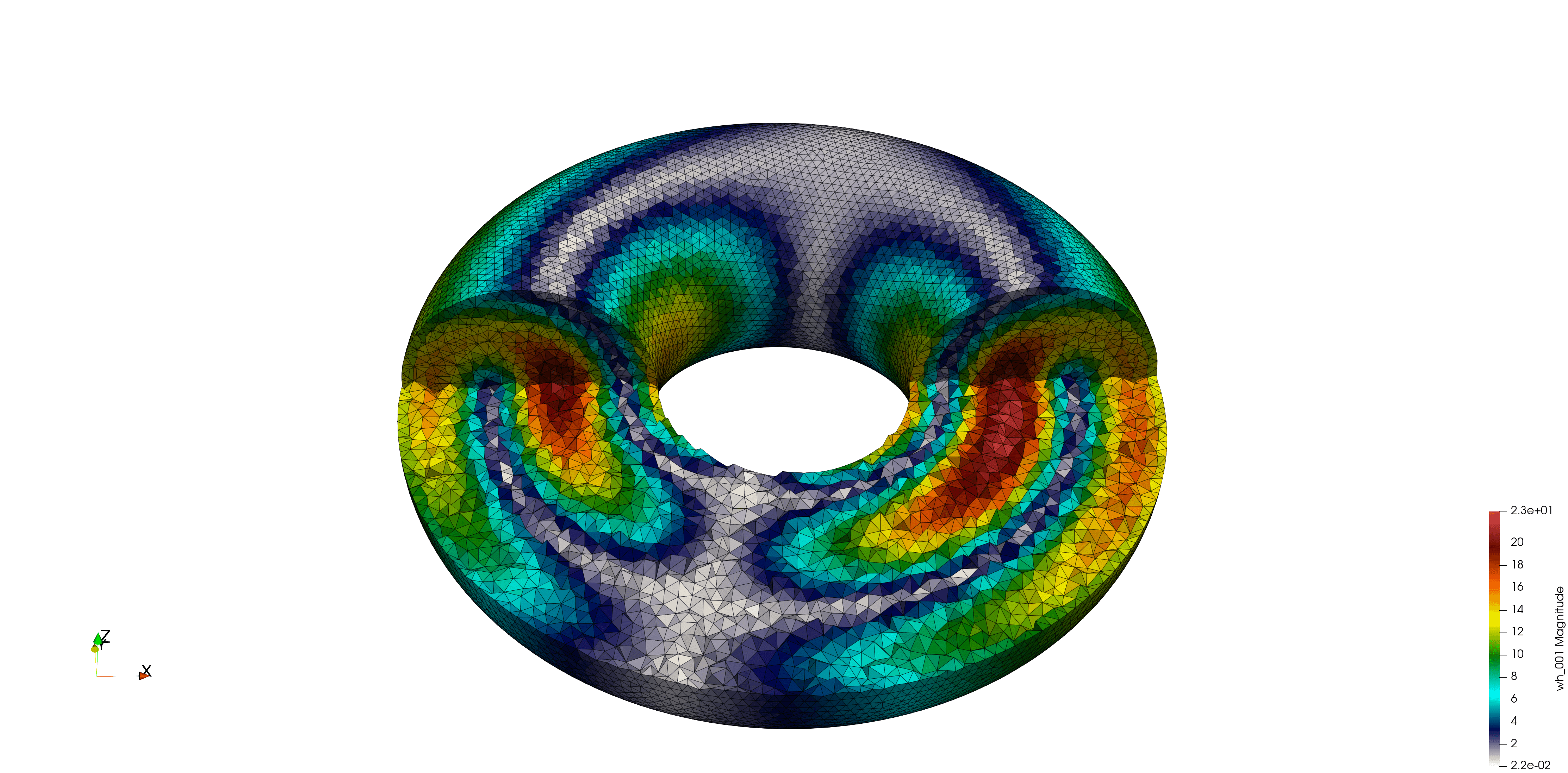}
		{\footnotesize{$|\boldsymbol{\omega}_h|$}}
	\end{minipage}
	\caption{Test 2. Slices of the torus domain showing the eigenmodes for the second lowest frequency, with $\nu=0.49$.}
	\label{fig:torus-nu049}
\end{figure}
\begin{figure}[!h]
	\centering
	\begin{minipage}{0.32\linewidth}
		\centering
		\includegraphics[scale=0.05, trim=40cm 3cm 38cm 10cm, clip]{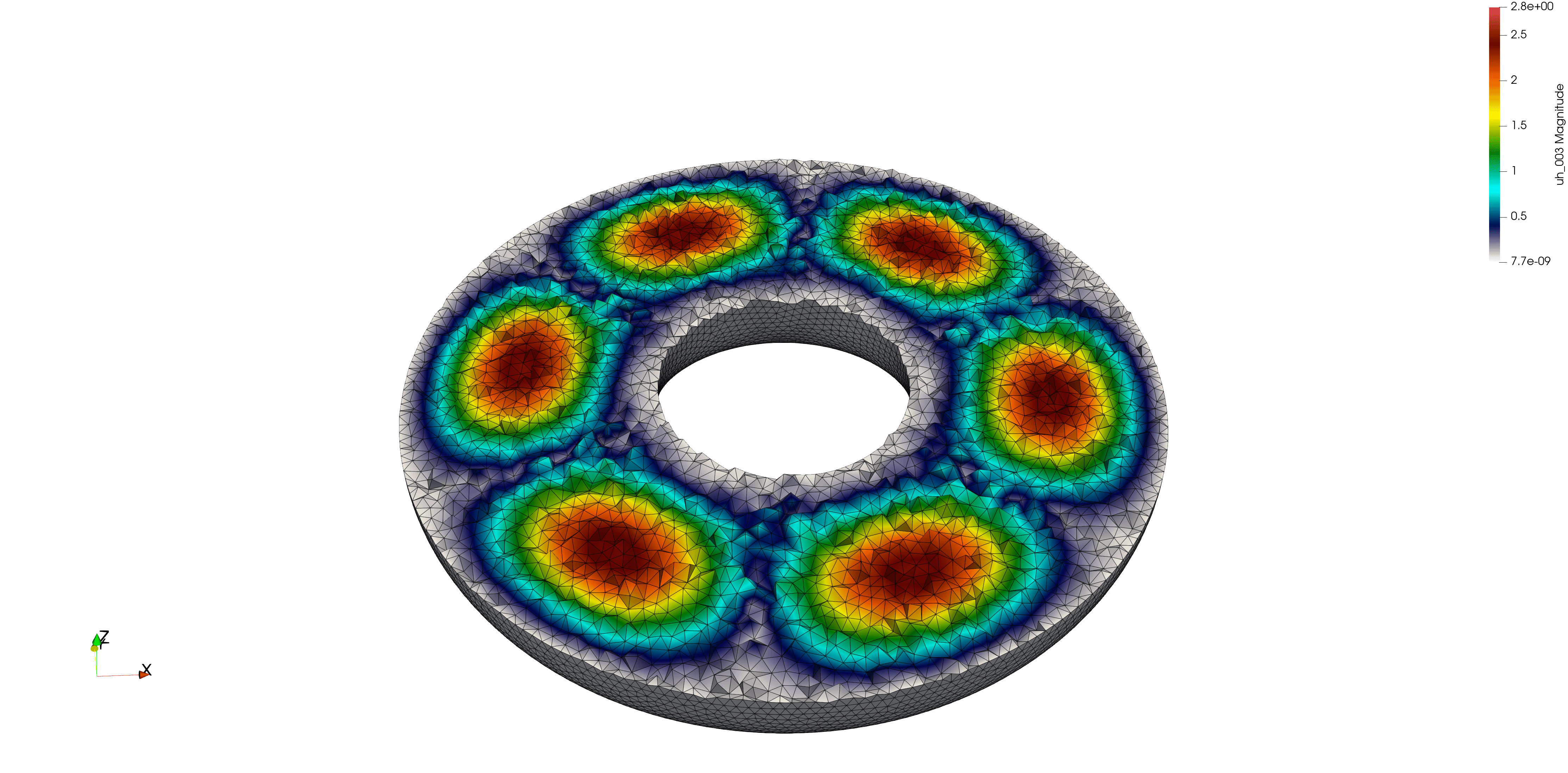}
		{\footnotesize{$|\bu_h|$}}
	\end{minipage}
	\begin{minipage}{0.32\linewidth}
		\centering
		\includegraphics[scale=0.05, trim=40cm 3cm 38cm 10cm, clip]{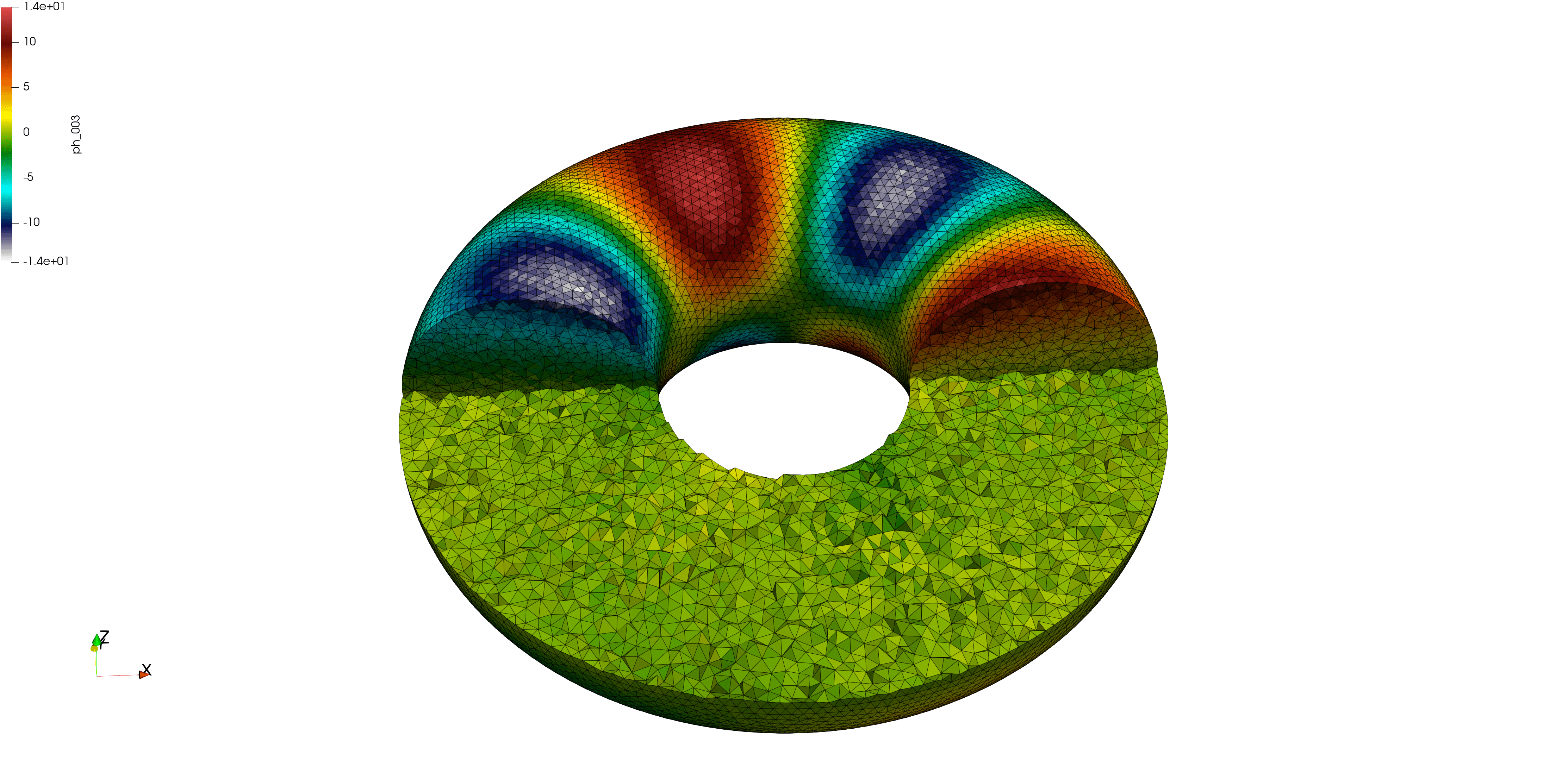}
		{\footnotesize{$p_h$}}
	\end{minipage}
	\begin{minipage}{0.32\linewidth}
		\centering
		\includegraphics[scale=0.05, trim=40cm 3cm 38cm 10cm, clip]{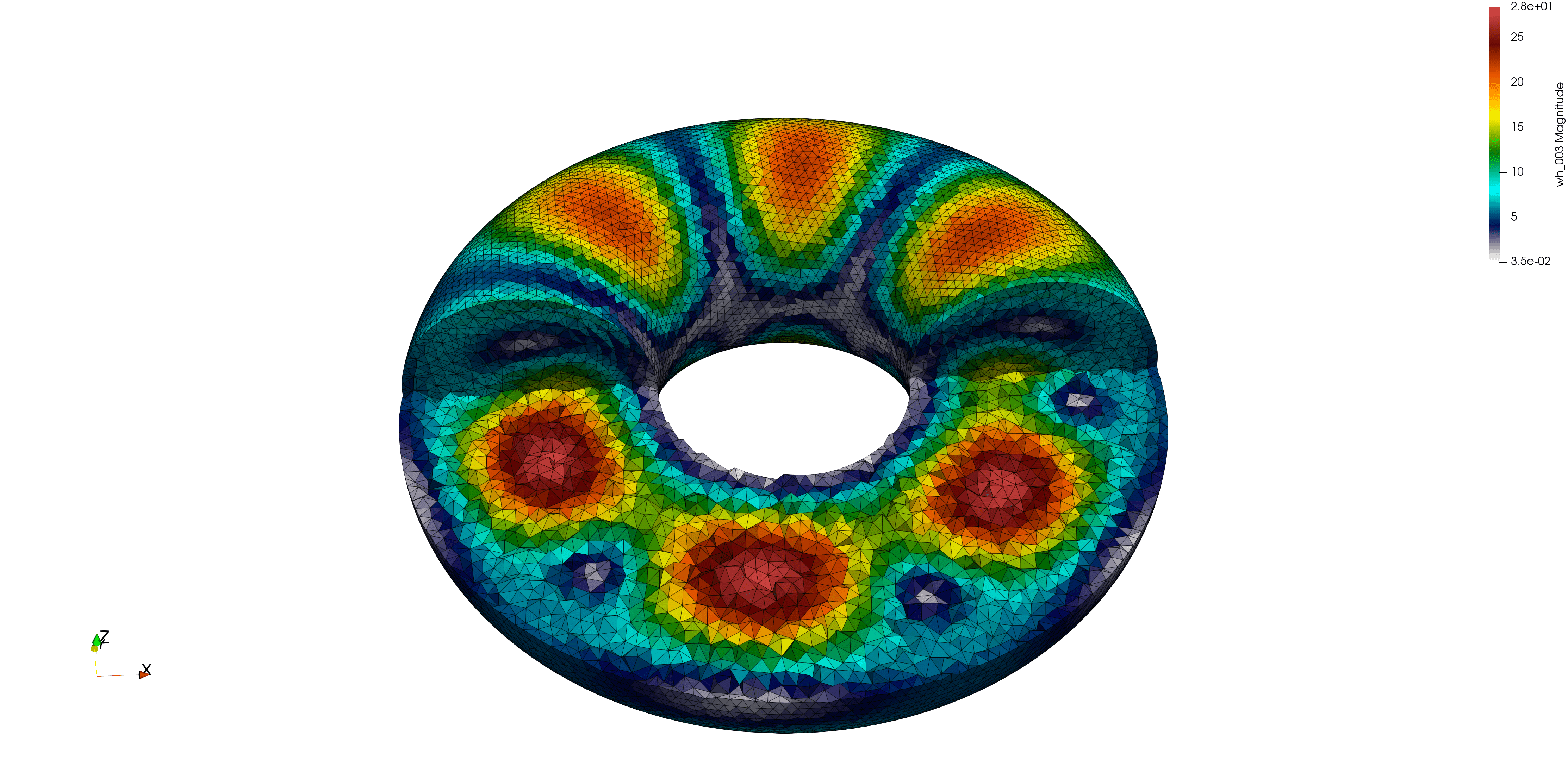}
		{\footnotesize{$|\boldsymbol{\omega}_h|$}}
	\end{minipage}
	\caption{Test 2. Slices of the torus domain showing the eigenmodes for the fourth lowest frequency, with $\nu=0.4999$.}
	\label{fig:torus-nu04999}
\end{figure}
\subsection{Experiments using adaptive refinements}
\label{subsec:test-afem}
In this section we test our a posteriori estimator $\zeta$, previously introduced and analyzed in Section \ref{sec:apost}. With the aim of assess the performance of our estimator,  we consider domains   with singularities in two and three dimensions, expecting that the estimator be capable of identify such singularities and refine adaptively. For this tests, we consider two and three dimensional domains. On each adaptive iteration, we use the blue-green marking strategy to refine each $T'\in \CT_{h}$ whose indicator $\zeta_{T'}$ satisfies
$$
\zeta_{T'}\geq 0.5\max\{\zeta_{T}\,:\,T\in\CT_{h} \},
$$
The effectivity indexes with respect to $\zeta$ and the eigenvalue $\kappa_i$ are defined by
$$
\eff(\kappa_i):=\frac{\err(\kappa_i)}{\zeta^2}.
$$
\subsubsection{Test 3. A square with a hole} In this experiment we test the adaptive algorithm in a domain containing singularities. We consider the domain $\Omega=\widehat{\Omega}\backslash\widetilde{\Omega}$, where $\widehat{\Omega}$ is the unit square rotated $\pi/4$ around its centroid, whereas $\widetilde{\Omega}:=(129/400,271/400)^2$. An example of this domain is depicted in Figure \ref{fig:square-with-hole-inicial}. The stabilization parameter for the incompressible limit is taken to be $\alpha^{-1}=10$.
\begin{figure}[!h]
	\centering
	\includegraphics[scale=0.05, trim=40cm 0cm 40cm 0,clip]{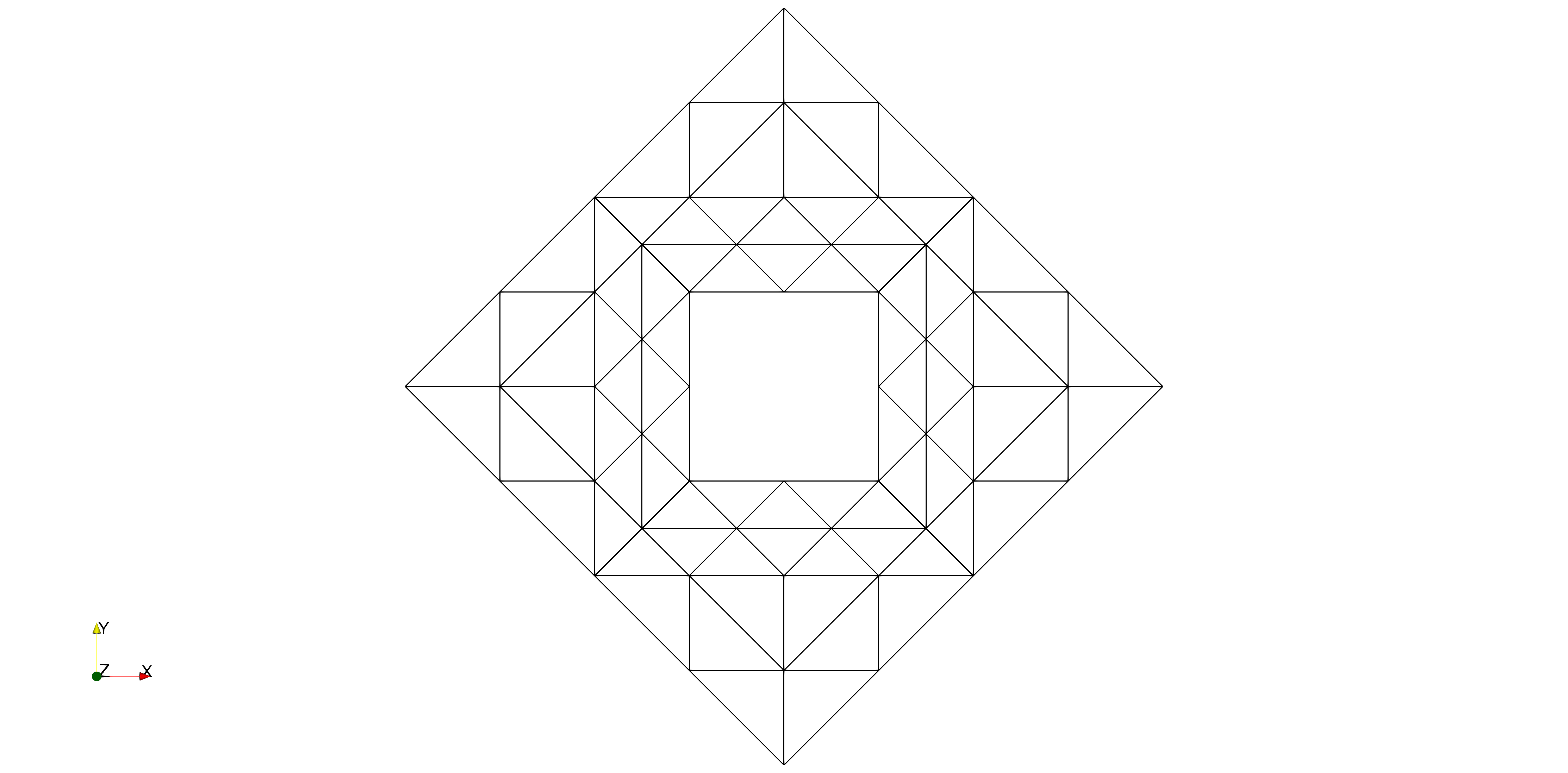}
	\caption{Test 3. Initial mesh on the square with a hole.}
	\label{fig:square-with-hole-inicial}
\end{figure}

In this case, we have 4 reentrant angles, so the first eigenfunction is singular.  In fact, the optimal convergence order with uniform refinements is expected to be approximately $\mathcal{O}(N^{-\sqrt{3}/2})\approx \mathcal{O}(h^{1.7})$ for all $k\geq1$. 
To observe in detail the error behavior, we show in Figure \ref{fig:adaptive-square-with-hole-errores} the adaptive and uniform error curves, where the rapid decay of the adaptive error is evident for $k>1$. Moreover, in Figures \ref{fig:adaptive-square-with-hole-eff-etas-035}-\ref{fig:adaptive-square-with-hole-eff-etas-04999} we observe that the estimator $\zeta^2$ behaves like $\mathcal{O}(h^{2k})$, keeping the effectivity indexes appropriate bounded above and below. This verifies that our estimator in efficient and reliable.  The experimental convergence rate for this experiment is based on the comparison of the computed eigenvalues with the values on Table \ref{tabla:extrapolated_square_hole}.
\begin{table}[!h]
	\centering
	\begin{tabular}{c|c}
		$\nu$& $\sqrt{\kappa_1}$\\
		\midrule
		0.35 & 6.14518\\
		0.4999 & 6.04518
	\end{tabular}
	\caption{Lowest computed eigenvalues obtained from highly refined meshes and least-square fitting.}
	\label{tabla:extrapolated_square_hole}
\end{table}
On the other hand, Figures \ref{fig:adaptive-square-with-hole-nu035}--\ref{fig:adaptive-square-with-hole-nu04999} show the meshes at different stages of the adaptive refinement for $\nu=0.35$ and $\nu=0.4999$, respectively. We complete this experiment by showing the first eigenmode magnitudes, together with the computed pressure and rotations, which in this case are similar for the selected values of $\nu$. It notes that high gradients of pressure, together with high rotations are present near the singularities.
\begin{figure}[!h]
	\centering
	\begin{minipage}{0.495\linewidth}\centering
		\includegraphics[scale=0.17, trim=40cm 0cm 40cm 0,clip]{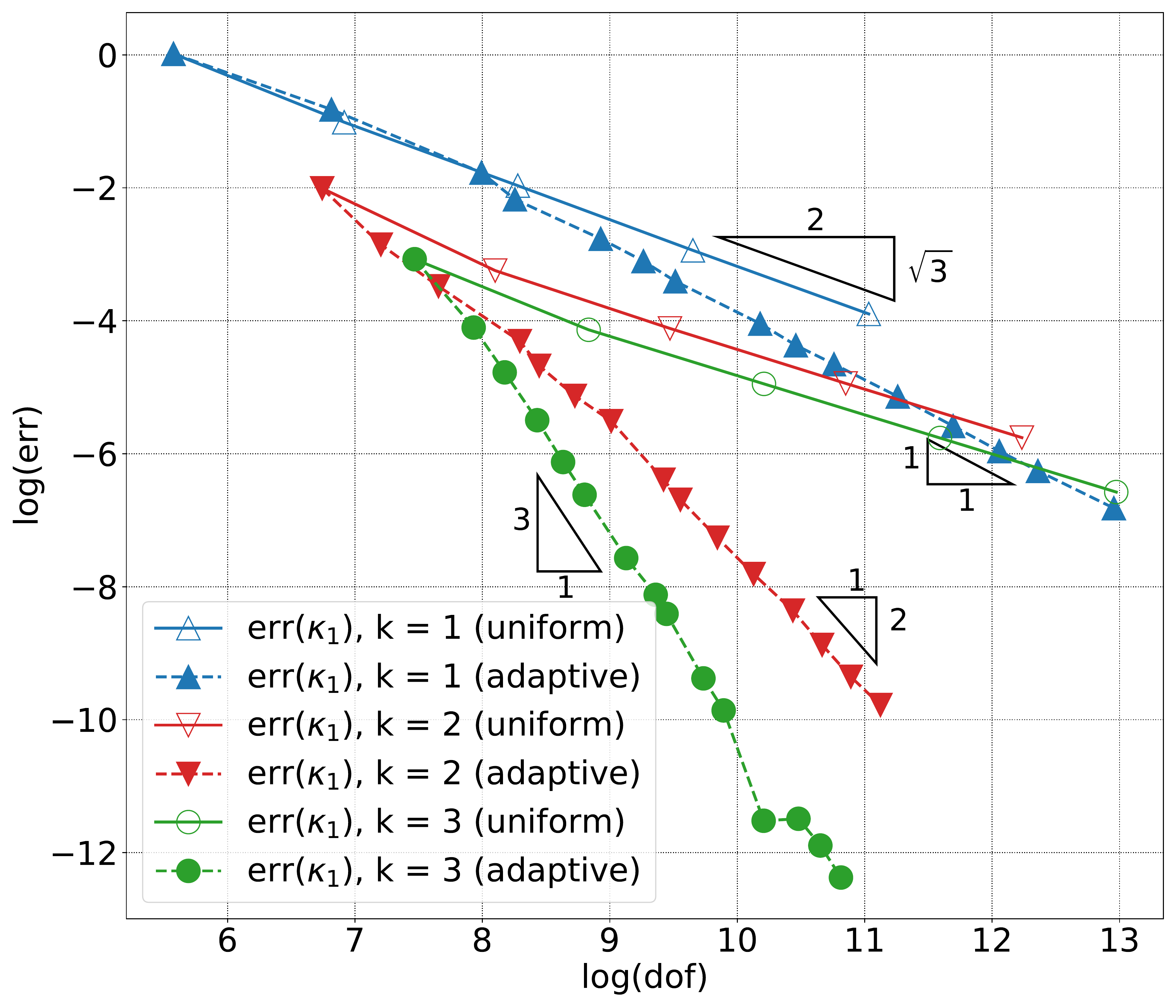}\\
		{\footnotesize $\nu=0.35$}
	\end{minipage}
	\begin{minipage}{0.495\linewidth}\centering
		\includegraphics[scale=0.17, trim=40cm 0cm 40cm 0,clip]{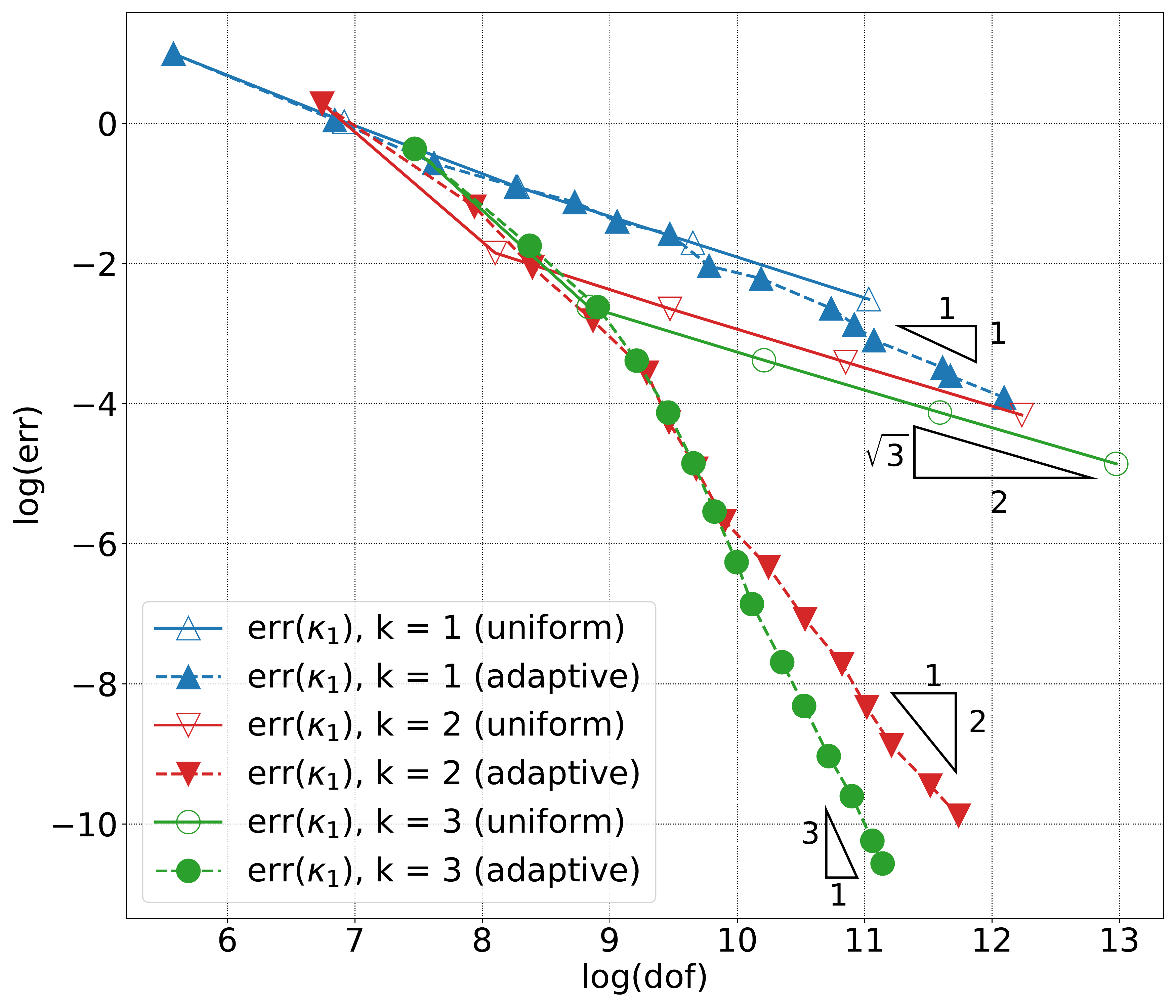}\\
		{\footnotesize $\nu=0.4999$}
	\end{minipage}\\
	\caption{Test 3. Error curves for different values of $k$ and $\nu$.}
	\label{fig:adaptive-square-with-hole-errores}
\end{figure}
\begin{figure}[!h]
	\centering
	\begin{minipage}{0.495\linewidth}\centering
		\includegraphics[scale=0.17, trim=40cm 0cm 40cm 0,clip]{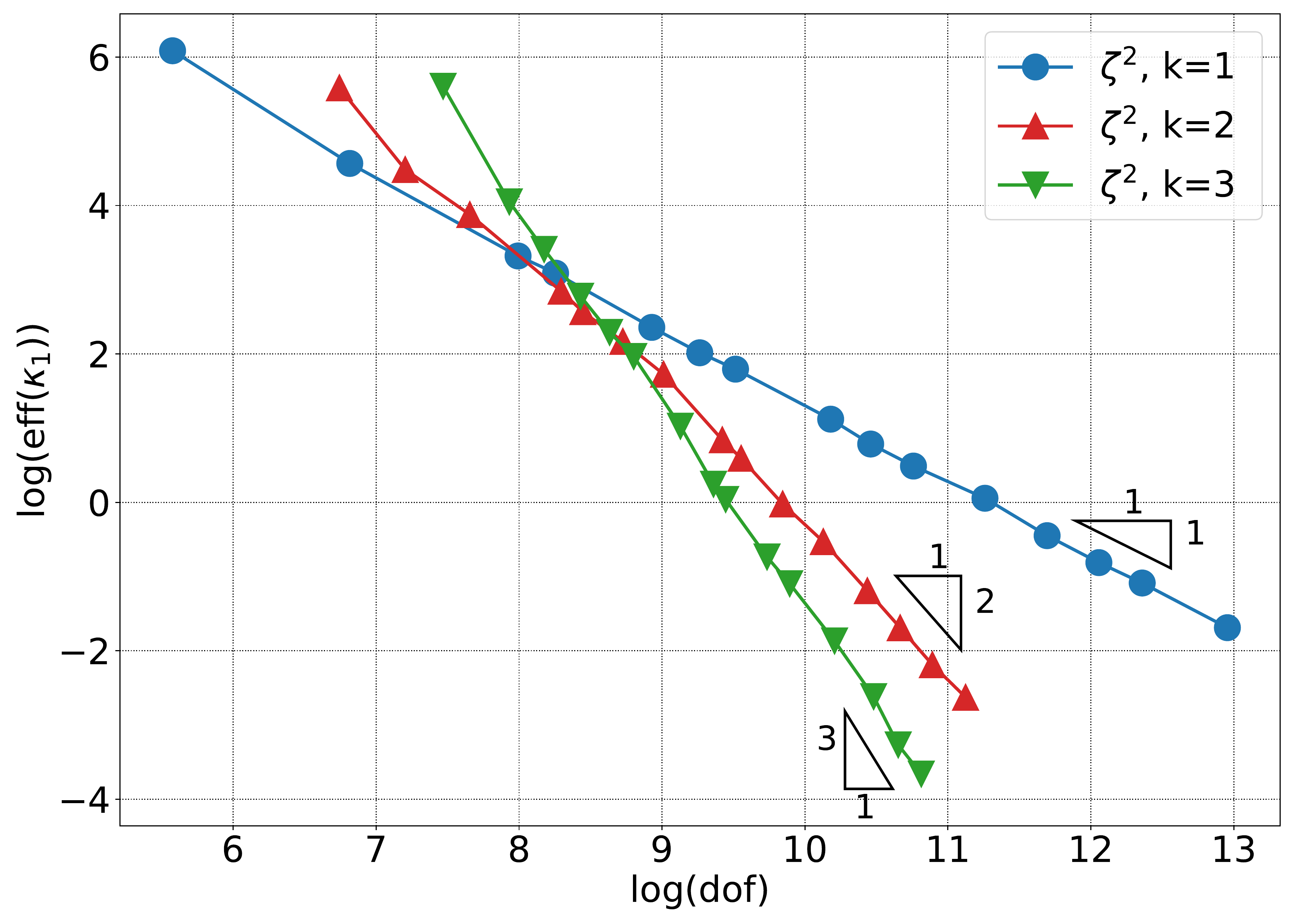}\\
	\end{minipage}
	\begin{minipage}{0.495\linewidth}\centering
		\includegraphics[scale=0.17, trim=40cm 0cm 40cm 0,clip]{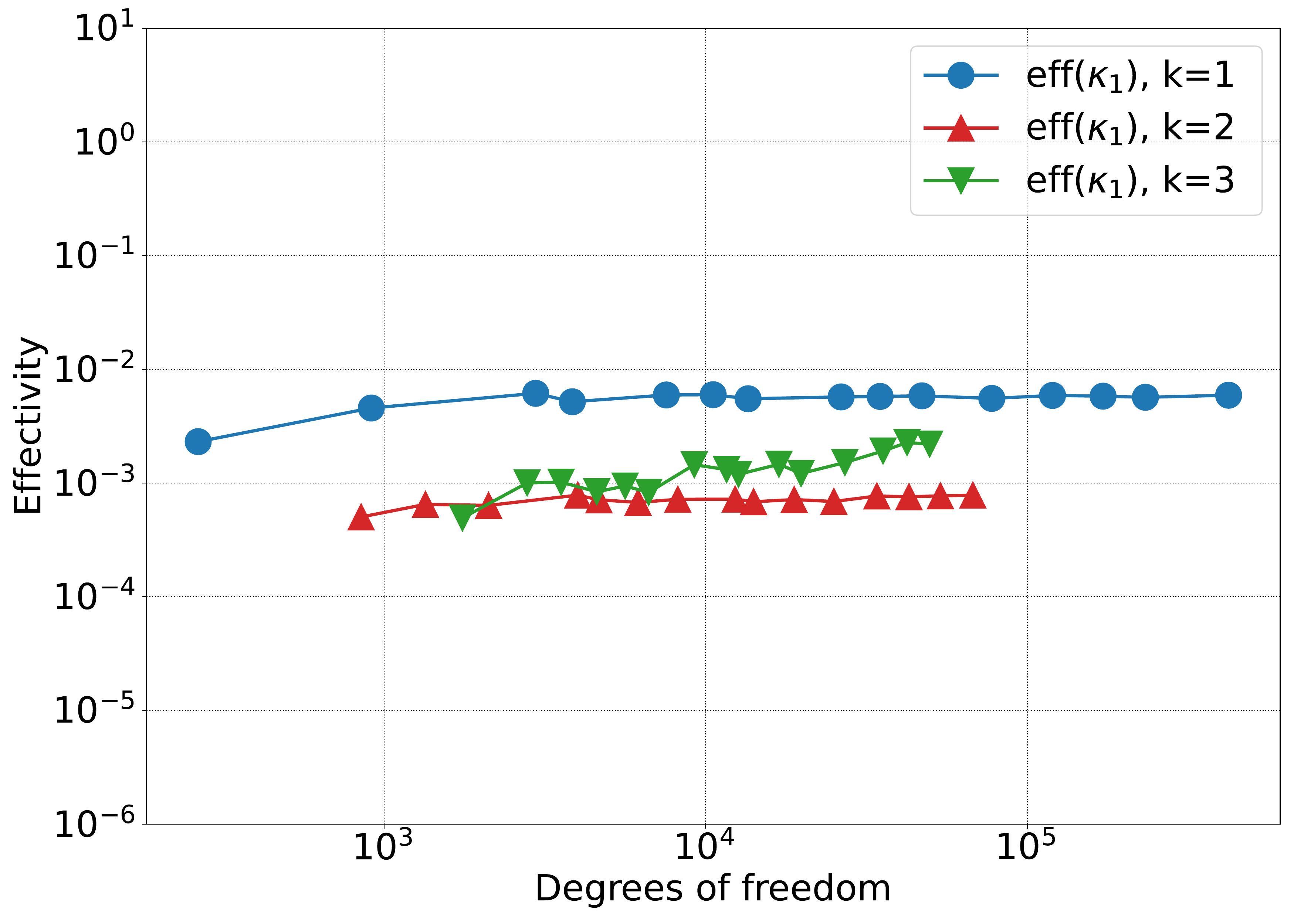}\\
	\end{minipage}\\
	\caption{Test 3. Estimator convergence and effectivity indexes for different values of $k$, with $\nu=0.35$.}
	\label{fig:adaptive-square-with-hole-eff-etas-035}
\end{figure}
\begin{figure}[!h]
	\centering
	\begin{minipage}{0.495\linewidth}\centering
		\includegraphics[scale=0.17, trim=40cm 0cm 40cm 0,clip]{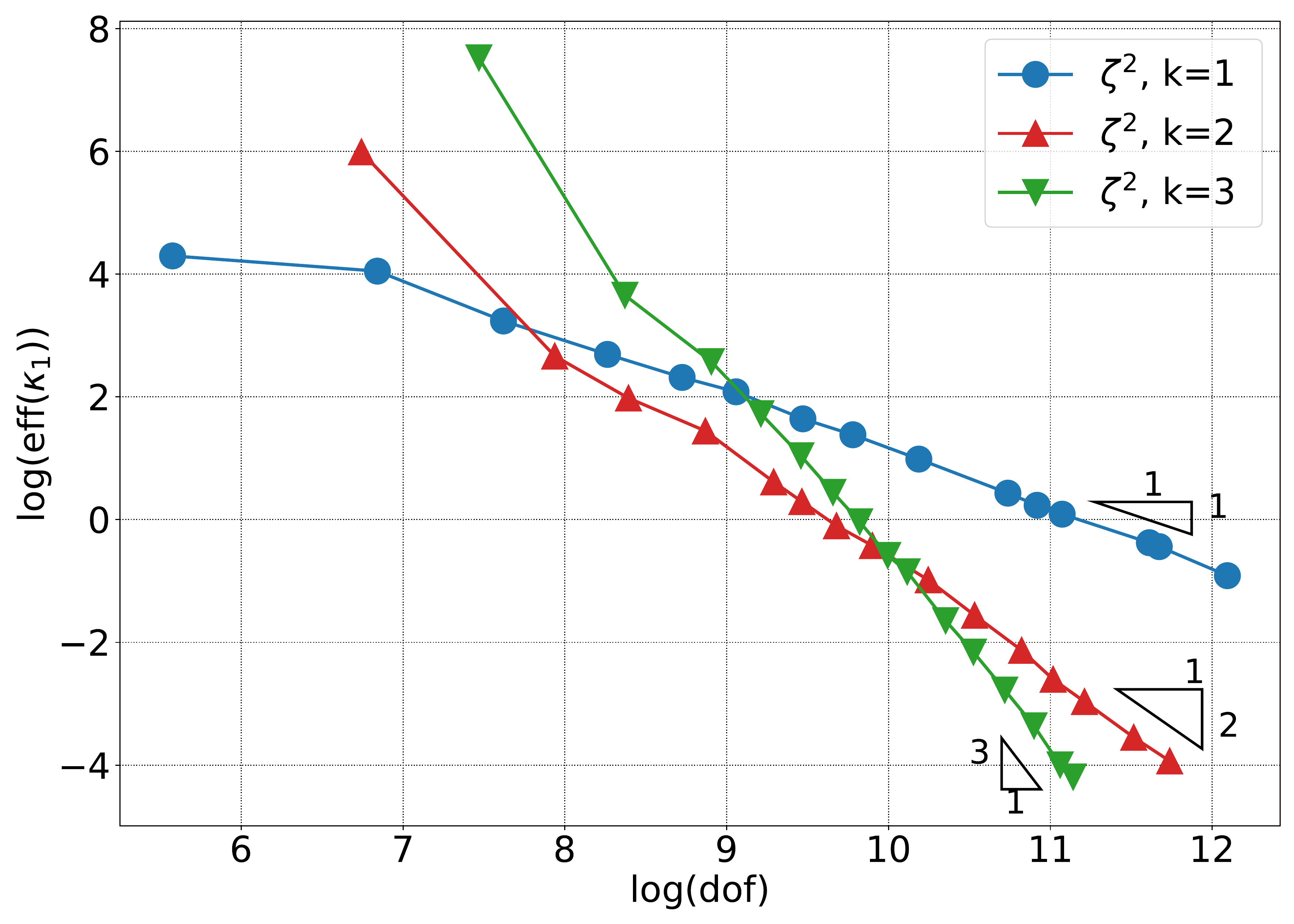}\\
	\end{minipage}
	\begin{minipage}{0.495\linewidth}\centering
		\includegraphics[scale=0.17, trim=40cm 0cm 40cm 0,clip]{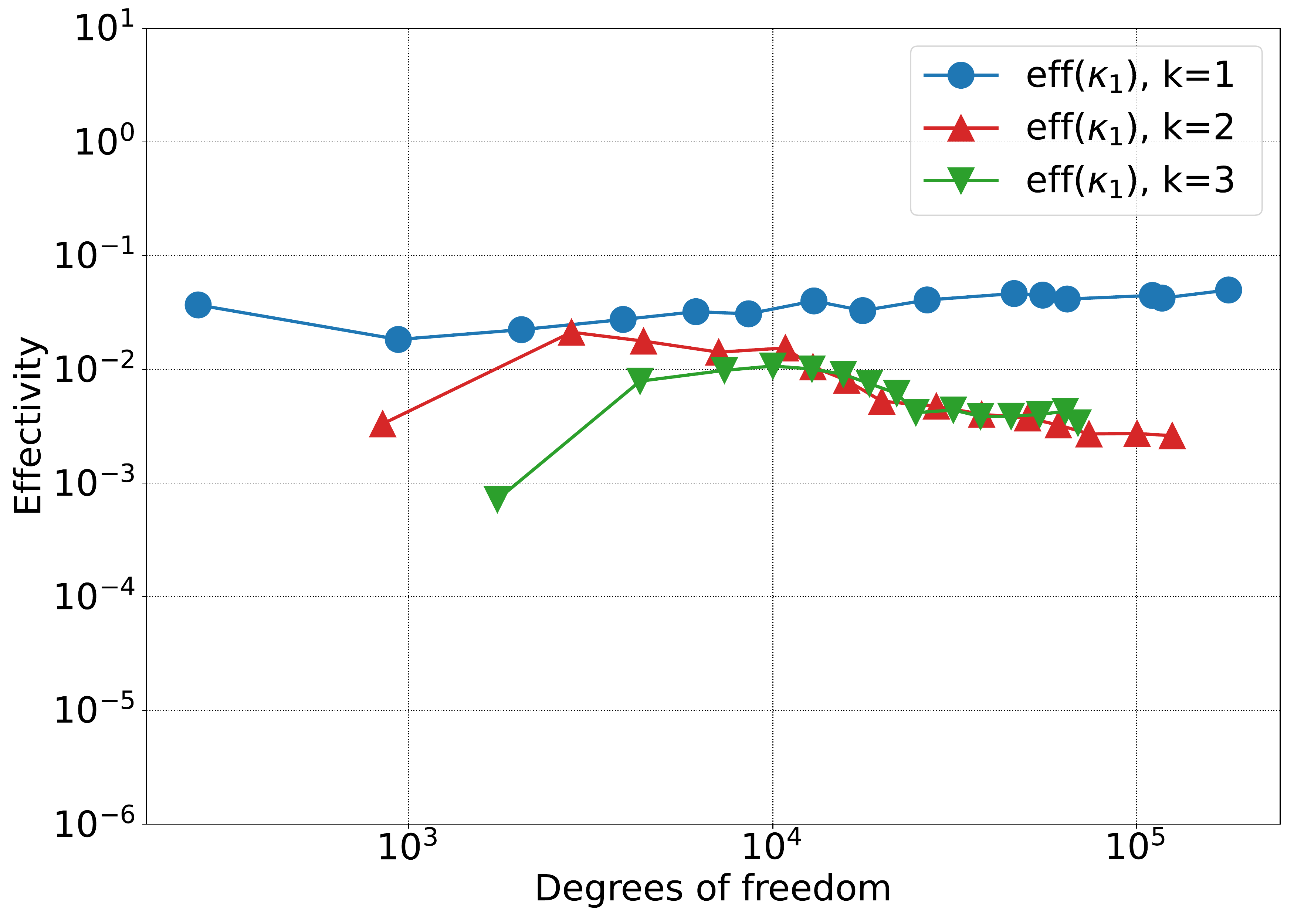}\\
	\end{minipage}\\
	\caption{Test 3. Estimator convergence and effectivity indexes for different values of $k$, with $\nu=0.4999$.}
	\label{fig:adaptive-square-with-hole-eff-etas-04999}
\end{figure}
\begin{figure}[!h]
	\centering
	\begin{minipage}{0.49\linewidth}\centering
		\includegraphics[scale=0.05, trim=40cm 0cm 40cm 0,clip]{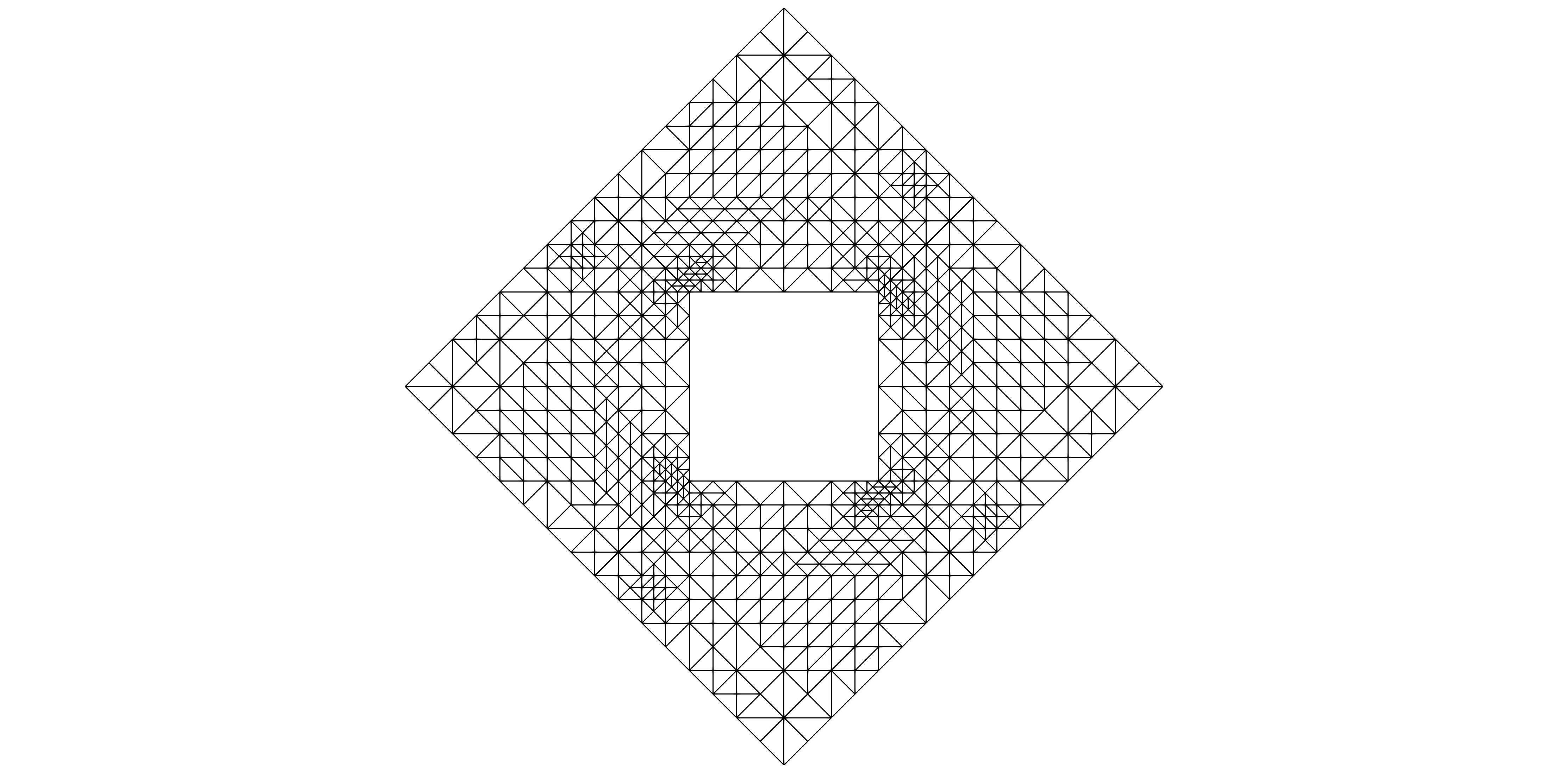}\\
		{\footnotesize 4-th iteration}
	\end{minipage}
	\begin{minipage}{0.49\linewidth}\centering
		\includegraphics[scale=0.05, trim=40cm 0cm 40cm 0,clip]{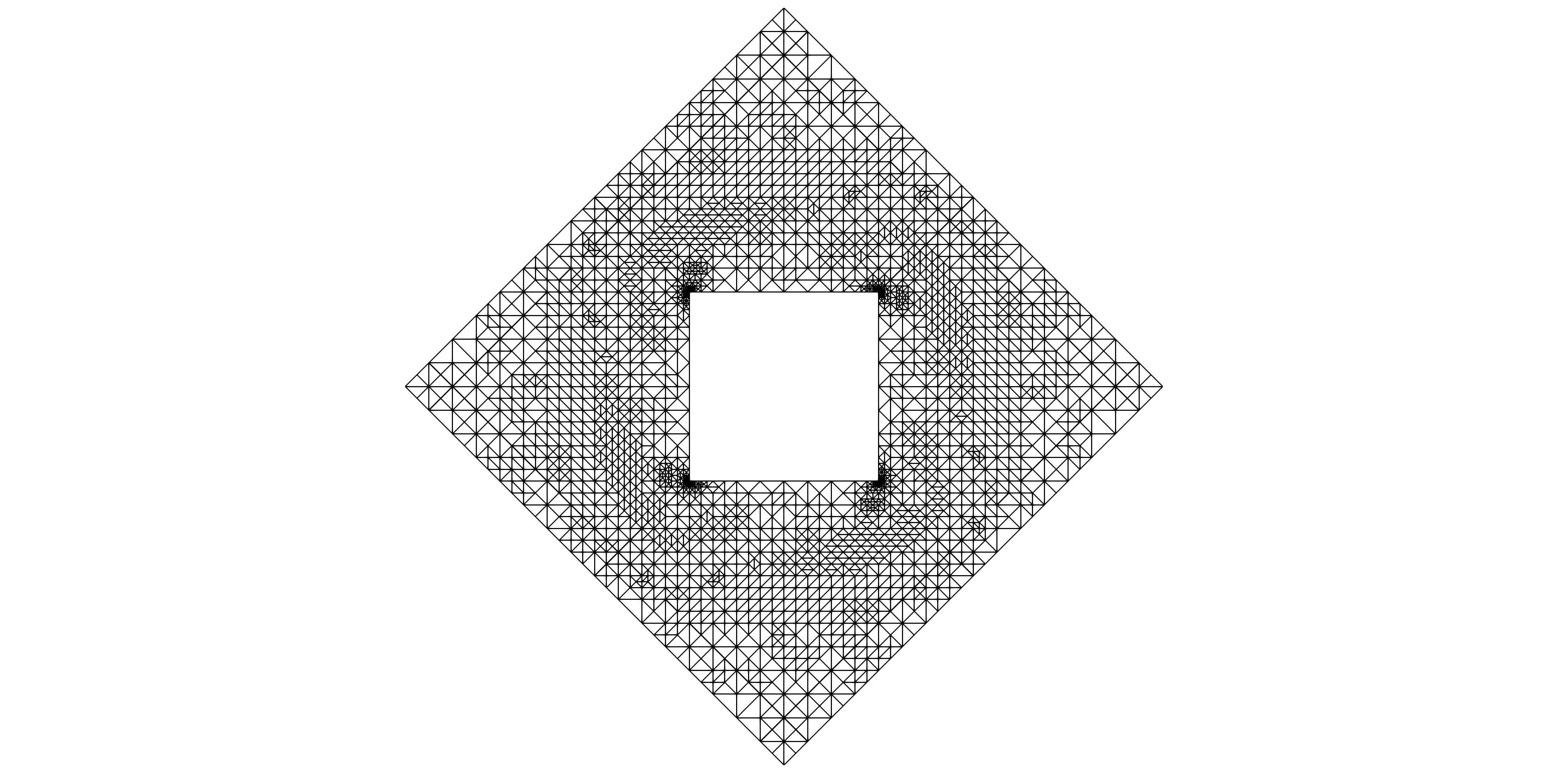}\\
		{\footnotesize 8-th iteration}
	\end{minipage}\\
	\begin{minipage}{0.49\linewidth}\centering
		\includegraphics[scale=0.05, trim=40cm 0cm 40cm 0,clip]{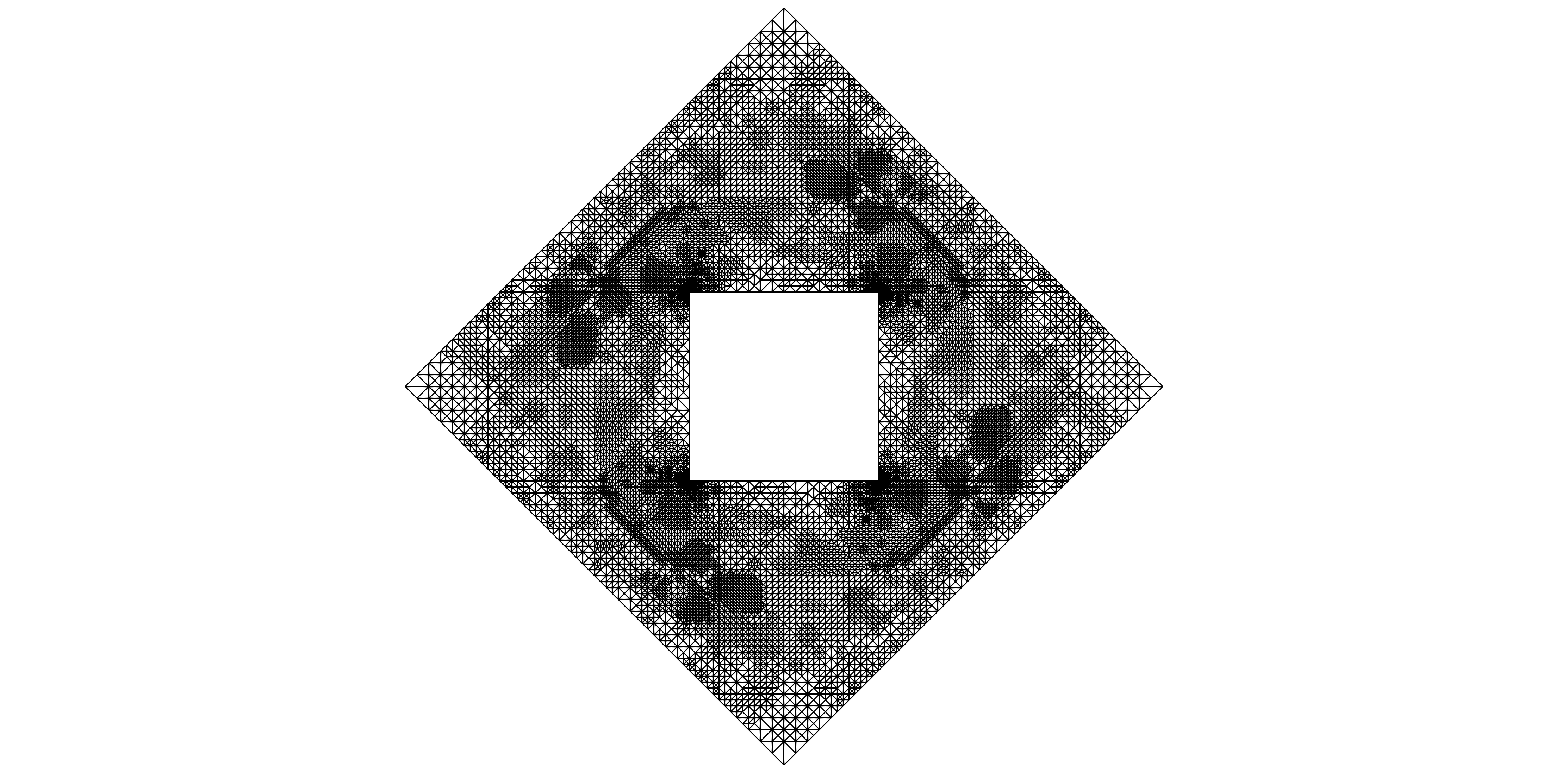}\\
		{\footnotesize 11-th iteration}
	\end{minipage}
	\begin{minipage}{0.49\linewidth}\centering
		\includegraphics[scale=0.05, trim=40cm 0cm 40cm 0,clip]{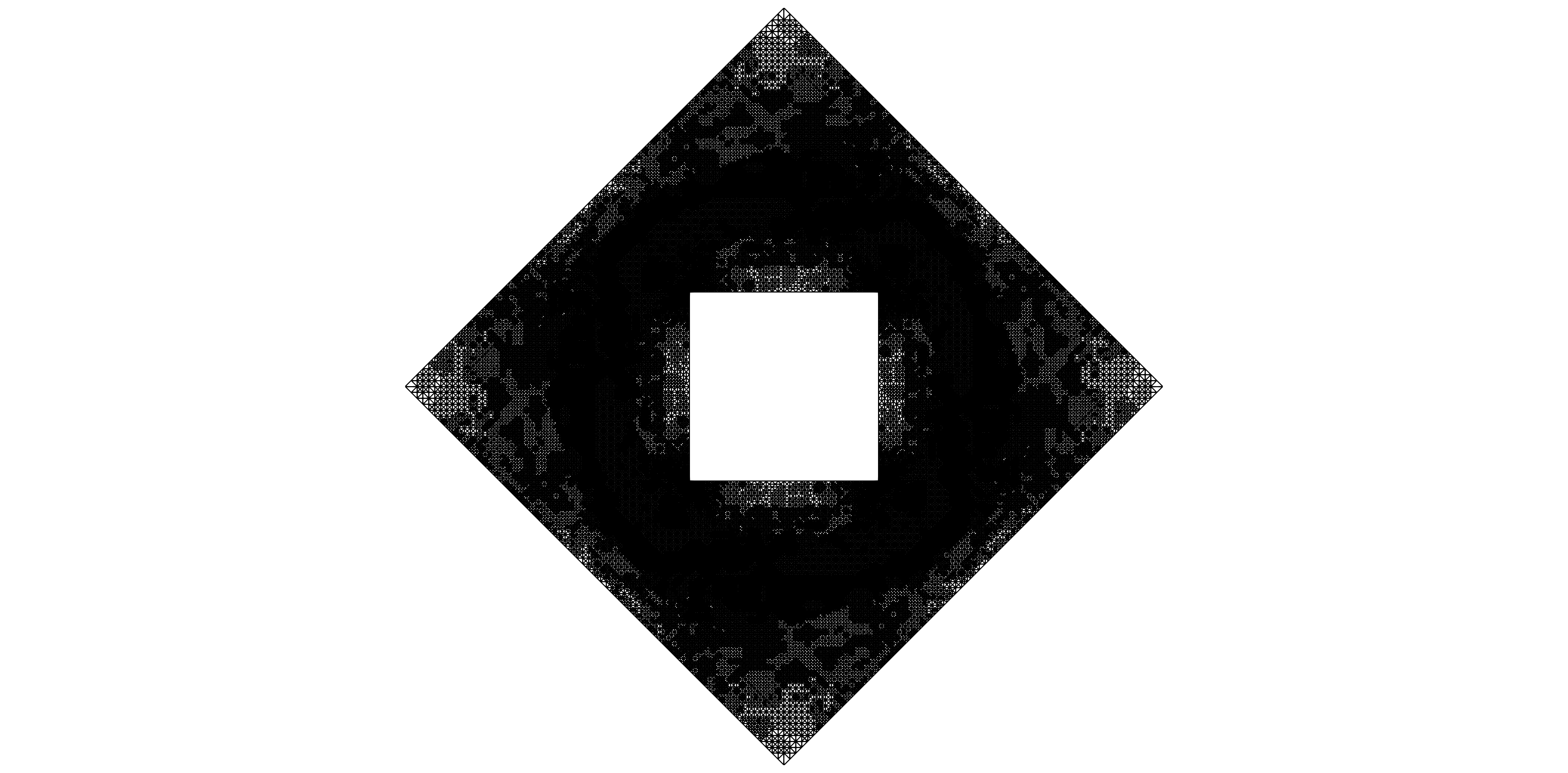}\\
		{\footnotesize 15-th iteration}
	\end{minipage}\\
	\caption{Test 3. Intermediate meshes in the adaptive refinement algorithm for $k=1$ and $\nu=0.35$.}
	\label{fig:adaptive-square-with-hole-nu035}
\end{figure}
\begin{figure}[!h]
	\centering
	\begin{minipage}{0.49\linewidth}\centering
		\includegraphics[scale=0.05, trim=40cm 0cm 40cm 0,clip]{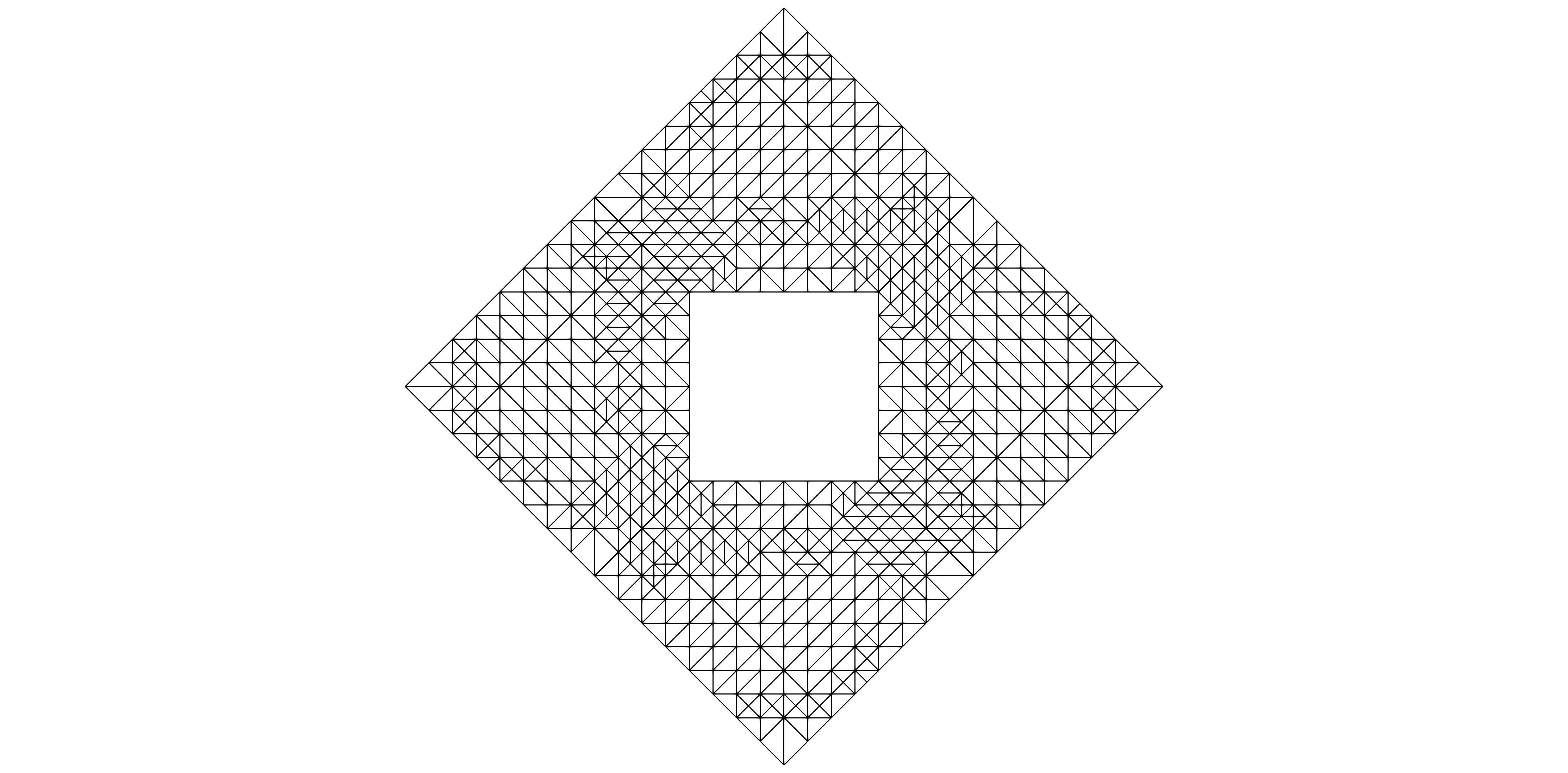}\\
		{\footnotesize 4-th iteration}
	\end{minipage}
	\begin{minipage}{0.49\linewidth}\centering
		\includegraphics[scale=0.05, trim=40cm 0cm 40cm 0,clip]{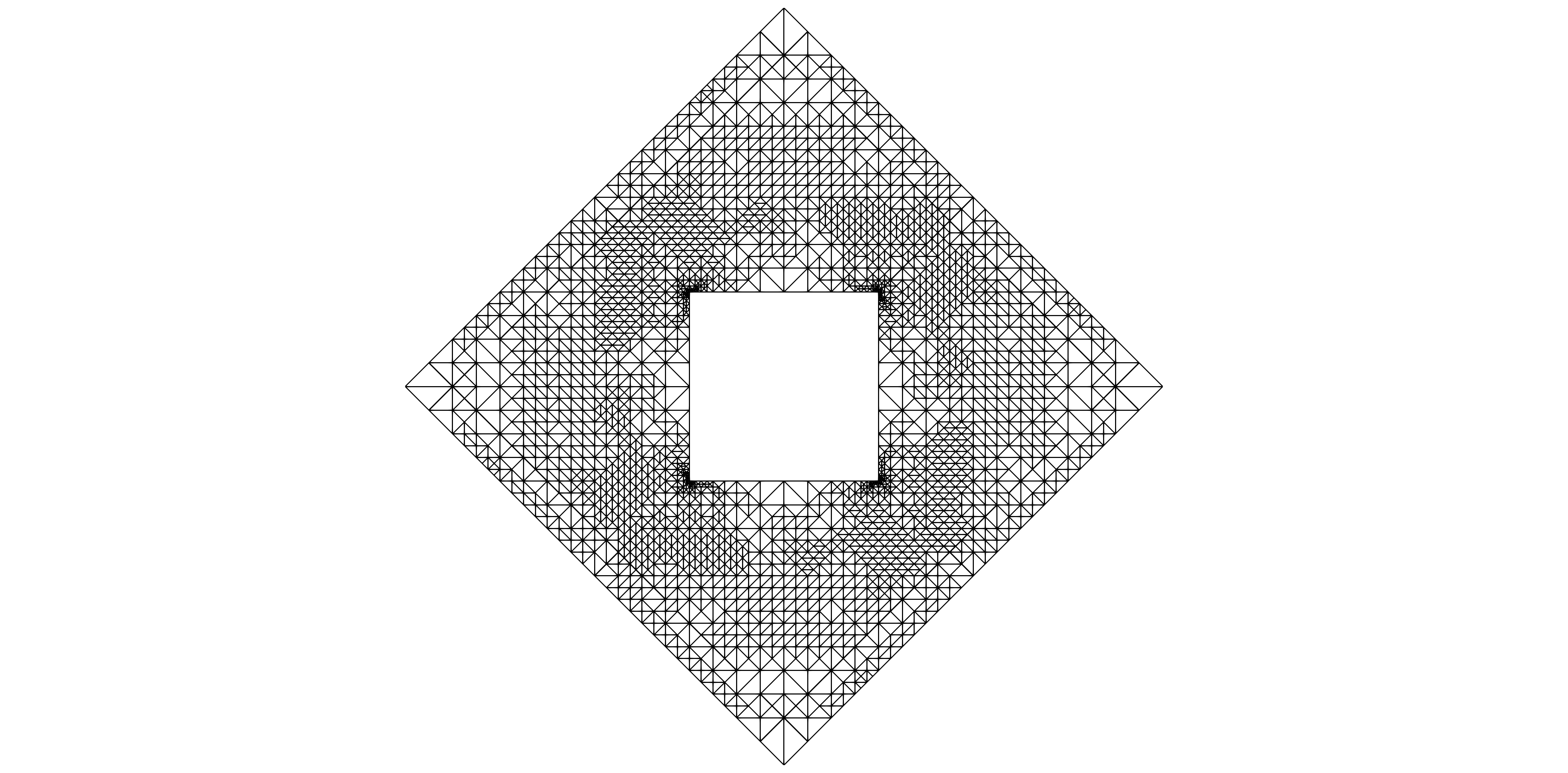}\\
		{\footnotesize 7-th iteration}
	\end{minipage}\\
	\begin{minipage}{0.49\linewidth}\centering
		\includegraphics[scale=0.05, trim=40cm 0cm 40cm 0,clip]{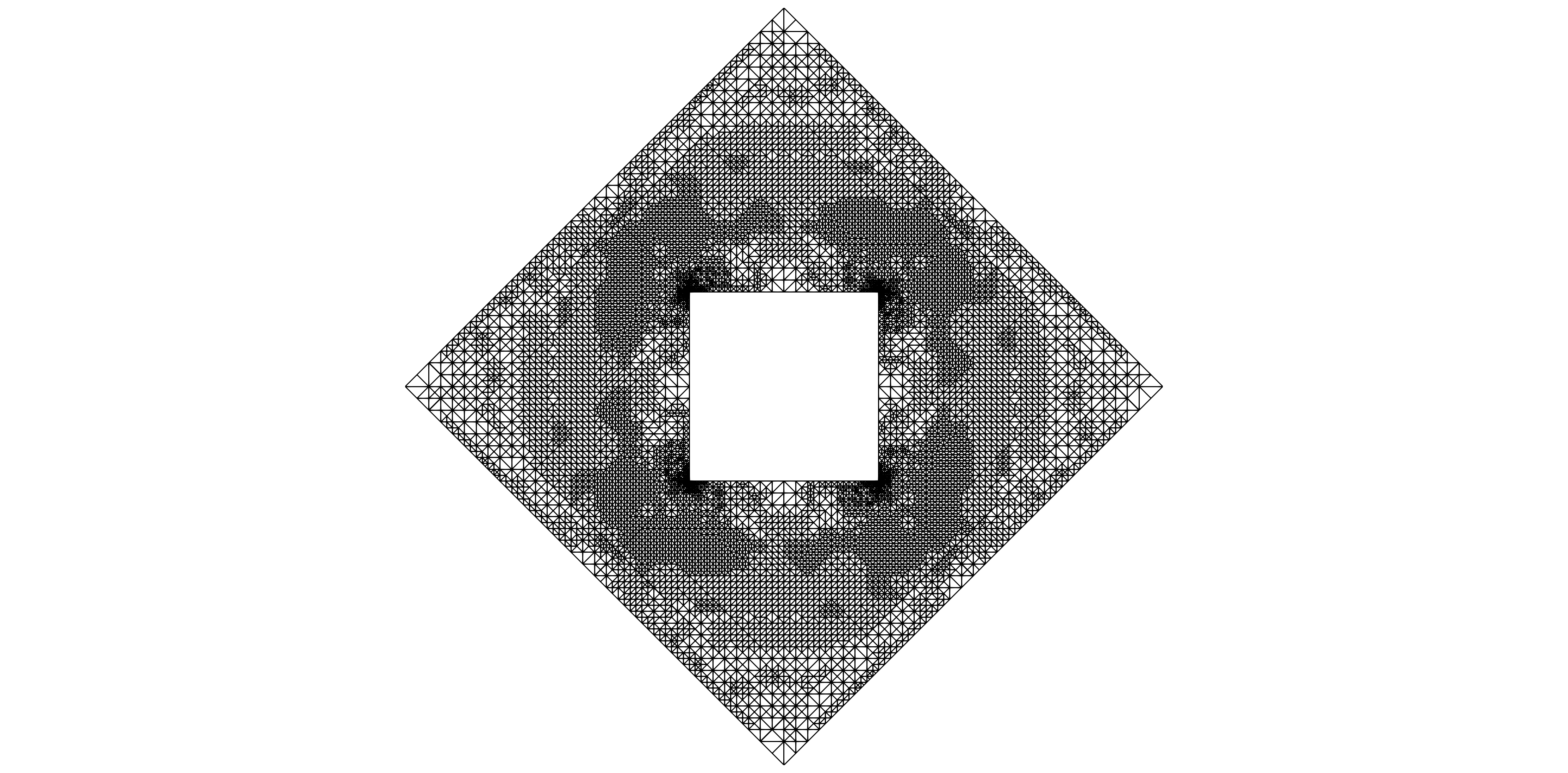}\\
		{\footnotesize 11-th iteration}
	\end{minipage}
	\begin{minipage}{0.49\linewidth}\centering
		\includegraphics[scale=0.05, trim=40cm 0cm 40cm 0,clip]{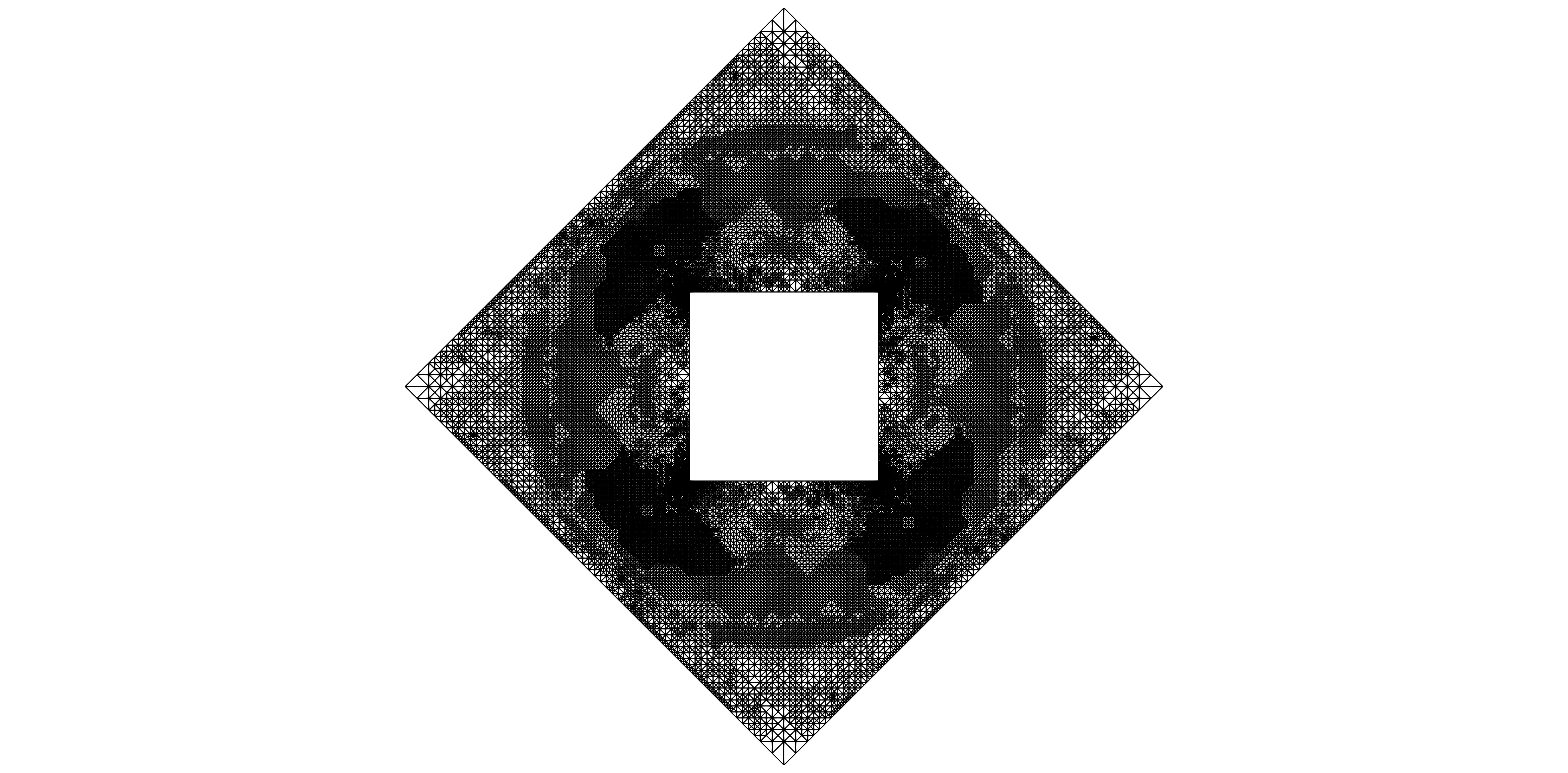}\\
		{\footnotesize 15-th iteration}
	\end{minipage}\\
	\caption{Test 3. Intermediate meshes in the adaptive refinement algorithm for $k=1$ and $\nu=0.4999$.}
	\label{fig:adaptive-square-with-hole-nu04999}
\end{figure}
\begin{figure}[!h]
	\centering
	\begin{minipage}{0.49\linewidth}\centering
		\includegraphics[scale=0.05, trim=40cm 0cm 40cm 0,clip]{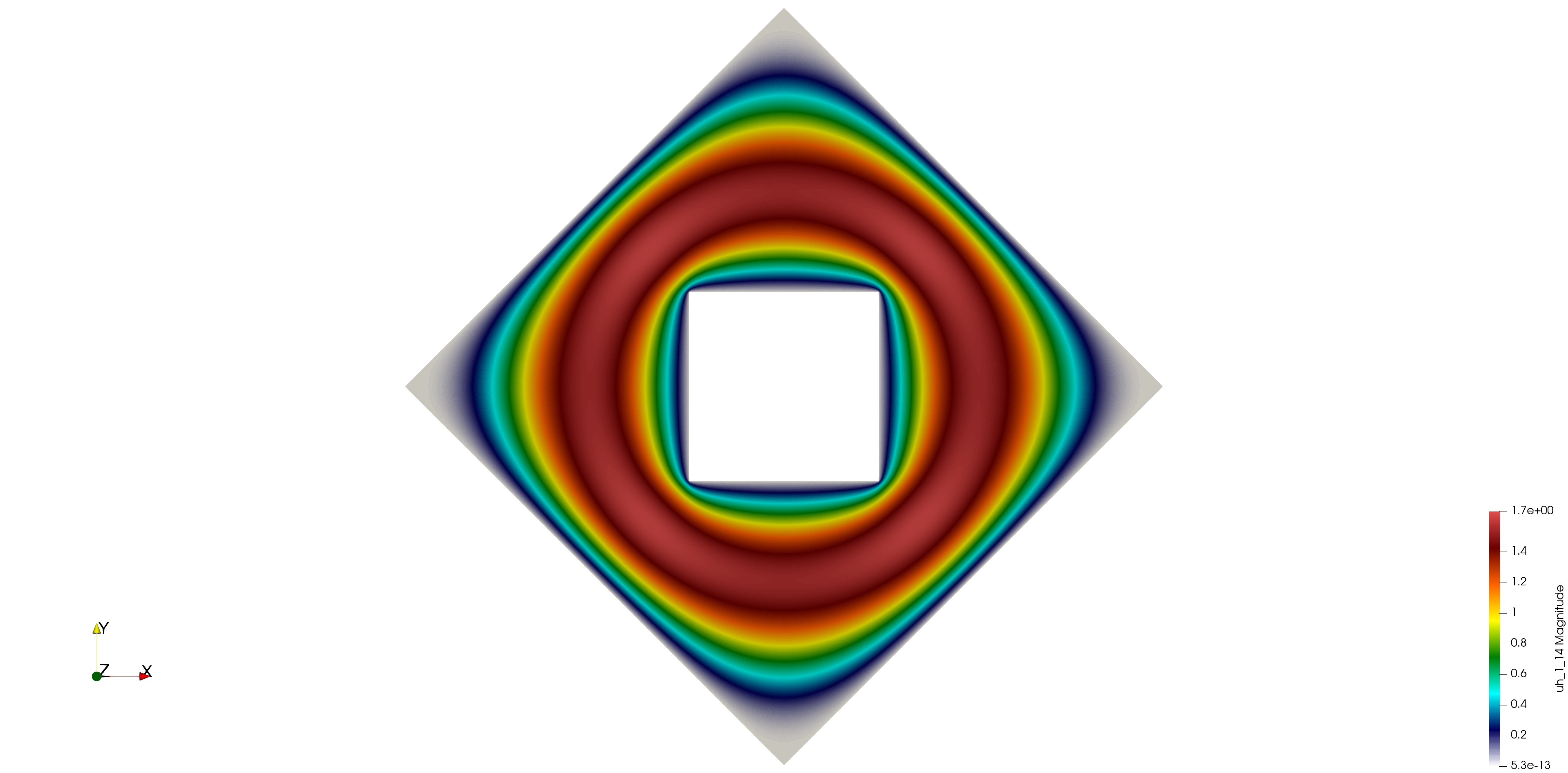}\\
		{\footnotesize $|\bu_h|,\nu=0.35$}
	\end{minipage}
	\begin{minipage}{0.49\linewidth}\centering
		\includegraphics[scale=0.05, trim=40cm 0cm 40cm 0,clip]{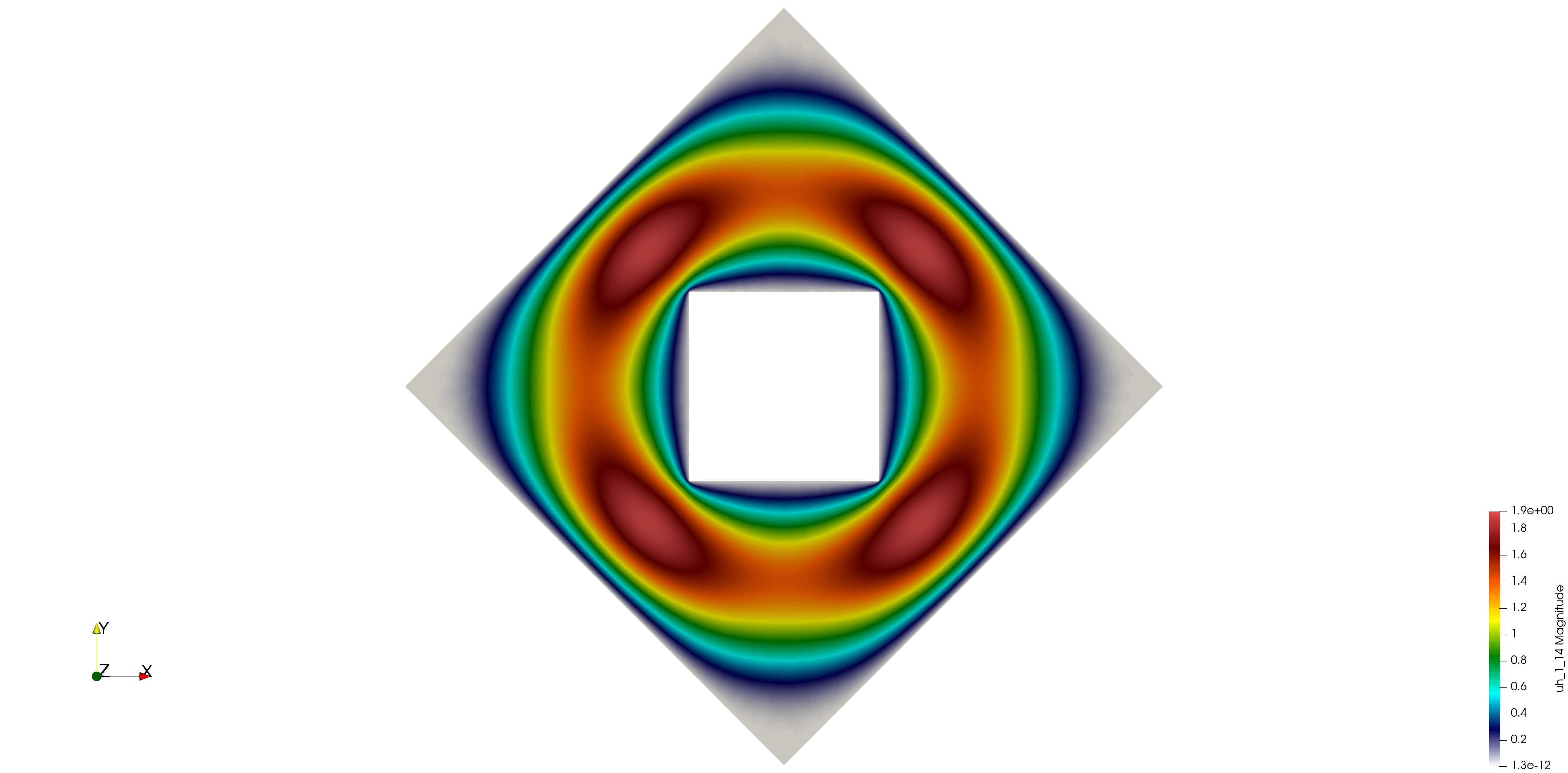}\\
		{\footnotesize $|\bu_h|,\nu=0.4999$}
	\end{minipage}\\
	\begin{minipage}{0.49\linewidth}\centering
		\includegraphics[scale=0.05, trim=40cm 0cm 40cm 0,clip]{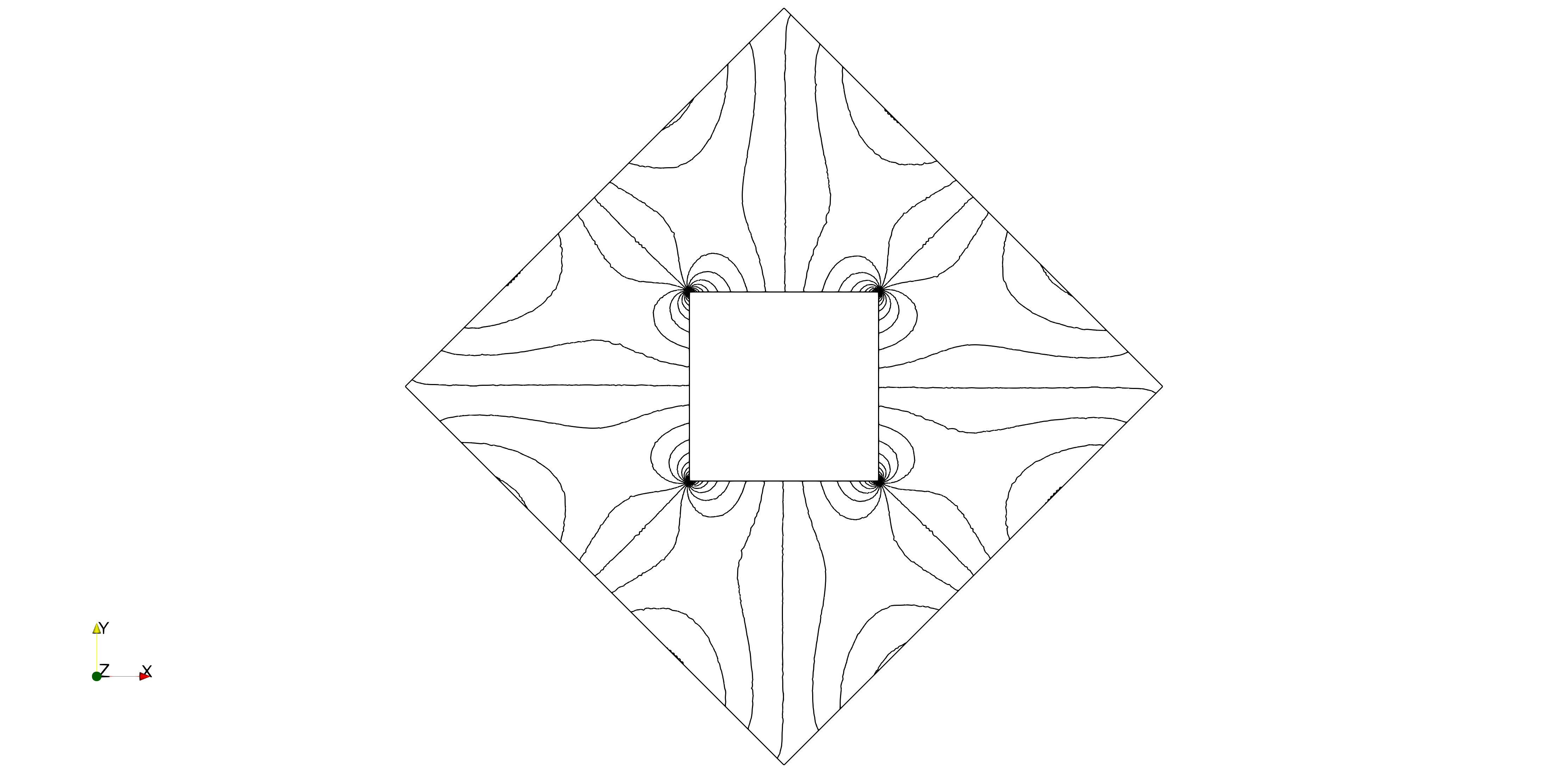}\\
		{\footnotesize Contour $p_h$}
	\end{minipage}
	\begin{minipage}{0.49\linewidth}\centering
		\includegraphics[scale=0.05, trim=40cm 0cm 40cm 0,clip]{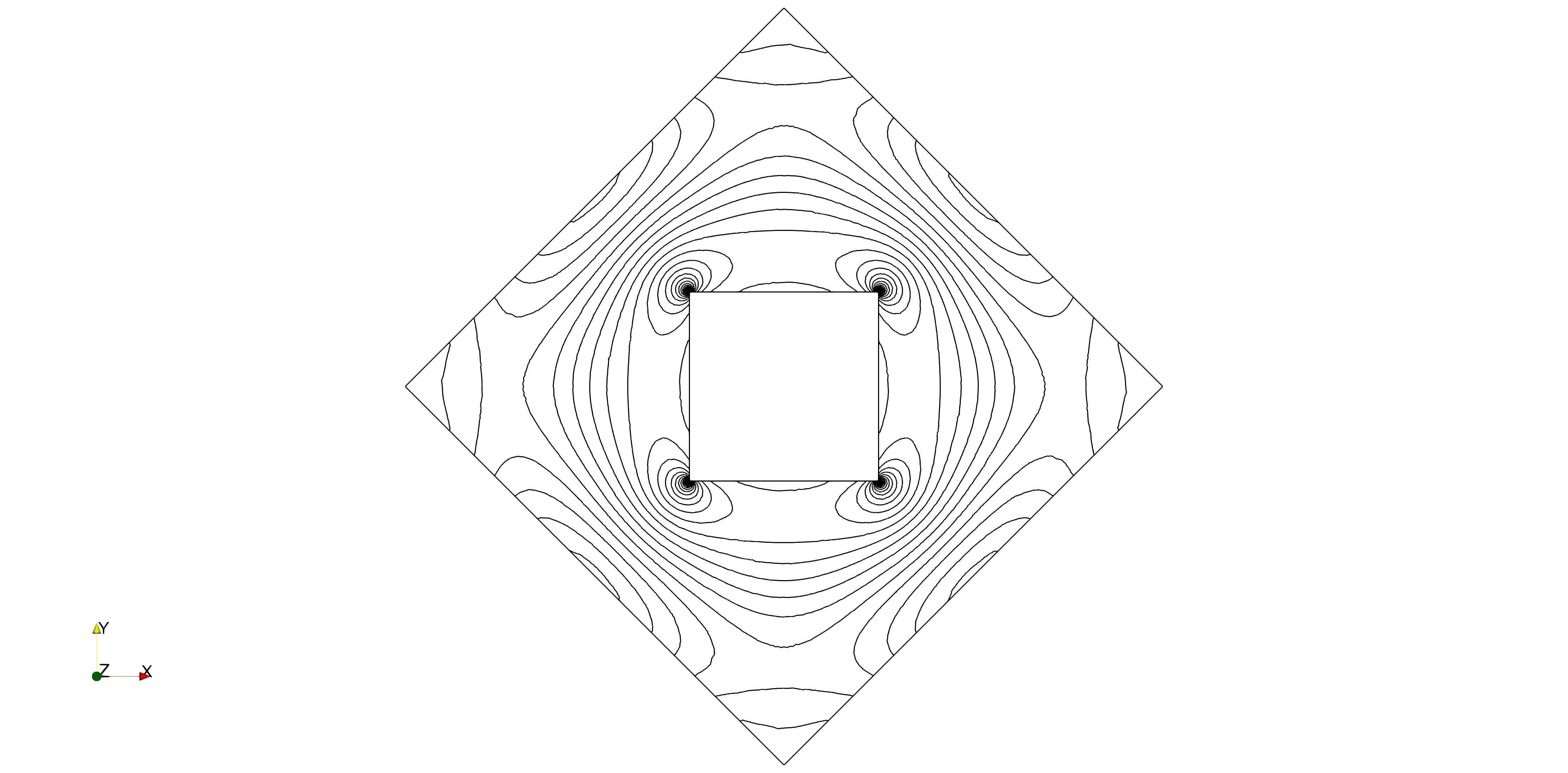}\\
		{\footnotesize Contour $\boldsymbol{\omega}_h$}
	\end{minipage}\\
	\begin{minipage}{0.49\linewidth}\centering
	\includegraphics[scale=0.05, trim=20cm 5cm 23cm 0,clip]{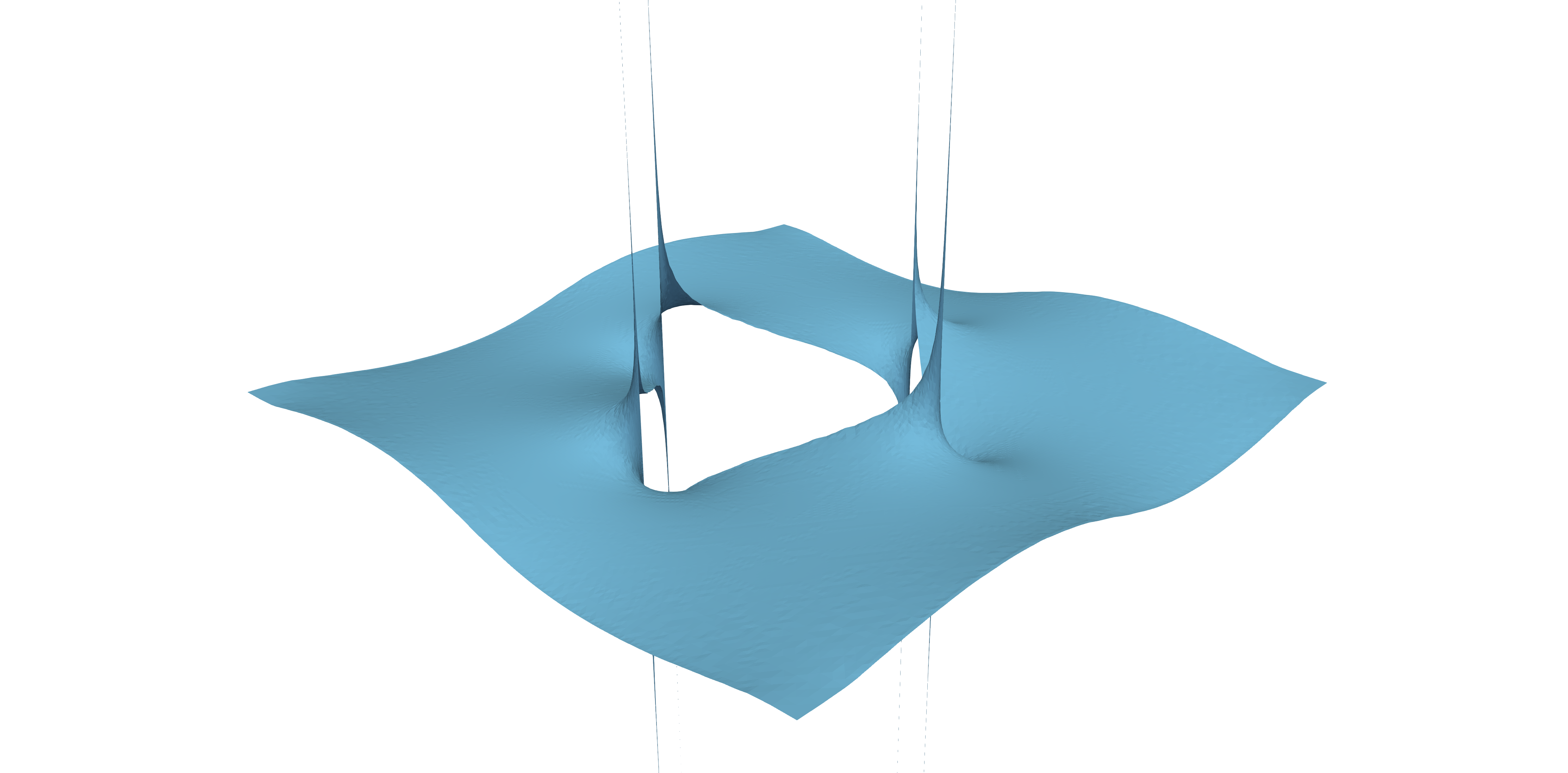}\\
	{\footnotesize Warped $p_h$}
	\end{minipage}
	\begin{minipage}{0.49\linewidth}\centering
		\includegraphics[scale=0.05, trim=20cm 5cm 23cm 0,clip]{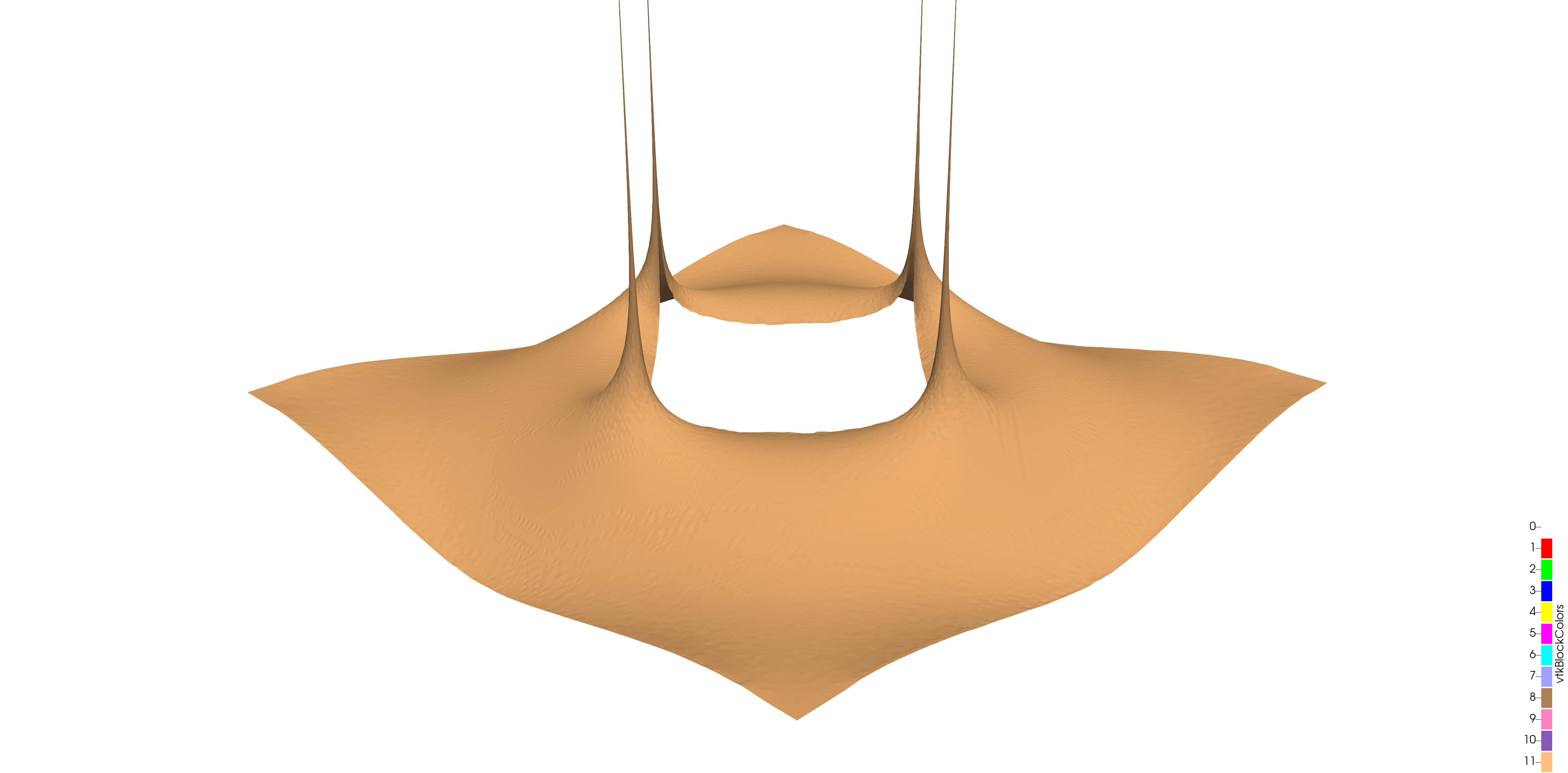}\\
		{\footnotesize Warped $\boldsymbol{\omega}_h$}
	\end{minipage}\\
	\caption{Test 3. Lowest computed eigenmodes for $k=1$ and different values of $\nu$. The high gradients of $p_h$ and high rotations $\boldsymbol{\omega}_h$ near the singularity for $\nu=0.35$ and $\nu=0.4999$ are similar.}
	\label{fig:adaptive-square-with-hole-eigenmodes}
\end{figure}

\subsubsection{Test 4. A 3D L-shaped domain} In this experiment we consider the classical 3D L-shape domain, which is described as
$$
\Omega:=(-1/2,1/2)\times(0,1)\times(-1/2,1/2)\backslash\big((0,1/2)\times(0,1)\times(0,1/2)\big).
$$
This domain has the characteristic of having a singularity along the line $(0,y,0)$, for $y\in[0,1]$, so the convergence with uniform meshes will be suboptimal. The initial mesh for this experiment is depicted in Figure \ref{fig:L-shape-initial-mesh}. For the limiting case, we consider $\alpha^{-1}=1/2$.
\begin{figure}[!h]
	\centering
	\includegraphics[scale=0.06,trim= 40cm 5cm 45cm 3.5cm,, clip]{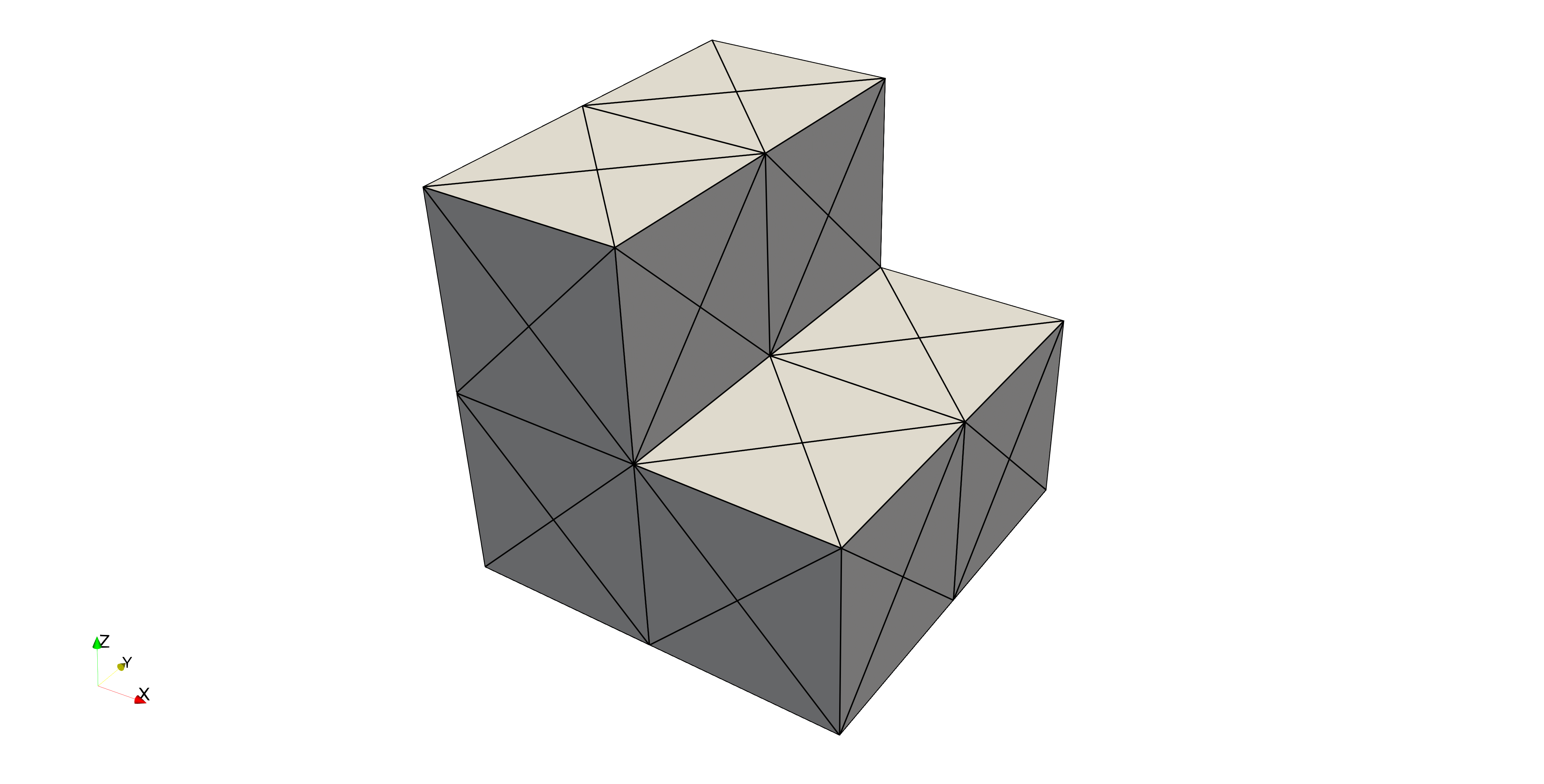}
	\caption{Test 4. Initial shape for the 3D L-shaped domain.}
	\label{fig:L-shape-initial-mesh}
\end{figure}

In Table \ref{tabla:lshape-GG-3D-uniform-vs-adaptive} we observe the behavior of the adaptive scheme for the values of $\nu$ at the lowest polynomial order. The use of uniform meshes for this case results in a experimental convergence rate between $\mathcal{O}(\texttt{dof}^{-\sqrt{3}/3})\approx\mathcal{O}(h^{1.7})$ and $\mathcal{O}(\texttt{dof}^{-2/3})\approx\mathcal{O}(h^2)$, which is the best expected order for this type of refinement. The adaptive scheme in the compressible and nearly incompressible case exhibits a behavior asymptotically similar to $\mathcal{O}(h^2)$, which is the optimal order according to the theory. In Figure \ref{fig:adaptive-lshape-3D-errores-eficiencia} we can observe the curve of these errors and the effectivity of our estimator, proving to be reliable and efficient.  To be more precise, in Table \ref{tabla:lshape-GG-3D-035-vs-04999} we describe the results obtained with different values of $\nu$, where we can observe that the errors and the estimator behave like  $\mathcal{O}(h^2)$. 

On the other hand, in Figure \ref{fig:adaptive-lshape-3D-mallas} we show intermediate meshes in our adaptive iterations. Note that the refinement is concentrated in the zone $(0,y,0)$, therefore the estimator is able to detect and refine close to the singularity. To conclude, we show in Figure \ref{fig:adaptive-lshape-3D-eigenmodes} the first eigenmode for each value of $\nu$. The vector fields $\bu_h$ and $\bw_h$ together with their magnitudes $|\bu_h|$ and $|\bw_h|$, respectively, are presented in the same mesh. The pressure is represented trough a surface contour plot. We note the high pressure gradients and high rotations near the singularity, as is expected.

\begin{table}[!h]\centering
	{\footnotesize\setlength{\tabcolsep}{5pt}
		\caption{Test 4: Comparison between the lowest computed eigenvalue for $k=1$, $\nu=0.35$ and $\nu=0.4999$ with uniform and adaptive refinements.}
		\label{tabla:lshape-GG-3D-uniform-vs-adaptive}
		\begin{center}
			\begin{tabular}{l c l c  |l c l c }
				\toprule
				\multicolumn{4}{c}{$\nu=0.35$} & \multicolumn{4}{|c}{$\nu=0.4999$}\\
				\midrule
				\multicolumn{2}{c}{Uniform} &\multicolumn{2}{c}{Adaptive}& \multicolumn{2}{|c}{Uniform} & \multicolumn{2}{c}{Adaptive}\\
				\midrule
				\texttt{dof}&$\sqrt{\kappa_{h1}}$ & \texttt{dof}&$\sqrt{\kappa_{h1}}$&\texttt{dof}&$\sqrt{\kappa_{h1}}$ & \texttt{dof}&$\sqrt{\kappa_{h1}}$ \\ 
				\midrule
				830		& 7.69255 &830		& 7.69255 &830		& 6.12221 &830		& 6.12221  \\
				6211	& 5.76928 &2930		& 5.97646 &6211		& 5.78695 &4645	    & 5.77952   \\
				48347	& 5.25662 &10132	& 5.49598 &48347	& 5.41616 &23289	& 5.45506 \\
				382099	& 5.10803 &30913	& 5.28891 &382099	& 5.28658 &88486	& 5.34401 \\
				3039395 & 5.06572 &68740	& 5.19462 &3039395	& 5.25196 &138855	& 5.31763 \\
						&		  &129778	& 5.14555 &			&		  &461939	& 5.27548  \\
						& 		  &275240	& 5.10640 &			& 		  &679613   & 5.26666 \\
						&  		  &550531	& 5.08668 &			& 	      &1690232	& 5.25403 \\
						&  		  &882797	& 5.07516 &			& 		  &2828623  & 5.24948 \\
						&  		  &1849627	& 5.06375 &			& 		  &  		&			\\
						&  		  &1856515	& 5.06367 &			& 		  &  		&			\\
						&  		  &3753754	& 5.05753 &			& 		  &  		&			\\
				\midrule
				Order	&$\mathcal{O}(N^{-0.61})$		&Order	&$\mathcal{O}(N^{-0.66})$& Order	&$\mathcal{O}(N^{-0.62})$		&Order	&$\mathcal{O}(N^{-0.66})$ \\
				$\sqrt{\kappa_1}$	& 5.04874 		&$\sqrt{\kappa_1}$	& 5.04874 & $\sqrt{\kappa_1}$	& 5.24009 		&$\sqrt{\kappa_1}$	& 5.24009 \\
				\bottomrule             
			\end{tabular}
	\end{center}}
\end{table}
\begin{table}[!h]
	{\footnotesize
		\caption{Test 4: Computed errors and effectivity indexes on the adaptively refinement meshes for $k=1$, $\nu=0.35$ and $\nu=0.4999$.}
		\label{tabla:lshape-GG-3D-035-vs-04999}
		\begin{center}
			\begin{tabular}{c c c c|c c c c }
				\toprule
				\multicolumn{3}{c}{$\nu=0.35$} &&& \multicolumn{3}{c}{$\nu=0.4999$}\\
				\midrule
				$\err(\kappa_{h1})$&$\zeta^2$&$\eff(\kappa_{h1})$ &&& $\err(\kappa_{h1})$&$\zeta^2$&$\eff(\kappa_{h1})$ \\ 
				\midrule
				2.64380e+00	&2.33470e+03& 1.13239e-03 &&&8.82122e-01	&1.04890e+03& 8.40995e-04  \\
				9.27714e-01	&4.57985e+02& 2.02564e-03 &&&5.39432e-01	&3.22277e+02& 1.67381e-03  \\
				4.47229e-01	&1.48726e+02& 3.00707e-03 &&&2.14972e-01	&8.02553e+01& 2.67860e-03 \\
				2.40160e-01	&6.52344e+01& 3.68149e-03 &&&1.03919e-01	&2.97943e+01& 3.48787e-03  \\
				1.45869e-01	&3.69662e+01& 3.94601e-03 &&&7.75365e-02	&2.23126e+01& 3.47501e-03  \\
				9.67976e-02	&2.39336e+01& 4.04443e-03 &&&3.53827e-02	&9.90877e+00& 3.57084e-03  \\
				5.76481e-02	&1.44241e+01& 3.99665e-03 &&&2.65636e-02	&7.54860e+00& 3.51901e-03 \\
				3.79299e-02 &9.09617e+00& 4.16987e-03 &&&1.39398e-02	&4.34538e+00& 3.20796e-03 \\
				2.64066e-02 &6.67283e+00& 3.95733e-03 &&&9.38544e-03    &3.01470e+00& 3.11322e-03 \\
				1.50016e-02 &4.07549e+00& 3.68093e-03 &&& 				&			& 			  \\
				1.49166e-02 &4.06522e+00& 3.66932e-03 &&& 				&			& 			  \\
				8.77733e-03 &2.53519e+00& 3.46220e-03 &&& 				&			& 			  \\
				\bottomrule             
			\end{tabular}
	\end{center}}
\end{table}

\begin{figure}[!h]
	\centering
	\begin{minipage}{0.495\linewidth}\centering
		\includegraphics[scale=0.17, trim=40cm 0cm 40cm 0,clip]{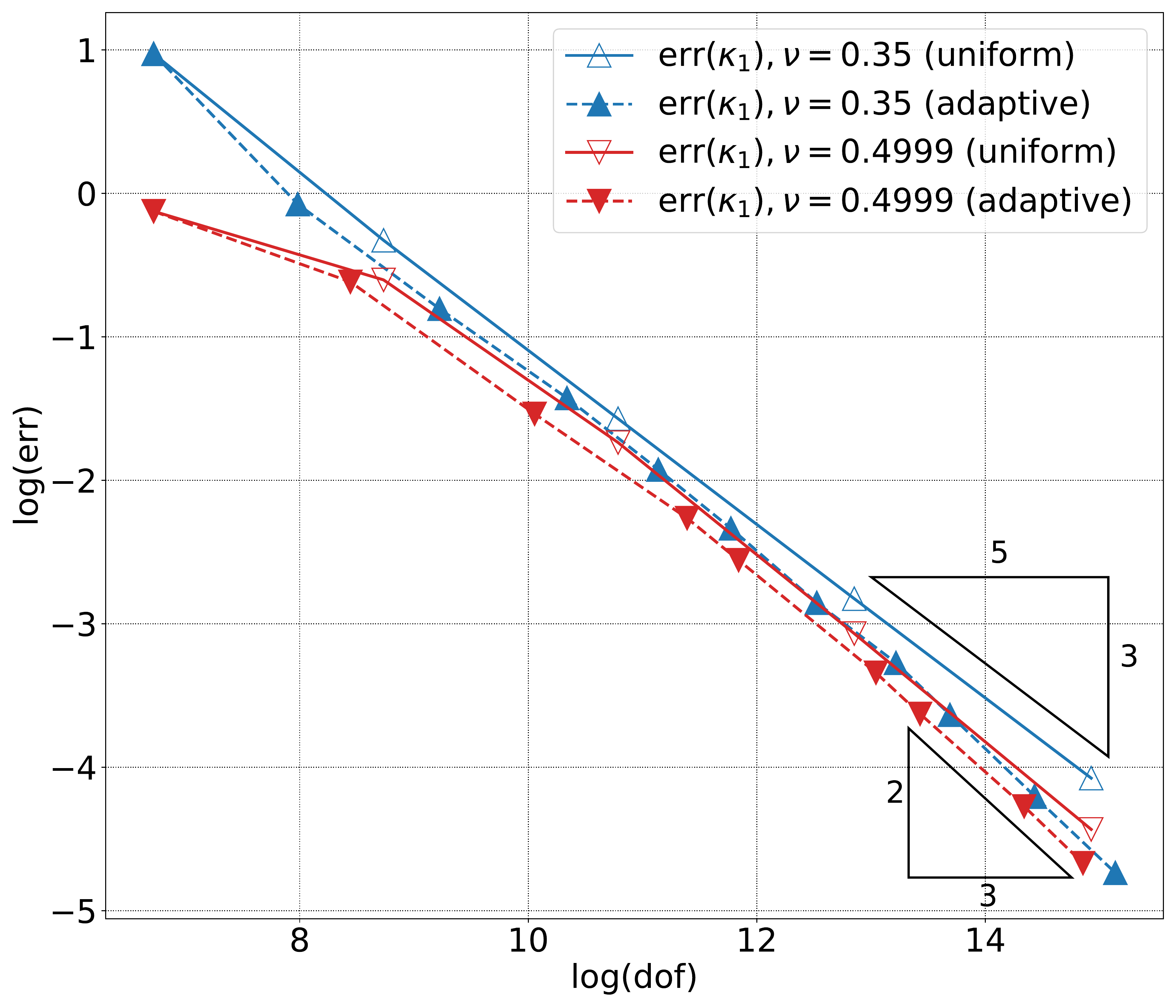}\\
	\end{minipage}
	\begin{minipage}{0.495\linewidth}\centering
		\includegraphics[scale=0.17, trim=40cm 0cm 40cm 0,clip]{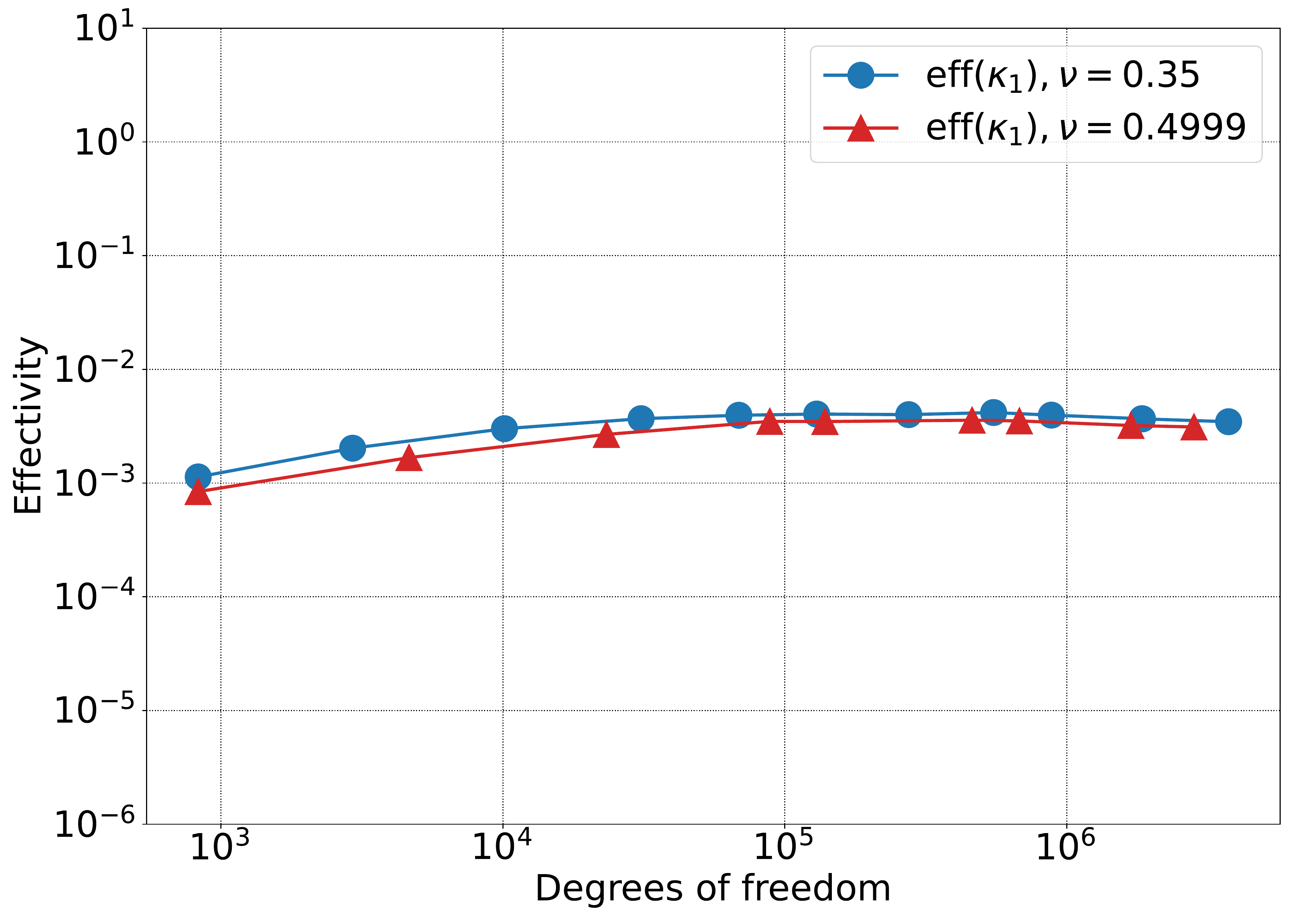}\\
	\end{minipage}\\
	\caption{Test 4. Error and effectivity curves for $k=1$, $\nu=0.35$ and $\nu=0.4999$.}
	\label{fig:adaptive-lshape-3D-errores-eficiencia}
\end{figure}
\begin{figure}[!h]
	\centering
	\begin{minipage}{0.49\linewidth}
		\centering
		\includegraphics[scale=0.06,trim= 40cm 5cm 45cm 3.5cm,, clip]{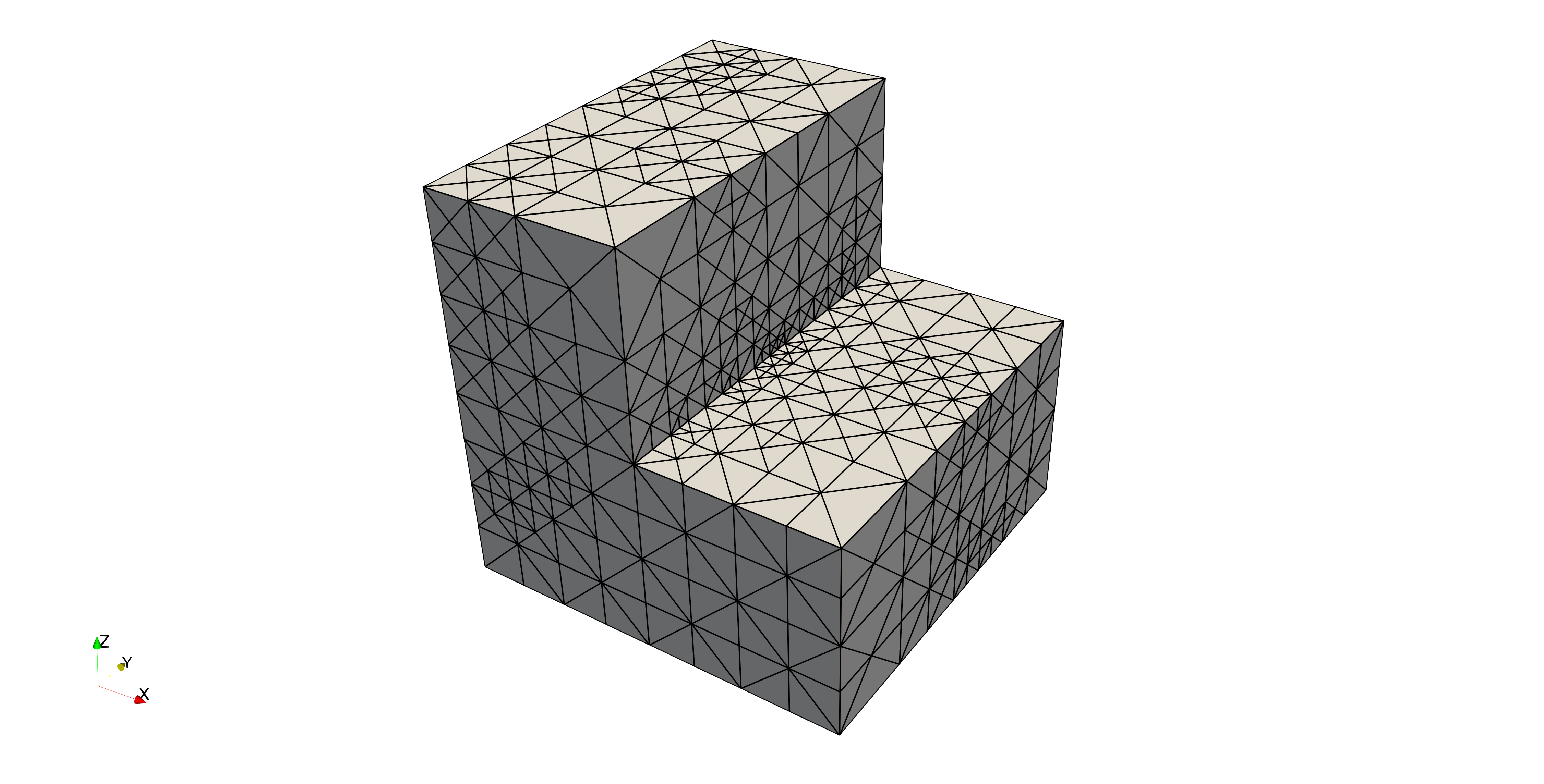}\\
		{\footnotesize $\nu=0.35$, 7-th iteration}
	\end{minipage}
	\begin{minipage}{0.49\linewidth}
		\centering
		\includegraphics[scale=0.06,trim= 40cm 5cm 45cm 3.5cm,, clip]{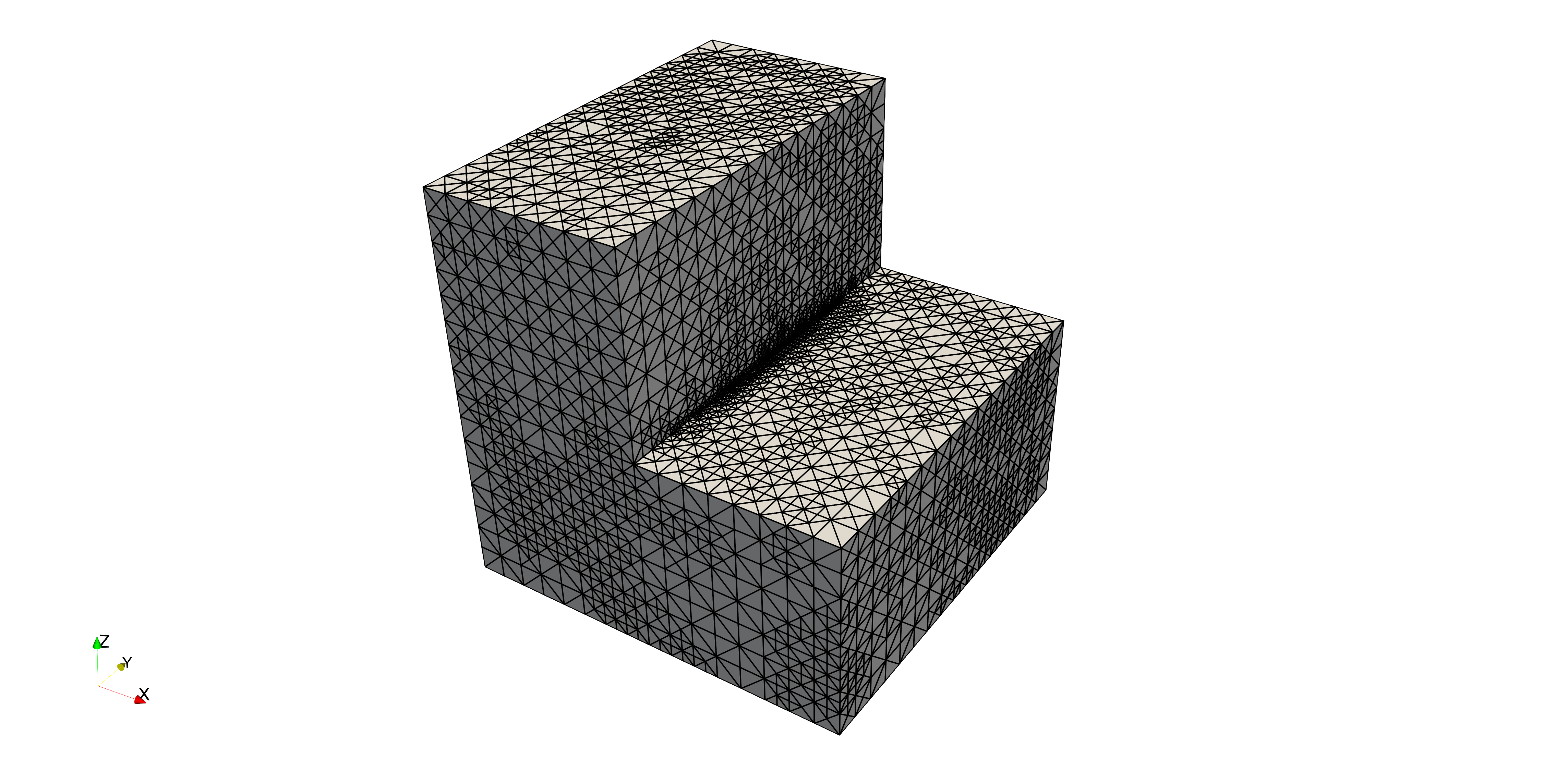}\\
		{\footnotesize $\nu=0.35$, 12-th iteration}
	\end{minipage}\\
	\begin{minipage}{0.49\linewidth}
		\centering
		\includegraphics[scale=0.06,trim= 40cm 5cm 45cm 3.5cm,, clip]{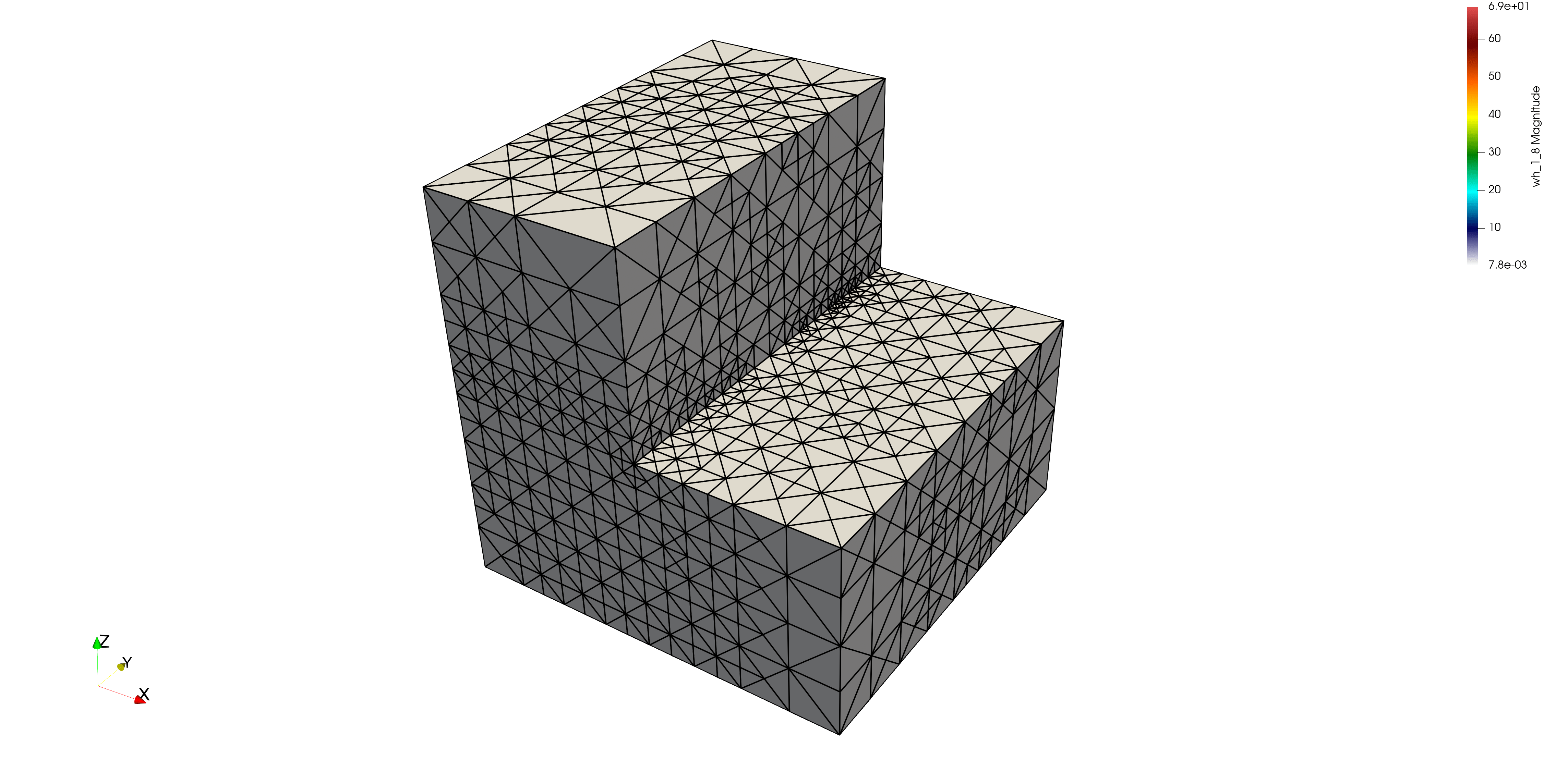}\\
		{\footnotesize $\nu=0.4999$, 7-th iteration}
	\end{minipage}
		\begin{minipage}{0.49\linewidth}
		\centering
		\includegraphics[scale=0.06,trim= 40cm 5cm 45cm 3.5cm,, clip]{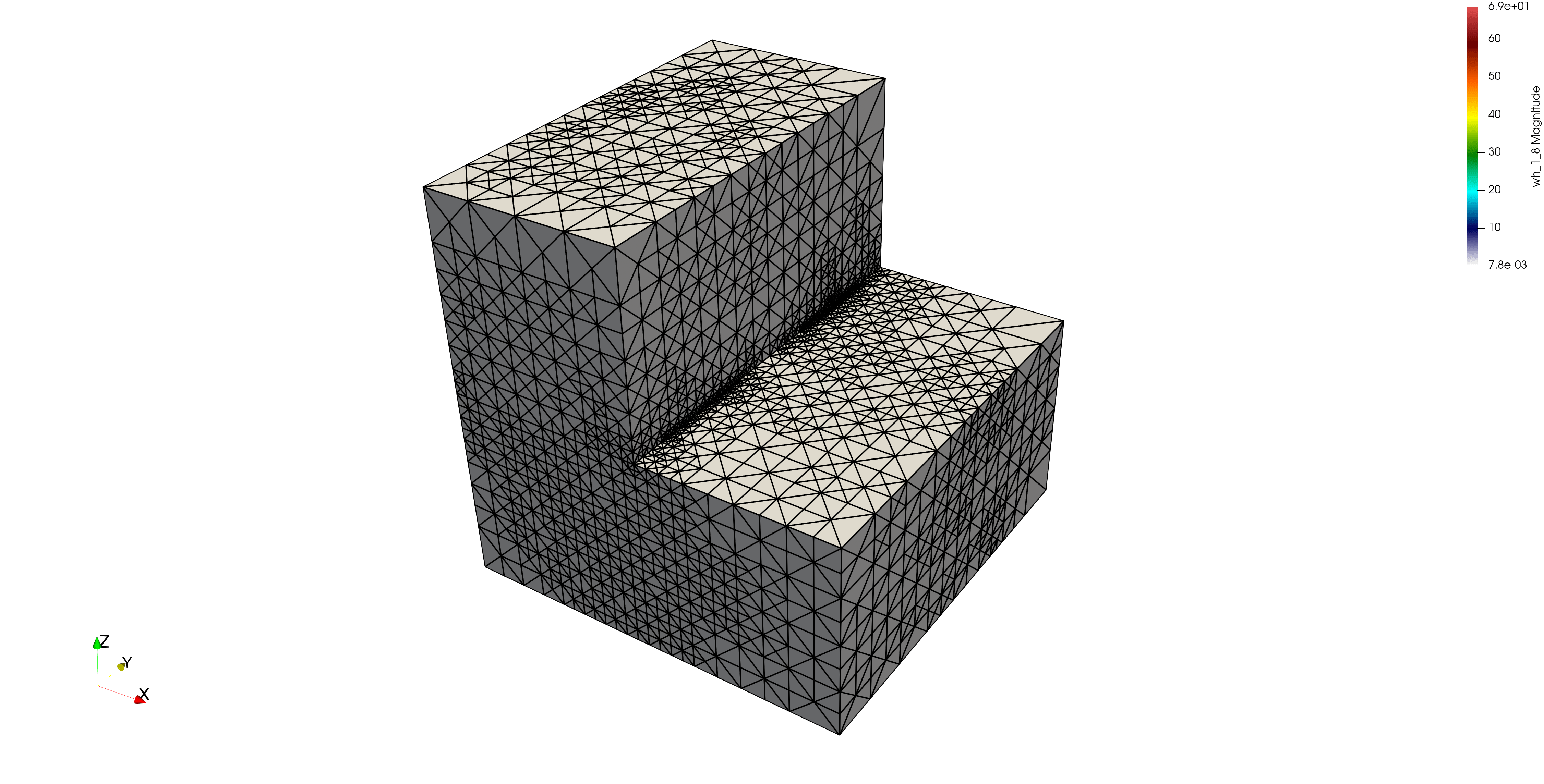}\\
		{\footnotesize $\nu=0.4999$, 9-th iteration}
	\end{minipage}
	\caption{Test 4. Intermediate meshes in the adaptive refinement algorithm for $k=1$, $\nu=0.35$ and $\nu=0.4999$.}
	\label{fig:adaptive-lshape-3D-mallas}
\end{figure}
\begin{figure}[!h]
	\centering
	\begin{minipage}{0.49\linewidth}
		\centering
		\includegraphics[scale=0.06,trim= 40cm 5cm 45cm 3.5cm,, clip]{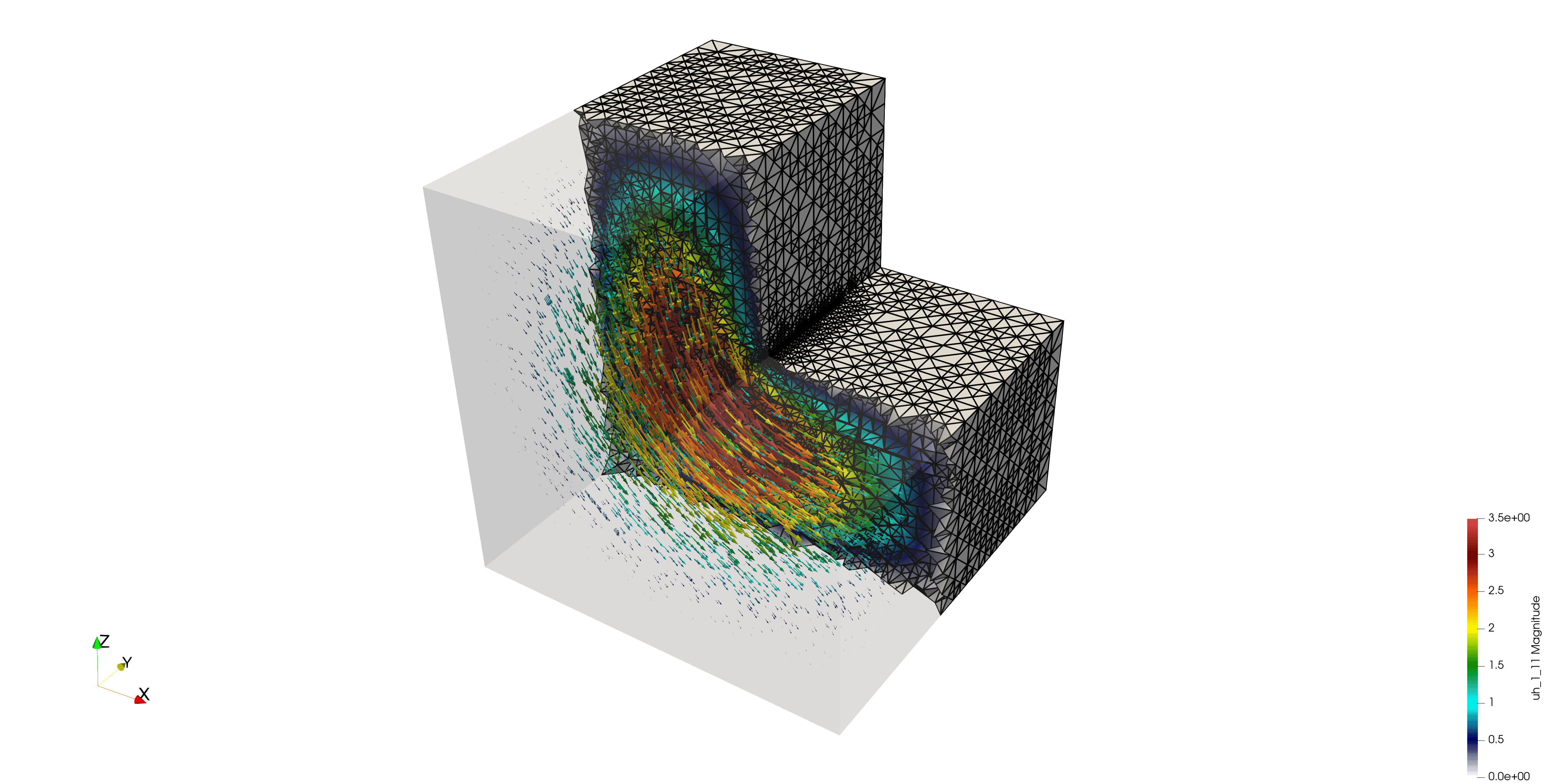}\\
		{\footnotesize $\bu_h,|\bu_h|,\nu=0.35$}
	\end{minipage}
	\begin{minipage}{0.49\linewidth}
		\centering
		\includegraphics[scale=0.06,trim= 40cm 5cm 45cm 3.5cm,, clip]{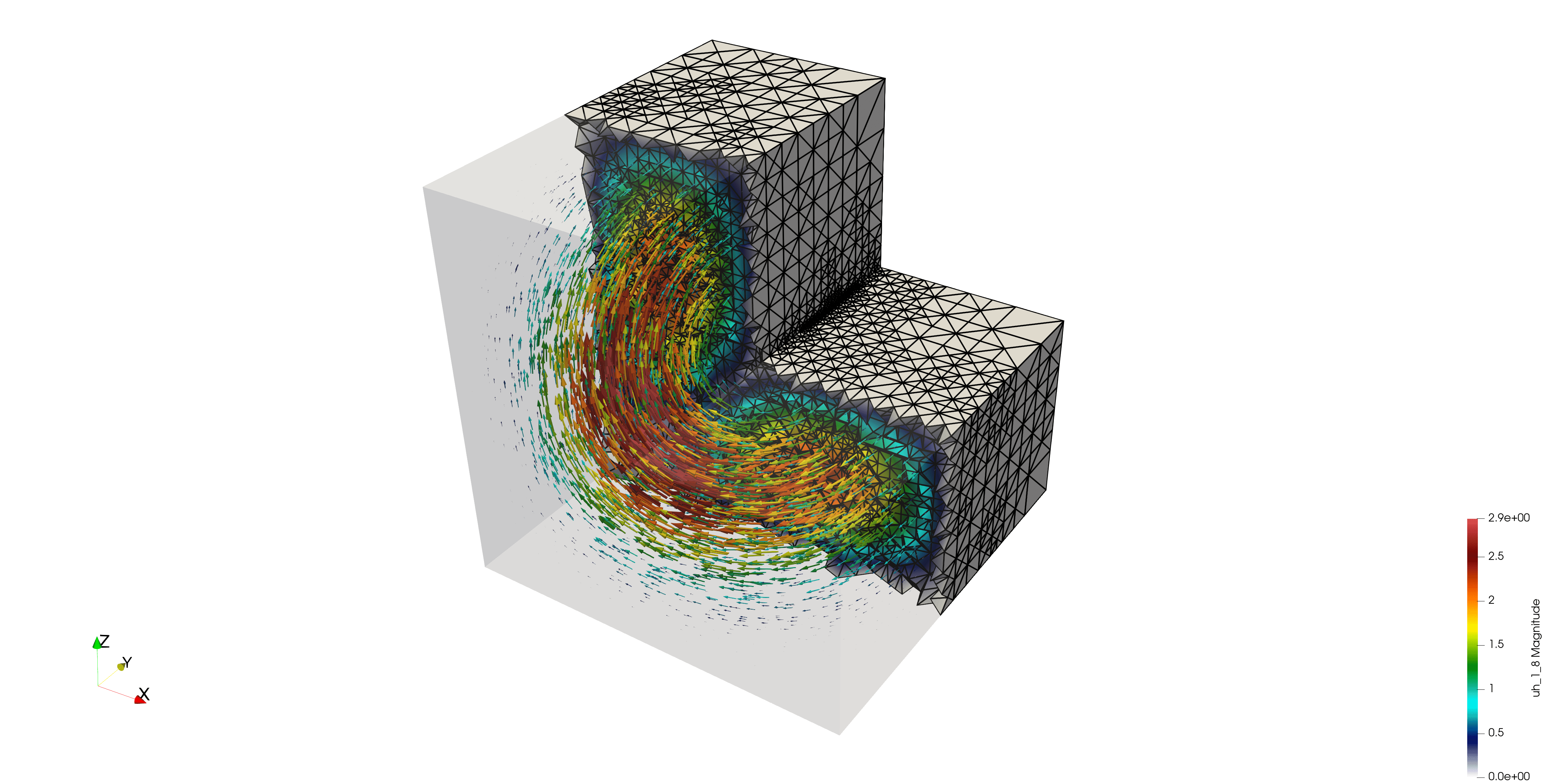}\\
		{\footnotesize $\bu_h,|\bu_h|,\nu=0.4999$}
	\end{minipage}\\
	\begin{minipage}{0.49\linewidth}
		\centering
		\includegraphics[scale=0.06,trim= 40cm 5cm 45cm 3.5cm,, clip]{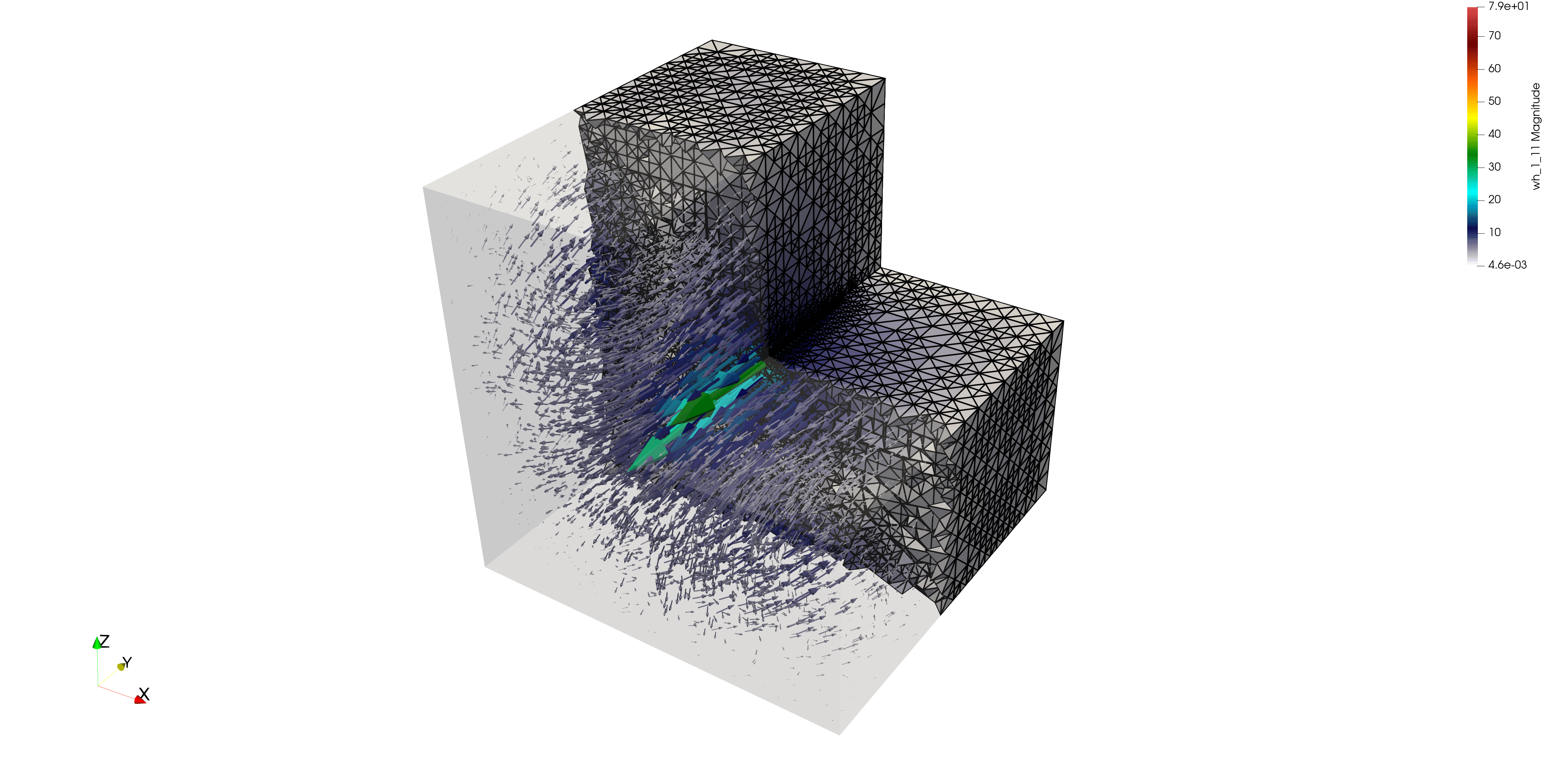}\\
		{\footnotesize $\bw_h,|\bw_h|,\nu=0.35$}
	\end{minipage}
	\begin{minipage}{0.49\linewidth}
		\centering
		\includegraphics[scale=0.06,trim= 40cm 5cm 45cm 3.5cm,, clip]{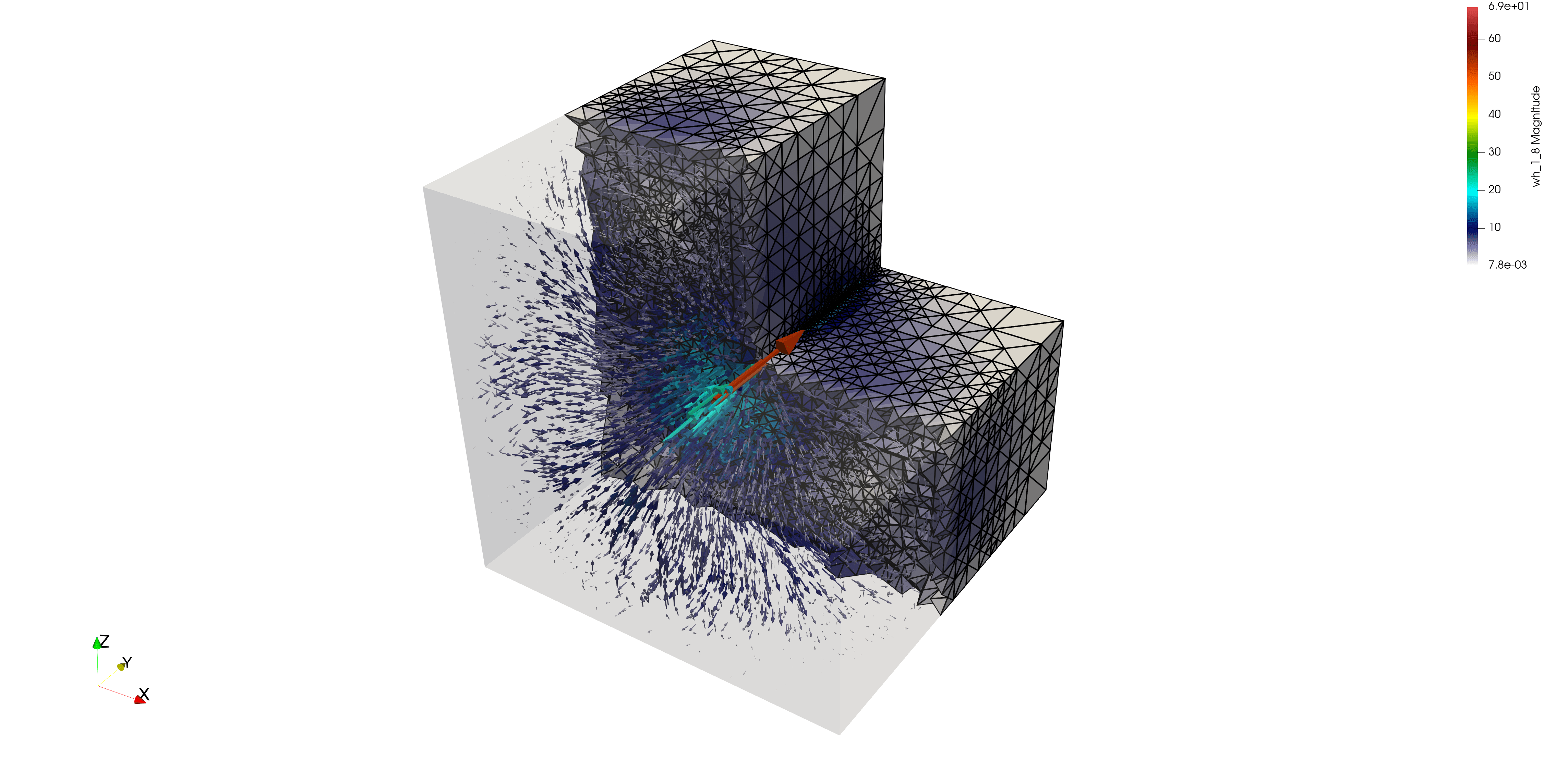}\\
		{\footnotesize $\bw_h,|\bw_h|,\nu=0.4999$}
	\end{minipage}\\
	\begin{minipage}{0.49\linewidth}
		\centering
		\includegraphics[scale=0.06,trim= 40cm 5cm 45cm 3.5cm,, clip]{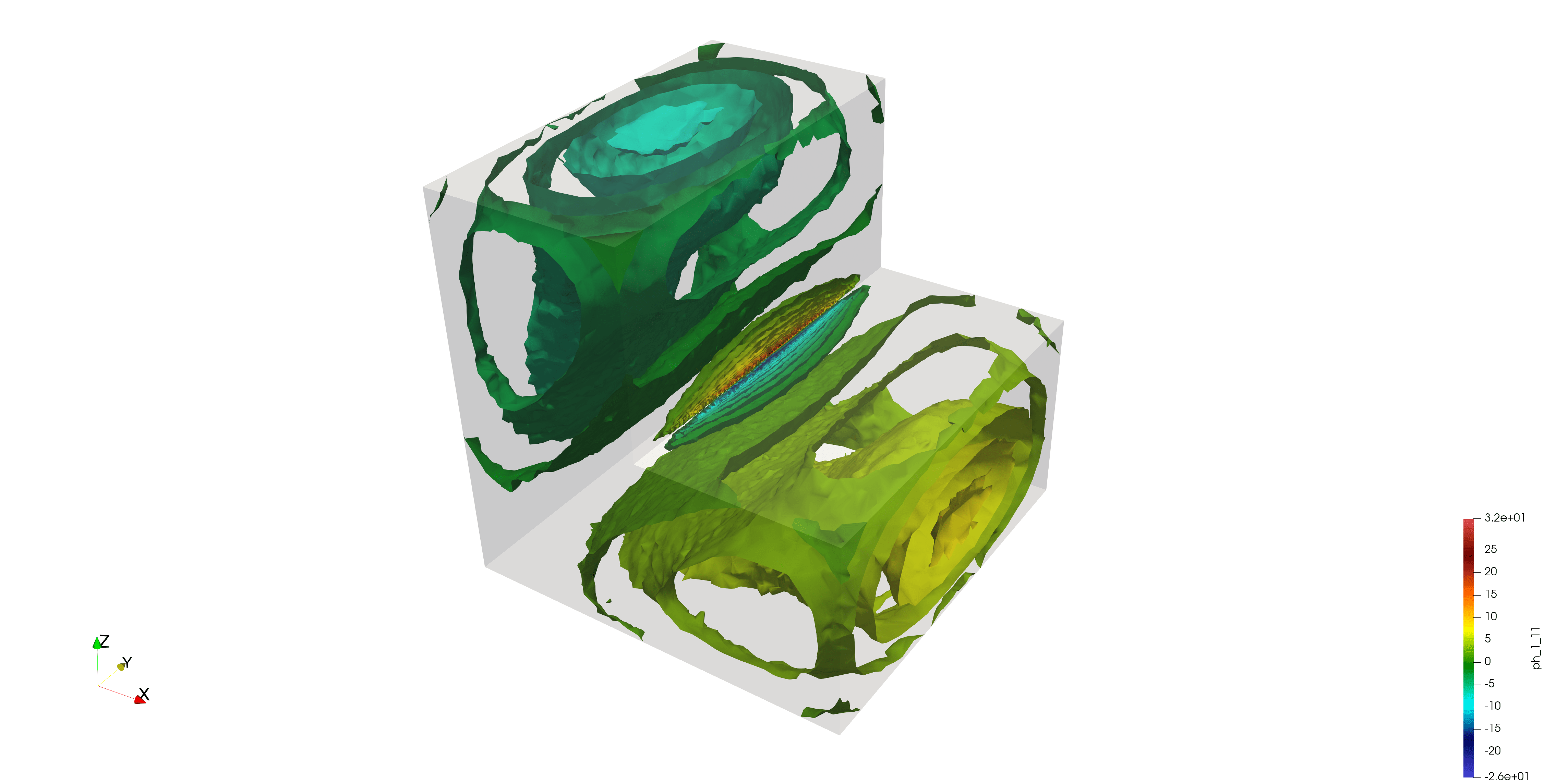}\\
		{\footnotesize $p_h, \nu=0.35$}
	\end{minipage}
	\begin{minipage}{0.49\linewidth}
		\centering
		\includegraphics[scale=0.06,trim= 40cm 5cm 45cm 3.5cm,, clip]{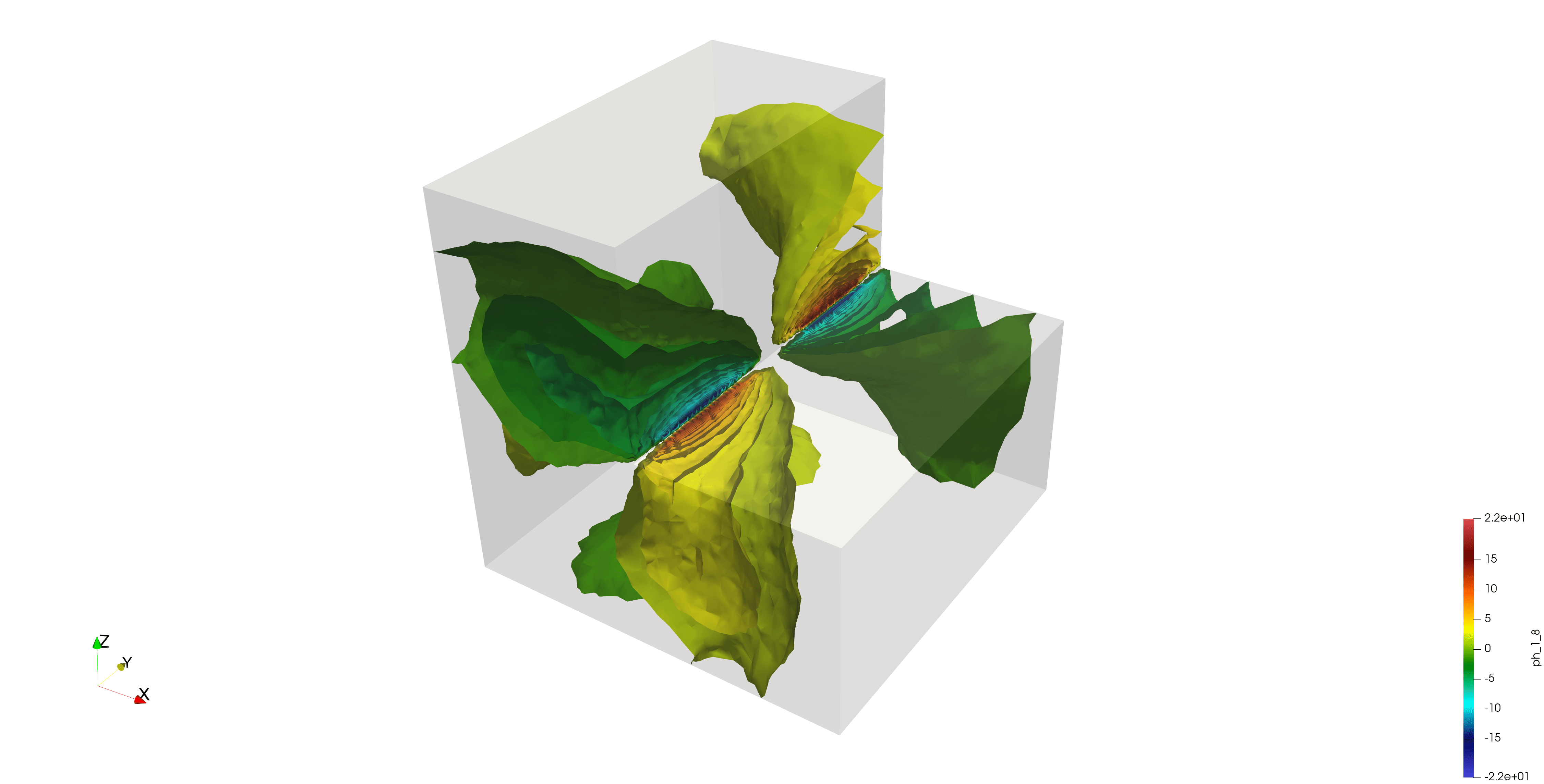}\\
		{\footnotesize $p_h, \nu=0.4999$}
	\end{minipage}
	\caption{Test 4. Lowest computed eigenmodes for $k=1$ and different values of $\nu$ in the last adaptive iteration.}
	\label{fig:adaptive-lshape-3D-eigenmodes}
\end{figure}
\section{Conclusions}
In this paper, we have introduced an alternative formulation to study the elasticity eigenvalue problem, where the main unknowns are the displacement, rotations and pressure. More precisely, we have used this formulation to analyze a mixed finite element method, based in polynomials, to approximate the eigenvalues and associated eigenfunctions. In the a priori analysis, we have proved convergence of the proposed method and optimal order of convergence. The numerical tests, in two and three dimension, for arbitrary polynomial degrees, show two important features: in one hand, that the method is spurious free, and on the other, the stability of the numerical method when we are close to the limit case (namely $\nu=1/2$). Also, we have proposed an a posteriori error estimator for the eigenvalue problem, that is reliable and efficient. The numerical tests that we reported show successfully that the optimal order of convergence is recovered for non smooth eigenfunctions, when the adaptive strategy is performed. Furthermore, the experiments revealed that the estimator performs an adaptive refinement that recovers the optimal order of convergence of the method. Moreover,  the tests show that even for polynomial degrees $k>1$, which are beyond the scope of the presented analysis, the estimator works correctly.
\bibliographystyle{siamplain}
\bibliography{LRV_navier_lame}
\end{document}